\documentclass[letterpaper, reqno,11pt]{article}
\usepackage[margin=1.0in]{geometry}
\usepackage{color,latexsym,amsmath,amssymb, amsthm, graphicx,subfigure,revsymb, enumerate,xcolor, mathtools, enumitem}
\usepackage[percent]{overpic}
\usepackage{calc}
\usepackage{accents}
\newcommand{\dbtilde}[1]{\accentset{\approx}{#1}}

\usepackage{hyperref}

\numberwithin{equation}{section}

\newcommand{\RR}{\mathbb{R}}
\newcommand{\CC}{\mathbb{C}}
\newcommand{\ZZ}{\mathbb{Z}}

\newcommand{\eps}{\varepsilon}
\newcommand{\tubes}{\mathbb{T}}

\newcommand{\dir}{\operatorname{dir}}
\newcommand{\dist}{\operatorname{dist}}
\newcommand{\gtrapproxdelta}{\gtrapprox_\delta}

\newtheorem{thm}{Theorem}[section]
\newtheorem{lem}[thm]{Lemma}
\newtheorem{prop}[thm]{Proposition}
\newtheorem{cor}[thm]{Corollary}
\newtheorem{conj}{Conjecture}[section]

\newtheorem*{mainThmPrime}{Theorem \ref{mainThm}$^\prime$}
\newtheorem*{SWThmRepeat}{Theorem \ref{SWThm}}
\newtheorem*{SWThmRepeatPrime}{Theorem \ref{SWThm}$'$}
\newtheorem*{SWThmInf}{Theorem \ref{SWThm}, informal version}
\newtheorem*{smallL3NormPropInf}{Proposition \ref{smallL3Norm}, informal version}

\newtheorem*{bootstrapLemDiscretized}{Proposition \ref{bootstrapLem}$'$}

\theoremstyle{remark}
\newtheorem{defn}{Definition}[section]

\newtheorem{rem}{Remark}[section]

\newtheorem*{discretizedThinTubesDefn}{Definition \ref{thinTubes}$'$}

\begin{document}
\title{Sticky Kakeya sets and the sticky Kakeya conjecture}
\author{Hong Wang\thanks{Department of Mathematics, New York University, New York, NY; and IHES, Universite Paris-Saclay, Bures sur Yvette, France. hw3639@nyu.edu} \and Joshua Zahl\thanks{Chern Institute of Mathematics and LPMC, Nankai University, Tianjin 300071, China; and The University of British Columbia, Vancouver, BC, Canada. jzahl@nankai.edu.cn}}
\maketitle


\begin{abstract}
A Kakeya set is a compact subset of $\RR^n$ that contains a unit line segment pointing in every direction. The Kakeya conjecture asserts that such sets must have Hausdorff and Minkowski dimension $n$. There is a special class of Kakeya sets, called sticky Kakeya sets. Sticky Kakeya sets exhibit an approximate multi-scale self-similarity, and sets of this type played an important role in Katz, \L{}aba, and Tao's groundbreaking 1999 work on the Kakeya problem. We propose a special case of the Kakeya conjecture, which asserts that sticky Kakeya sets must have Hausdorff and Minkowski dimension $n$. We prove this conjecture in three dimensions. 
\end{abstract}

\section{Introduction}\label{introSection} 
A compact set $K\subset\RR^n$ is called a \emph{Kakeya set} if it contains a unit line segment pointing in every direction. A surprising construction by Besicovitch \cite{bes} shows that such sets can have measure 0. The Kakeya set conjecture asserts that every Kakeya set in $\RR^n$ has Hausdorff and Minkowski dimension $n$. There is also a slightly more technical, single-scale variant of this conjecture, which is called the Kakeya maximal function conjecture: let $\delta>0$ and let $\tubes$ be a set of $1\times\delta$ tubes in $\RR^n$ whose coaxial lines point in $\delta$-separated directions. Then the tubes must be almost disjoint, in the sense that for every $1\leq d\leq n$ and every $\eps>0$, there exists $C=C(n,d,\eps)$ (independent of $\delta$) so that
\begin{equation}\label{KakeyaMaximalFnEstimate}
\Big\Vert \sum_{T\in\tubes}\chi_T \Big\Vert_{\frac{d}{d-1}} \leq C\Big(\frac{1}{\delta}\Big)^{\frac{n}{d}-1+\eps}\Big(\sum_{T\in\tubes}|T|\Big)^{\frac{d-1}{d}}.
\end{equation}
The Kakeya set conjecture and Kakeya maximal conjecture were proved for $n=2$ by Davies \cite{dav} and Cordoba \cite{cordoba}, respectively. The conjectures remain open in three and higher dimensions.

The Kakeya conjecture is closely related to questions in Fourier analysis. This connection was first explored by Fefferman \cite{fef}, who used a variant of Besicovitch's construction to show that the ball multiplier $\widehat{Tf} = \chi_B \hat f$ is unbounded on $L^p(\RR^2)$ when $p\neq 2$. In \cite{bour91}, Bourgain obtained new estimates for Stein's Fourier restriction conjecture in $\RR^3$ by first proving, and then using estimates of the form \eqref{KakeyaMaximalFnEstimate} (for certain $d<3$) in three dimensions. In brief, a function whose Fourier transform is supported on a curved manifold $M$ can be decomposed into a sum of ``wave packets,'' each of which is supported on a tube $T$. The curvature of $M$ ensures that many of these wave packets point in different directions, and estimates of the form \eqref{KakeyaMaximalFnEstimate} can be used to analyze the possible intersection patterns of these wave packets. 

Since Bourgain's seminal work \cite{bour91}, Kakeya estimates have served as an input when studying the Fourier restriction problem \cite{BG, dem, WW}, and methods that were developed in the context of the Kakeya problem have been successfully applied to the Fourier restriction problem, Bochner-Riesz problem, and related questions. Indeed, the modern renaissance in polynomial method techniques was sparked by Zeev Dvir's proof \cite{dvi} of Wolff's finite field Kakeya conjecture. These polynomial method techniques have since revolutionized discrete math, combinatorial geometry, and harmonic analysis. See e.g.~\cite{GutBk, shef} for a modern survey of these developments. 

In this paper we will restrict our attention to the Kakeya set conjecture in three dimensions. In \cite{Wo95}, Wolff proved that every Kakeya set in $\RR^3$ has Hausdorff dimension at least $5/2$. In \cite{KLT}, Katz, \L{}aba, and Tao made the following improvement: every Kakeya set in $\RR^3$ has upper Minkowski dimension at least $5/2+c$, where $c>0$ is a small absolute constant. Katz, \L{}aba, and Tao's argument began by carefully analyzing the structure of a (hypothetical) Kakeya set that had dimension exactly 5/2. At scale $\delta$, such a set $K$ would contain a union of roughly $\delta^{-2}$ $1\times\delta$ tubes that point in $\delta$-separated directions, and the union of these tubes would have volume roughly $\delta^{1/2}$. Next, Katz, \L{}aba, and Tao considered how these $1\times\delta$ tubes would arrange themselves at scale $\rho=\delta^{1/2}$. Since the $1\times\delta$ tubes point in different directions, each $1\times\rho$ tube can contain at most $\rho^2/\delta^2$ $1\times\delta$ tubes. If equality (or near equality) holds (i.e. if the $1\times\delta$ tubes arrange themselves into roughly $\rho^{-2}$ many $1\times\rho$ tubes, each of which contain $\rho^2/\delta^2$ $1\times\delta$ tubes) then we say the arrangement of $1\times\delta$ tubes is \emph{sticky} at scale $\rho$. As a starting point for their arguments in \cite{KLT}, Katz, \L{}aba, and Tao used a result by Wolff \cite{Wo98} to prove that that if a Kakeya set $K$ has upper Minkowski dimension $5/2$, then the corresponding set of $1\times\delta$ tubes must be sticky at every scale $\rho\in (\delta,1)$. 

Motivated by this observation, we introduce a special class of Kakeya sets, which we call sticky Kakeya sets.
Let $\mathcal{L}$ be the set of (affine) lines in $\RR^n$, equipped with the metric $d(l,l') = |p-p'| + \angle(u,u')$. Here $u$ (resp $u'$) is a unit vector parallel to $\ell$ and $p$ (resp. $p'$) is the unique point on $\ell$ with $p\perp u$. Since the direction map from $\mathcal{L}$ to $\RR\mathbb{P}^{n-1}$ is Lipschitz, if $L\subset\mathcal{L}$ contains a line in every direction then $\dim_P(L)\geq n-1$. (Here $\dim_P$ denotes packing dimension; similar statements hold for other metric notions of dimension, but this is less relevant to the current discussion). We say a Kakeya set is sticky if equality holds:

\begin{defn}\label{defnStickyKakeya}
A compact set $K\subset\RR^n$ is called a \emph{sticky Kakeya set} if there is a set of lines $L$ with packing dimension $n-1$ that contains at least one line in each direction, so that $\ell\cap K$ contains a unit interval for each $\ell\in L$. 
\end{defn}

Observe that if we drop the requirement that $L$ has packing dimension $n-1$, then we have recovered the usual definition of a Kakeya set. In particular, every sticky Kakeya set is also a (classical) Kakeya set. As we will see in Section \ref{discretizationSection}, when sticky Kakeya sets (in the sense of Definition \ref{defnStickyKakeya}) are discretized at a small scale $\delta>0$, the corresponding collection of $1\times\delta$ tubes is sticky in the sense of Katz, \L{}aba, and Tao. 
With these definitions, we can introduce the sticky Kakeya set conjecture.

\begin{conj}\label{stickyKakeyaConj}
Every sticky Kakeya set in $\RR^n$ has Hausdorff and Minkowski dimension $n$.
\end{conj} 

A defining feature of sticky Kakeya sets is that if $K$ is a (hypothetical) counter-example to the sticky Kakeya conjecture, then $K$ must be roughly self-similar at many different scales. This was first observed in \cite{KLT} and discussed in greater detail in \cite{TaoBlog}. A precise version of this principle is stated in Proposition \ref{multiScaleStructureExtremal} below.

Since every sticky Kakeya set is also a (classical) Kakeya set, the Kakeya set conjecture implies the sticky Kakeya conjecture, and partial progress toward the Kakeya set conjecture implies the same partial progress for the sticky Kakeya conjecture. In particular, Davies' solution \cite{dav} to the Kakeya set conjecture in the plane implies that Conjecture \ref{stickyKakeyaConj} holds when $n=2$. In this paper, we will prove Conjecture \ref{stickyKakeyaConj} for $n=3$. 

\begin{thm}\label{mainThm}
Every sticky Kakeya set in $\RR^3$ has Hausdorff dimension 3. 
\end{thm}

Our proof of Theorem \ref{mainThm} is inspired by, and partially follows, the arguments recorded in Terence Tao's blog entry \cite{TaoBlog}. That blog entry sketches an approach to solving the Kakeya conjecture in $\RR^3$ that was explored by Nets Katz and Terence Tao in the early 2000s. In brief, \cite{TaoBlog} begins by conjecturing that a (hypothetical) counter-example $K\subset\RR^3$ to the Kakeya conjecture in $\RR^3$ must be sticky, and must also have two additional structural properties, which are called planiness and graininess. We will discuss this conjecture in Section \ref{KakeyaVsStickySection} below. Next, the blog entry describes how these three properties can be used to obtain multi-scale structural information about $K$. Finally, the blog entry suggests how to use this structural information to produce a counter-example to Bourgain's discretized sum-product theorem \cite{bour}, and hence obtain a contradiction.

\subsection{A sketch of the proof}\label{proofSketchSection}
We will briefly outline our approach to proving Theorem \ref{mainThm}, and discuss how it relates to Katz and Tao's strategy, as described in \cite{TaoBlog}. Katz and Tao's strategy was influenced by the following example from \cite{KLT}, which is called the (second) Heisenberg group example:
\begin{equation}\label{defnOfH}
\mathbb{H} = \{(z_1,z_2,z_3)\in\CC^3\colon \operatorname{Im}(z_3) = \operatorname{Im}(z_1\bar z_2)\}.
\end{equation}
The set $\mathbb{H}$ is a counter-example to a strong form of the Kakeya conjecture. In particular, the $\delta$-neighbourhood of $\mathbb{H}\cap B(0,1)$ contains about $\delta^{-4}$ ``complex unit line segments,'' and at most $\delta^{-2}$ of these line segments are contained in the $\delta$-neighbourhood of a (complex) plane in $\CC^3$. This set has (complex analogues of) the structural properties stickiness, planiness, and graininess, but it has volume substantially smaller than 1. At a key step in our argument, we will use (a consequence of) Bourgain's discretized sum-product theorem to distinguish between a sticky Kakeya set in $\mathbb{R}^3$ and the Heisenberg group example. 

Our proof can be divided into three major steps. The first step follows some of the arguments described in \cite{TaoBlog}. For steps 2 and 3 our arguments take a different path from those presented in \cite{TaoBlog}; we do this in part because sticky Kakeya sets can have structural properties that were not discussed in \cite{TaoBlog}.

\medskip

\noindent{\bf Step 1: Discretization and multi-scale self-similarity}\\
In Section \ref{discretizationSection}, we set up a discretization of the sticky Kakeya problem that allows us to exploit multi-scale self-similarity. The technical details of this procedure are new; in particular, we explain how to replace each $1\times\delta$ tube by a reasonably large subset (called a shading), so that key statistics of the Kakeya set are maintained after successive refinements. We hope that this setup will be useful when studying the sticky Kakeya conjecture in higher dimensions. 

In Section \ref{extremalFamiliesSubsection}, we flesh-out the ideas from \cite{TaoBlog} to obtain multi-scale structural information about sticky Kakeya sets. The main conclusion of this section is as follows. Suppose that the sticky Kakeya conjecture in $\RR^n$ was false, i.e. $\inf \dim_H K = n-\sigma_n$, where the infimum is taken over all sticky Kakeya sets in $\RR^n$, and $\sigma_n>0$. Let $K$ be a sticky Kakeya set with dimension $n-\sigma_n+\eps$ for some small $\eps>0$ (we will call this an $\eps$-extremal sticky Kakeya set), and let $E$ be the discretization of $K$ at a small scale $\delta>0$. Then $E$ contains a union of $1\times\delta$-tubes, which are sticky in the sense of Katz, \L{}aba, and Tao. Furthermore, this collection of $\delta$-tubes is coarsely self-similar: for each intermediate scale $\rho\in(\delta,1)$, the $\rho$-thickening of $K$ again resembles an $\eps$-extremal sticky Kakeya set, and the union of $\delta$-tubes inside each $\rho$-tube resembles a (anisotropically re-scaled) $\eps$-extremal sticky Kakeya set. The precise statement is given by Proposition \ref{multiScaleStructureExtremal}.

In Section \ref{planinessAndGraininessSection} we restrict attention to the case $n=3$ and again follow arguments proposed by Katz and Tao in \cite{TaoBlog}. We show that extremal sticky Kakeya sets must have a structural property called planiness. Specifically, with $K$ and $E$ as above, if $\tubes$ is the set of $1\times\delta$ tubes contained in $E$, then there is a function $V \colon E \to S^2$ so that for each point $x\in E$, the tubes in $\tubes$ containing $x$ must make small angle with the subspace $V(x)^\perp$. $V(x)$ is called a plane map, and we use the multi-scale self-similarity of extremal Kakeya sets to show that $V$ must be Lipschitz. An important consequence is that inside each ball of radius $\delta^{1/2}$, the set $E$ can be decomposed into a disjoint union of parallel rectangular prisms of dimensions roughly $\delta^{1/2}\times\delta^{1/2}\times\delta$; these rectangular prisms are called grains. After an anisotropic re-scaling, we obtain a new collection of $1\times\delta$ tubes whose union $E'\subset\RR^3$ satisfies the following {\bf Property (P):} For each $z_0\in [0,1]$, the slice $E' \cap \{z=z_0\}$ is a union of parallel $1\times\delta$ rectangles whose spacing forms an Ahlfors-regular set of dimension $1-\sigma_3$; the direction of these rectangles is determined by a ``slope function'' $f(z_0)$, which is Lipschitz. All of the arguments described thus far also apply to the Heisenberg group example.

\medskip

\noindent{\bf Step 2: Regularity of the slope function}\\
In Section \ref{globalGrainsIsC2Sec}, we show that the slope function $f(z)$ described above is $C^2$ (or more precisely, it agrees at scale $\delta$ with a function that has controlled $C^2$ norm). To do this, it suffices to show that at every scale $\rho\in (\delta,1)$ and for every interval $I$ of length $\rho^{1/2}$, the graph of $f$ above $\rho$ can be contained in the $\rho$-neighborhood of a line. We will briefly describe how this is done. Fix a scale $\rho$ and an interval $I$, and let $F\subset[0,1]^2$ be a re-scaled copy of the graph of $f$ above $I$. We will also construct a set $G\subset[0,1]^2$, which is the graph of a Lipschitz function $g\colon [0,1]\to\RR$ that encodes a local analogue of the plane map $V$. We construct the set $G$ so that certain closed paths inside $E'$ are encoded by arithmetic operations on the sets $F$ and $G$. Specifically, whenever we select points $p\in F$ and $q,q'\in G$, we have that $p\cdot(q-q')$ specifies the location of one of the $1\times\delta$ rectangles in the set $E' \cap \{z=z_0\}$ that was described above in Property (P). Since the spacing of these rectangles form an Ahlfors-regular set, we conclude that 
\begin{equation}\label{FCdotGGContained}
F\cdot (G-G)\quad\textrm{is contained in an Ahlfors regular set of dimension}\ 1-\sigma_3.
\end{equation} 

Recall that $F$ and $G$ are the graphs of Lipschitz functions, and thus satisfy a non-concentration condition analogous to having Hausdorff dimension 1. Is it possible for \eqref{FCdotGGContained} to occur? If $F$ and $G$ are contained in orthogonal lines, then $F\cdot(G-G)$ is a point, so the answer is yes. The next result says that this is the only way that \eqref{FCdotGGContained} can occur.

\begin{SWThmInf}
Let $F$ and $G$ be one-dimensional sets that have been discretized at a small scale $\rho>0$. Then either (A): (most of) $F$ and $G$ are contained in the $\rho$-neighborhoods of orthogonal lines, or (B): there exists some scale $\rho \leq t <\!\! < 1$ so that the $t$-neighborhood of $F\cdot(G-G)$ contains an interval.
\end{SWThmInf}

The precise version of Theorem \ref{SWThm} is stated in Section \ref{globalGrainsIsC2Sec} and proved in Section \ref{SWProofSec}. Our proof of Theorem \ref{SWThm} uses recent ideas from projection theory that were developed in the context of the Falconer distance problem \cite{Orp, SW0, OSW}. Specifically, the projection theorem we use generalizes Bourgain's discretized sum-product estimates; see \cite{OrponenShmerkin, RenWang} for further recent developments in this area. It is at this step where we distinguish between a sticky Kakeya set (which is a subset of $\RR^3$) and the Heisenberg group example (which is a subset of $\CC^3$). In brief, if $G$ is not contained in a line, then by an analogue of Beck's theorem due to Orponen, Shmerkin, and the first author \cite{OSW}, we might expect $G$ to span a two-dimensional set of lines, and hence a one-dimensional set of directions, i.e. there is a one-dimensional set $\Theta\subset S^1$ so that for each $\theta\in\Theta$, there are points $q,q'\in G$ with $\theta = \frac{q-q'}{|q-q'|}$. If this happens, then Kaufman's projection theorem says that there exists a direction $\theta\in \Theta$ for which $F\cdot \theta$ (and hence $F\cdot (q-q')$ ) has dimension 1. The actual proof of Theorem \ref{SWThm} must overcome several difficulties when executing the above strategy. First, Theorem \ref{SWThm} has weaker hypotheses than the ones stated above; the set $G\cdot(F-F)$ is replaced by a smaller set, where the differences and dot products are taken along a sparse (but not too sparse!) subset of $F\times G\times G$. Second, the proof sketch above supposes a dichotomy, where either $G$ is contained in a line, or it spans a one-dimensional set of directions. In the discretized setting, however, this dichotomy is less apparent because $G$ can behave differently at different scales. 

Returning to the proof of Theorem \ref{mainThm},  if Item (B) from Theorem \ref{SWThm} holds, then by an observation of Dyatlov and the second author \cite[Proposition 6.13]{DZ}, $F\cdot(G-G)$ cannot be contained in an Ahlfors-regular set of dimension strictly smaller than 1. Thus Theorem \ref{SWThm} implies that the graph of $f$ above $I$ must be contained in the $\rho$ neighborhood of a line. Since this holds for every interval $I$ at every scale $\rho$, we conclude that $f$ is $C^2$. Kakeya sets with regularity conditions have been studied in the past \cite{FG, KR}, but different ideas are needed in this setting. 

\medskip

\noindent{\bf Step 3: Twisted projections}\\
In Section \ref{largeSlopeSec} we exploit the fact that lines in a Kakeya set point in different directions to show that the slope function $f$ has (moderately) large derivative, i.e. $|f'|$ is bounded away from 0. 

In Section \ref{C2ProjThmSec}, we begin by observing the following consequence of Property (P) from Step 1 above. Since each slice $E'\cap \{z=z_0\}$ is a union of parallel $1\times\delta$ rectangles whose spacing forms an Ahlfors regular set and whose slope is given by $f(z_0)$, if we define $\pi_f(x,y,z) = \big( x + yf(z),\ z\big)$, then $\pi_f(E')\subset [0,1]^2$ has measure roughly $\delta^{\sigma_3}$. This is because each slice $\pi_f(E') \cap \{z =z_0\}$ is a union of $\delta$-intervals that are arranged like a $1-\sigma_3$ dimensional Ahlfors regular set. 

If $T$ is a $1\times\delta$ tube, then $\pi_f(T)$ is the $\delta$-neighborhood of a $C^2$ curve. Since the tubes in $\tubes$ point in different directions and since $|f'|$ is bounded away from 0, the corresponding curves $\pi_f(T)$ satisfy a separation property related to Sogge's cinematic curvature condition \cite{KoW,sog}. The next result says that the union of these curves must be large
\begin{smallL3NormPropInf}
Let $\tubes$ be a set of $1\times\delta$ tubes pointing in $\delta$-separated directions, with $\#\tubes\gtrsim \delta^{-2}$. Let $f\colon [0,1]\to\RR$ with $\Vert f\Vert_{C^2}\lesssim 1$ and $|f'|\sim 1$. Then for all $\eps>0$, there exists $c_\eps>0$ so that $\bigcup_{T\in\tubes}\pi_f(T)$ has measure at least $c_\eps \delta^\eps$.
\end{smallL3NormPropInf}
The precise version of Proposition \ref{smallL3Norm} is a maximal function estimate analogous to \eqref{KakeyaMaximalFnEstimate}. It is proved using a variant of Wolff's $L^{3}$ bound for the Wolff circular maximal function \cite{Wo}, where circles are replaced by a family of $C^2$ curves that satisfy Sogge's cinematic curvature condition. This latter result was recently proved by Pramanik, Yang, and the second author \cite{PYZ}. Since $E'$ contains about $\delta^{-2}$ tubes that point in $\delta$-separated directions, Proposition \ref{smallL3Norm} says that $\pi_f(E')$ must have measure at least $c_\eps\delta^\eps$, for each $\eps>0$. On the other hand, Property (P) says that $\pi_f(E')$ has measure at most $\delta^{\sigma_3}$. We conclude that $\sigma_3=0$, which finishes the proof. 


\subsection{The sticky Kakeya conjecture versus the Kakeya conjecture}\label{KakeyaVsStickySection}
In this section, we will informally discuss what Theorem \ref{mainThm} tells us about the Kakeya conjecture. In particular, how close does Theorem \ref{mainThm} get us to proving the Kakeya conjecture in $\RR^3$? Katz and Tao's arguments from  \cite{TaoBlog} begin by conjecturing that a (hypothetical) counter-example $K\subset\RR^3$ to the Kakeya conjecture in $\RR^3$ must have the structural properties stickiness, planiness, and graininess. The results from \cite{BCT} and \cite{Gut} show that such a counter-example must be plainy and grainy, but it is unclear whether $K$ must be sticky. To reformulate the remaining part of Katz and Tao's conjecture in our language, we need the following definition.
\begin{defn}
A Kakeya set $K\subset\RR^n$ is $\eps$-close to being a sticky if there is a set of lines $L$ with packing dimension at most $n-1+\eps$ that contains at least one line in each direction, so that $\ell\cap K$ contains a unit interval for each $\ell\in L$. 
\end{defn}
In particular, a sticky Kakeya set is $\eps$-close to being sticky for every $\eps>0$. The following version of Katz and Tao's conjecture says that Kakeya sets that are nearly extremal are close to being sticky. 
\begin{conj}\label{extremalIsStickyConj}
Suppose the Kakeya conjecture in $\RR^n$ is false, i.e. $\inf \dim_H K = d<n$, where the infimum is taken over all Kakeya sets in $\RR^n$. Then for all $\eps>0$, there exists $\eta>0$ so that the following holds. Let $K\subset\RR^n$ be a Kakeya set with $\dim_H(K)\leq d+\eta$. Then $K$ is $\eps$-close to being sticky.
\end{conj}

If Conjecture \ref{extremalIsStickyConj} is true for $n=3$, then the following generalization of Theorem \ref{mainThm} would imply the Kakeya conjecture in $\RR^3$.
\begin{mainThmPrime}
For all $\eps>0$, there exists $\eta>0$ so that the following holds. Let $K\subset\RR^3$ be a Kakeya set that is $\eta$-close to being sticky. Then $\dim_H K \geq 3-\eps$.
\end{mainThmPrime}
\begin{rem}
A consequence of Theorem \ref{mainThm}$^\prime$ is that if an X-ray estimate holds in $\RR^3$ at dimension $d<3$, then there exists $c>0$ so that every Kakeya set in $\RR^3$ must have upper Minkowski dimension at least $d+c$.
\end{rem}

One piece of evidence in favor of Conjecture \ref{extremalIsStickyConj} at the time it was formulated in \cite{TaoBlog} was that the only known example of a Kakeya-like object with Hausdorff dimension less than 3 was the (second) Heisenberg group example \eqref{defnOfH}. As noted above, $\mathbb{H}$ is a counter-example to a strong form of the Kakeya conjecture, and it has (complex analogues of) the structural properties stickiness, planiness, and graininess. 

More recently, however, Katz and the second author \cite{KZ} discovered a second Kakeya-like object that has small volume at certain scales---this is analogous to having Hausdorff dimension less than 3. This object is called the $SL_2$ example, and it is not sticky. The $SL_2$ example is not a counter-example to Conjecture \ref{extremalIsStickyConj} (nor to the Kakeya conjecture) because it is not a Kakeya set in $\RR^3$. Nonetheless, the existence of the $SL_2$ example suggests that in order to resolve the Kakeya conjecture in $\RR^3$, it will be necessary to study Kakeya sets that are far from sticky. 

As a starting point for this latter program, it seems reasonable to show that that the $SL_2$ example cannot be realized in $\RR^3$. To make this statement precise, define $\mathcal L_{SL_2}$ to be the set of lines in $\RR^3$ that can either be written in the form $(a,b,0) + \RR(c,d,1)$ with $ad-bc=1$, or $(0,0,0)+\RR(c,d,0)$. Equivalently, $\mathcal L_{SL_2}$ is the set of horizontal lines in the first Heisenberg group. 

We conjecture that lines from this set cannot be used to construct a counter-example to the Kakeya conjecture.

\begin{conj}\label{sl2Conj}
Let $K\subset\RR^3$ be compact, and suppose there is a set of lines $L\subset \mathcal L_{SL_2}$ that contains at least one line in each direction, so that $\ell\cap K$ contains a unit interval for each $\ell\in L$. Then $K$ has Hausdorff and Minkowski dimension 3.
\end{conj}

\noindent \emph{Added November 3, 2022:} F\"assler and Orponen \cite{FO} have recently proved Conjecture \ref{sl2Conj}.

\medskip

We will not discuss Conjectures \ref{extremalIsStickyConj} or \ref{sl2Conj} further. The remainder of this paper will be devoted to proving Theorem \ref{mainThm} and \ref{mainThm}$^\prime$.

\subsection{Thanks}
The authors would like to thank Larry Guth, Nets Katz, Pablo Shmerkin, and Terence Tao  for helpful conversations and suggestions during the preparation of this manuscript. The authors would like to thank Mukul Rai Choudhuri, Keith Rogers, and the anonymous referees for corrections to an earlier version of this manuscript. Hong Wang was supported by NSF Grant DMS-2055544. Joshua Zahl was supported by a NSERC Discovery grant and by the Nankai Zhide Foundation.


\section{Discretization}\label{discretizationSection}
\subsection{Lines, tubes and shadings}
We begin by recalling some standard definitions and terminology associated with the Kakeya problem. Throughout this section, we will fix an integer $n\geq 2$; all implicit constants may depend on $n$. Slightly abusing the notation we introduced in the introduction, we will define $\mathcal{L}_n$ to be the set of lines in $\RR^n$ of the form $(\underline{p},0)+\RR v$, where $\underline{p}\in [-\frac{1}{n},\frac{1}{n}]^{n-1}$ and $v\in S^{n-1}$ has final coordinate $v_n\geq 1/2$. We define $d(\ell,\ell')=|\underline{p}-\underline{p}'| + \angle(v, v')$. This definition of distance is comparable to the definition from Section \ref{introSection}. We define the measure $\lambda_n$ on $\mathcal{L}_n$ to be the product measure of $(n-1)$-dimensional Euclidean measure and normalized surface measure on $S^{n-1}$ (the latter we denote by $\nu_{n-1}$). With these definitions of distance and measure, the packing dimension of a set $L\subset \mathcal{L}_n$ agrees with its upper modified box dimension.


For $\ell\in\mathcal{L}_n$ and $\delta>0$, we define the $\delta$-tube $T$ with coaxial line $\ell$ to be $N_{2n\delta}(\ell)\cap [-1,1]^n$, where $N_t(X)$ denotes the $t$-neighborhood of $X$. This definition is slightly nonstandard, since we use the $2n\delta$-neighborhood rather than the $\delta$-neighborhood of $\ell$; we do this so that a cube of side-length $\delta$ intersecting $\ell$ will be contained in the associated tube $T$. Observe that any tube of this type has volume comparable to $\delta^{n-1}$. If $T$ is a $\delta$-tube and $C>0$, we will use $CT$ to denote the $C\delta$-tube with the same coaxial line. We say two $\delta$-tubes $T,T'$ are \emph{essentially identical} if their respective coaxial lines satisfy $d(\ell, \ell')\leq \delta$; otherwise $T$ and $T'$ are \emph{essentially distinct}. We say that $T$ and $T'$ are \emph{essentially parallel} if their respective coaxial lines satisfy $|\dir(\ell)-\dir(\ell')|\leq\delta$, where $\dir(\ell)$ is the unit vector $v$ described above.

A $\delta$-\emph{cube} is a set of the form $[0,\delta)^n+p$, where $p\in(\delta\ZZ)^n.$ A \emph{shading} of a $\delta$-tube $T$ is a set $Y(T)\subset T$ that is a union of $\delta$-cubes. We will write $(\tubes,Y)_\delta$ to denote a set of $\delta$ tubes and their associated shadings $\{Y(T)\colon T\in\tubes\}$. If $(\tubes,Y)_\delta$ is a set of tubes and their shadings, we write $E_{\tubes}$ to denote the set $\bigcup_{T\in\tubes}Y(T)$; the shading $Y$ will always be apparent from context. 

We say that $(\tubes',Y')_\delta$ is a \emph{sub-collection} of $(\tubes,Y)_\delta$ if $\tubes'\subset\tubes$ and $Y'(T')\subset Y(T')$ for each $T'\in\tubes'$. We say that $(\tubes',Y')_\delta$ is a \emph{refinement} of $(\tubes,Y)_\delta$ if in addition, $\sum_{T'\in\tubes'}|Y'(T)|\geq (\log(1/\delta))^{-C}\sum_{T\in\tubes}|Y(T)|$, where $C\geq 0$ is a constant that depends only on $n$. We will denote this by $\sum_{T'\in\tubes'}|Y'(T)|\gtrapproxdelta\sum_{T\in\tubes}|Y(T)|.$ In practice, we will only study collections of tubes $(\tubes,Y)_\delta$ with $\#\tubes\leq\delta^{-100n}$. By dyadic pigeonholing, every collection of tubes with this property has a refinement $(\tubes',Y')_\delta$ with $\sum_{T\in\tubes'}\chi_{Y'(T)}=\mu\chi_{E_{\tubes'}}$ for some integer $\mu\geq 1$; this is called a \emph{constant multiplicity} refinement of $(\tubes,Y)_\delta$.

Let $T$ be a $\delta$-tube and let $\tilde T$ be a $\rho$-tube, with $\delta\leq\rho$. We say that $\tilde T$ \emph{covers} $T$ if their respective coaxial lines satisfy $d(\ell, \tilde \ell)\leq \rho/2$. If $\tubes$ is a collection of $\delta$-tubes and $\tilde\tubes$ is a collection of $\rho$-tubes, we say $\tilde\tubes$ covers $\tubes$ if every tube in $\tubes$ is covered by some tube in $\tilde\tubes$. If this is the case, we write $\tubes[\tilde T]$ to denote the set of tubes in $\tubes$ that are covered by $\tilde T$. A similar definition holds for collections of tubes and their associated shadings: we say that $(\tilde\tubes,\tilde Y)_\rho$ \emph{covers} $(\tubes,Y)_\delta$ if $\tilde\tubes$ covers $\tubes$, and for each $\tilde T\in\tilde\tubes$ and each $T\in\tubes[\tilde T]$ we have $Y(T)\subset\tilde Y(\tilde T)$. Such a cover is called \emph{balanced} if $|E_{\tubes}\cap Q|$ is the same for each $\rho$-cube $Q\subset E_{\tilde\tubes}$


\subsection{Defining $\sigma_n$}
In this section we will define a quantity, called $\sigma_n$, that will allow us to formulate a discretized version of the sticky Kakeya conjecture.

\begin{defn}\label{defnOfM}
For $s,t,\delta\in(0,1)$, define

\begin{equation}\label{defnM}
M(s,t,\delta) = \inf_{\tubes,Y}\Big|\bigcup_{T\in\tubes}Y(T)\Big|.
\end{equation}
Here, the infimum is taken over all pairs $(\tubes,Y)_\delta$ with the following properties.
\begin{enumerate}[label=(\alph*)]
\item\label{tubesDistinct} The tubes in $\tubes$ are essentially distinct. 

\item\label{DefnOfMGoodCoveringItem} For each $\delta\leq \rho\leq 1$, $\tubes$ can be covered by a set of $\rho$ tubes, at most $\delta^{-t}$ of which are essentially parallel to a common $\rho$ tube.

\item\label{DefnOfMBigVolItem} $\sum_{T\in\tubes}|Y(T)|\geq\delta^s$
\end{enumerate}
\end{defn}

Next, let
\begin{equation}\label{defnNEpsS}
N(s,t)=\limsup_{\delta\to 0^+}\frac{\log M(s, t,\delta)}{\log\delta}.
\end{equation}
$N(s,t)$ is defined for all $(s,t)\in (0,1)^2$. Observe that for $\delta\in (0,1)$, if  $s\leq s'$ and $t\leq t'$, then $M(s, t,\delta) \geq M(s', t',\delta)$, and hence $N(s,t)\leq N(s',t')$. In particular, since $0\leq N(s,t)\leq n$, we have that
\[
\lim_{(s,t)\to (0,0)}N(s,t)=\inf_{(s,t)\in (0,1)^2}N(s,t).
\]
Denote this common value by $\sigma_n$. Our definition of $\sigma_n$ was chosen so that two properties hold. First, $\sigma_n$ allows us to bound the Hausdorff dimension of sticky Kakeya sets. This will be described in Section \ref{sigmaNStickyKakeyaSec}. Second, collections $(\tubes,Y)_\delta$ that nearly extremize the infimum in \eqref{defnM} will have useful structural properties at many scales. This will be described in Section \ref{extremalFamiliesSubsection}.


\subsection{$\sigma_n$ and the dimension of sticky Kakeya sets}\label{sigmaNStickyKakeyaSec}
In this section we will relate $\sigma_n$ to the dimension of sticky Kakeya sets. Our goal is to prove the following.
\begin{prop}\label{sigmaNVsDimensionProp}
For all $\eps>0$, there exists $\eta>0$ so that the following holds. Let $K\subset\RR^n$ be a Kakeya set that is $\eta$-close to being sticky. Then $\dim_H K \geq n - \sigma_n - \eps$. In particular, if $K$ is a sticky Kakeya set then $\dim_H K \geq n - \sigma_n$. 
\end{prop}

Our first task is to analyze families of lines in $\RR^n$ that point in many directions and have packing dimension close to $n-1$. For each $\rho>0$ and $\eps>0$, define the class of ``Quantitatively Sticky'' sets $\operatorname{QStick}(n, t, \rho)$ to be the collection of sets $L\subset\mathcal{L}_n$ that satisfy
 \[
\Big|\big\{\underline{p}\in \RR^{n-1} \colon (\underline{p},v)\in N_{\delta}(L)\big\}\Big| \leq \delta^{n-1-t}\quad\textrm{for all}\ v\in S^{n-1},\ \delta\in (0,\rho),
\]
where in the above expression, $N_{\delta}(L)$ denotes the $\delta$-neighborhood of $L$ in the metric space $\mathcal{L}_n$. 

For example, if $L = \{ (\underline p, v) \in \mathcal{L}_n\colon \underline p = f(v)\}$ for some Lipschitz function $f$, then for each $t>0$ we have $L\in \operatorname{QStick}(n, t, \rho)$ for all sufficiently small $\rho$. The next lemma explains how $\operatorname{QStick}(n, t, \rho)$ is related to sticky Kakeya sets. 
\begin{lem}\label{packingLinesVsFrakL}
Let $L\subset \mathcal{L}_n$ with $\nu_{n-1}(\dir(L))>0$. Then for all $t > \dim_P(L)-(n-1)$, there exists $L'\subset L$ and $\rho>0$ with $\nu_{n-1}(\dir(L'))>0$ and $L'\in \operatorname{QStick}(n, t, \rho)$. 
\end{lem}
\begin{proof}
Let $d = \dim_P L$, and let $\eps = t - d + (n-1)$. Since $d = \dim_P(L)=\overline{\dim}_{MB}(L)$, we can write $L$ as a countable union of sets $\{L_i\}$, with $\overline{\dim}_M L_i \leq d + \eps/4$ for each $i$. Since $\nu_{n-1}\big(\bigcup \dir(L_i)\big)=\nu_{n-1}(\dir(L))>0$, there exists an index $i$ so that $\nu_{n-1}(\dir L_i)>0$. Re-indexing if necessary, we can suppose that $\nu_{n-1}(\dir L_1)>0$. Since $\overline{\dim}_M L_1 \leq d + \eps/4$, we can select $\rho_1>0$ sufficiently small so that
\begin{equation}\label{uniformCoveringNumberBd}
\lambda_n\big(N_{2\delta}( L_1 )\big)\leq \delta^{(2n-2) - d-\eps/2}\quad \textrm{for all}\ \delta\in (0, \rho_1).
\end{equation}
We say a direction $v\in S^{n-1}$ is \emph{over-represented} at scale $\delta$ if
\[
\Big|\big\{\underline{p}\in \RR^{n-1} \colon (\underline{p},v)\in N_{2\delta}(L_1)\big\}\Big|> \delta^{n-1-t}=\delta^{(2n-2)-\delta-\eps}.
\]
By \eqref{uniformCoveringNumberBd} and Fubini (recall that $\lambda_n$ is a product measure on $\RR^{n-1}\times S^{n-1}$), for each $\delta\in (0, \rho_1)$, we have
\begin{equation}
\nu_{n-1}\big(\{ v\in S^{n-1}\colon v\ \textrm{is over-represented at scale}\ \delta \}\big) \lesssim  \delta^{\eps/2}.
\end{equation}
Thus for each integer $k_1$ with $2^{-k_1}\leq \rho_1$, we have
\begin{equation}\label{dyadicContributions}
\sum_{k\geq k_1} \nu_{n-1}\big(\{v \colon v\ \textrm{is over-represented at scale}\ 2^{-k} \}\big)\lesssim k_1^{-\eps/2}.
\end{equation}
Selecting $k_1$ appropriately, we can ensure that the LHS of \eqref{dyadicContributions} is at most $\nu_{n-1}(\dir(L_1))/2$. Note that if $v\in S^{n-1}$ is not over-represented at any dyadic scale $2^{-k}\leq 2^{-k_1}$, then 
\[
\Big|\big\{\underline{p}\in \RR^{n-1} \colon (\underline{p},v)\in N_{\delta}(L_1)\big\}\Big| \leq \delta^{n-1-t}\quad\textrm{for all}\ \ \delta\in (0,2^{-k_1}).
\]
To conclude the proof, let $\rho=2^{-k_1}$ and let
\[
L' = L_1\ \backslash\ \bigcup_{k=k_1}^\infty\operatorname{dir}^{-1}\big(\{v \colon v\ \textrm{is over-represented at scale}\ 2^{-k} \}\big).\qedhere
\]
\end{proof}

With this lemma, Proposition \ref{sigmaNVsDimensionProp} now follows from standard discretization arguments. The details are as follows.
\begin{proof}[Proof of Proposition \ref{sigmaNVsDimensionProp}]
Let $\eps>0$ be given. By the definition of $\sigma_n$, there exists $s,t\in (0,1)$ so that 
\begin{equation}\label{ChoiceOfNEpsS}
N(s,t) \leq \sigma_n+\eps/2.
\end{equation}
We will show that Proposition \ref{sigmaNVsDimensionProp} holds with $\eta = t/4$. 

Let $K$ be a Kakeya set that is $\eta$-close to being sticky. Since $K$ is compact, after a translation and dilation we can suppose that $K\subset [-\frac{1}{n},\frac{1}{n}]^n$. In particular, there exists a number $w>0$ and a set $L\subset\mathcal{L}_n$ with $\dim_P(L)\leq n-1 + \eta$ so that for each unit vector $v\in S^{n-1}$ with final coordinate $v_n\geq 1/2,$ there is a line $\ell\in L$ with $\dir(\ell) = v$ and $|\ell \cap K|\geq w$. Use Lemma \ref{packingLinesVsFrakL} to select a subset $L'\subset L$ and a number $\rho_1>0$ with $L'\in \operatorname{QStick}(n, t/2, \rho_1)$ and $\nu_{n-1}(\dir(L'))>0$. 

Let $k_1$ be an integer to specified below, with $2^{-k_1}< \rho_1^{2n/t}$. By the definition of Hausdorff dimension, we can cover $K$ by a union of cubes $K \subset \bigcup_{k=k_1}^\infty\bigcup_{Q\in\mathcal{Q}_k}Q$, where $\mathcal{Q}_k$ is a set of $2^{-k}$-cubes with 


\begin{equation}\label{numberOfBalls}
\# \mathcal{Q}_k\lesssim (2^{k})^{\dim_H K + \eps/2}.
\end{equation} 
For each $\ell\in L'$, we have
\[
K\cap\ell\ \subset\ \bigcup_{k= k_1}^\infty\Big(\ell\cap \bigcup_{Q\in\mathcal{Q}_k}Q\Big).
\]
Since $|K\cap\ell|\geq w$, there exists at least one index $k$ so that 
\begin{equation}\label{goodK}
\Big|\ell\cap \bigcup_{Q\in\mathcal{Q}_k}Q\Big|\geq \frac{w}{100 k^2}.
\end{equation}
For each $k\geq k_1$, let $L_k$ be the set of lines $\ell\in L'$ for which \eqref{goodK} holds for that choice of $k$. Then $L' = \bigcup_k L_k$, so there exists an index $k_2\geq k_1$ with $\nu_{n-1}(\dir(L_{k_2}))\geq \frac{\nu_{n-1}(\dir(L'))}{100k_2^2}$.

Let $\delta = 2^{-k_2}$, so in particular $\frac{1}{k_2}\geq \frac{\log 2}{\log(1/\delta)}$. Let $\Omega$ be a $\delta$-separated subset of $\dir(L_{k_2})$ of cardinality $\gtrsim \frac{\nu_{n-1}(\dir L')\delta^{1-n}}{\log(1/\delta)^2}$. For each direction $v\in \Omega$, let $T$ be a $\delta$-tube with coaxial line $\ell \in L_{k_2}$ pointing in direction $v$. By \eqref{goodK}, we can choose $T$ so that the set 
\[
Y(T) = \bigcup_{\substack{Q\in\mathcal{Q}_{k_2} \\ Q \cap \ell \neq\emptyset}}Q
\]
satisfies $Y(T)\subset T$ and $|Y(T)|\gtrsim  w\log(1/\delta)^{-2}\delta^{n-1}$. 

We claim that if $k_1$ is chosen sufficiently large, then $(\tubes,Y)_\delta$ satisfies Items \ref{tubesDistinct}, \ref{DefnOfMGoodCoveringItem}, and \ref{DefnOfMBigVolItem} from Definition \ref{defnOfM}. Item \ref{tubesDistinct} is immediate, since  $\Omega$ is $\delta$-separated. Item \ref{DefnOfMBigVolItem} follows from the inequality
\[
\sum_{T\in\tubes}|Y(T)|\gtrsim \frac{(\# \Omega) w\delta^{n-1}}{\log(1/\delta)^2} \gtrsim \frac{w\nu_{n-1}(\dir L')}{\log(1/\delta)^4}.
\]
This quantity is larger than $\delta^s$ provided we select $k_1$ sufficiently small large, depending on $\nu_{n-1}(\dir L'),$ $w$, and $n$. 

\medskip

Our final task is to show that $(\tubes,Y)_\delta$ satisfies Item \ref{DefnOfMGoodCoveringItem}. Let $\rho\in [\delta,1]$. Our goal is to find a set of $\rho$-tubes $\tilde \tubes$ that covers $\tubes$, so that at most $\delta^{-t}$ of them are pairwise parallel. We will divide our analysis into cases.

\medskip

\noindent {\bf Case 1: $\rho\in [\delta,\rho_1).$} let $\tilde\tubes$ be a collection of $\rho$-tubes, so that the set of coaxial lines $\tilde L$ has the following properties. 
\begin{enumerate}[label=(\roman*)]
\item Each $\tilde\ell\in\tilde L$ is contained in $N_\rho(L')$. 
\item The $\rho/2$-neighborhoods of the lines in $\tilde L$ cover $L'$. 
\item The $\rho/6$ neighborhoods of the lines in $\tilde L$ are disjoint. 
\end{enumerate}
Such a collection of tubes can be chosen greedily, as in the proof of the Vitali covering lemma. 

Each tube $T\in\tubes$ has a coaxial line $\ell\in L_1\subset L'$. Hence there is a tube $\tilde T\in\tilde\tubes$ so that $d(\ell,\tilde \ell)\leq\rho/2,$ and thus $\tilde T$ covers $T$. We conclude that $\tilde \tubes$ covers $\tubes$.

It remains to show that at most $\rho^{-t}$ tubes from $\tilde\tubes$ can all be essentially parallel to a common $\rho$ tube. Let $\tilde T_1,\ldots\tilde T_M\in\tilde\tubes$ be essentially parallel to a common $\rho$ tube $\tilde T_0$, and let $v = \dir(\tilde T_0)$. In particular, $|v - \dir(\tilde T_j)|\leq \rho$ for each index $j$. By Item (i) above, this means $\tilde T_j\in N_\rho(L') \cap \dir^{-1}(N_{\rho}(v))$. By item (iii) above, we conclude that
\[
\lambda_n\big(N_\rho(L') \cap \dir^{-1}(N_{\rho}(v))\big) \geq \sum_{i=1}^M \lambda_n\big( N_{\rho/6} (\tilde \ell_i)\big)\gtrsim M\rho^{2n-2}.
\]
Thus by Fubini, there exists $v'\in N_{\rho}(v)$ with 
\[
|N_\rho(L') \cap \dir^{-1}(v')| \gtrsim M\rho^{n-1}. 
\]
But since $L'\in \operatorname{QStick}(n, t/2, \rho_1)$ and $\rho<\rho_1$, if $\rho_1>0$ is selected sufficiently small (depending on the implicit constants above, which in turn depend on $n$) then $M\leq\rho^{-t}\leq \delta^{-t}$.

\medskip

{\bf \noindent Case 2: $\rho\in [\rho_1, 1]$.} Let $\tilde\tubes$ be a set of $\rho$-tubes, whose corresponding coaxial lines form a $\rho/2$-net in $\mathcal{L}_n$. Then $\tilde\tubes$ clearly covers $\tubes$. The maximum number of tubes in $\tilde\tubes$ that are all essentially parallel to a common $\rho$ tube is trivially bounded by $\#\tilde\tubes \sim \rho^{2-2n}$, and this is $\leq\delta^{-t}$ (provided $\rho_1$, and hence $\delta$ is selected sufficiently small compared to $n$), since $\delta\leq 2^{-k_1}\leq \rho_1^{2n/t}$.

\medskip

We have shown that $(\tubes,Y)_\delta$ satisfies Items \ref{tubesDistinct}, \ref{DefnOfMGoodCoveringItem}, and \ref{DefnOfMBigVolItem} from Definition \ref{defnOfM}. Thus we have that
\[
M(s, t, \delta) \leq \Big|\bigcup_{T\in\tubes}Y(T)\Big| \leq  \Big|\bigcup_{Q\in\mathcal{Q}_{k_2}} Q\Big|\leq (\# \mathcal{Q}_{k_2}) \delta^n \leq \delta^{n - \dim K - \eps/2},
\]
where the final inequality used \eqref{numberOfBalls}. We conclude that
\begin{equation}\label{witnessM}
\frac{\log M(s, t,\delta)}{\log\delta} \geq n - \dim_H K - \eps/2.
\end{equation}
We have just shown that for every $k_1$ sufficiently large, there exists $\delta\in (0,2^{-k_1}]$ for which \eqref{witnessM} holds. In particular, 
\begin{equation}\label{lowerBdNEpsS}
N( s,t) \geq n - \dim_H K - \eps/2.
\end{equation}
To conclude the proof, we compare \eqref{ChoiceOfNEpsS} and \eqref{lowerBdNEpsS}. 
\end{proof}

Proposition \ref{sigmaNVsDimensionProp} says that Theorems \ref{mainThm} and \ref{mainThm}$^\prime$ will follow from the estimate $\sigma_3=0$. 


\section{Extremal families of tubes and multi-scale structure}\label{extremalFamiliesSubsection}
In this section, we will study families of tubes that nearly extremize the quantity $M$ from Definition \ref{defnOfM}.
\begin{defn}\label{defnExtremal}
Let $\eps,\delta>0$. We say a pair $(\tubes,Y)_\delta$ is $\eps$-extremal if it satisfies Items \ref{tubesDistinct}, \ref{DefnOfMGoodCoveringItem}, and \ref{DefnOfMBigVolItem} from Definition \ref{defnOfM} with $s=t=\eps$, and furthermore $\Big|\bigcup_{T\in\tubes}Y(T)\Big|\leq \delta^{\sigma_n - \eps}$. 
\end{defn}

From the definition of $\sigma_n$, we immediately conclude that extremal collections of tubes always exist for small $\delta$. More precisely, we have the following.

\begin{lem}\label{existenceOfExtremalFamilies}
Let $\eps,\delta_0>0$. Then there exists a $\eps$-extremal collection of tubes $(\tubes,Y)_\delta$ for some $\delta\in (0, \delta_0]$. 
\end{lem}

The main result of this section says that extremal collections of tubes must be coarsely self-similar at every scale. To state this precisely, we need the following definition.
\begin{defn}\label{unitRescalingDefn}
Let $(\tubes,Y)_\delta$ be a collection of $\delta$ tubes and let $\tilde T$ be a $\rho$ tube that covers each tube in $\tubes$. The \emph{unit rescaling of $(\tubes,Y)_\delta$ relative to $\tilde T$} is the pair $(\hat\tubes, \hat Y)_{\delta/\rho}$ defined as follows. The coaxial lines of the tubes $\hat T\in\hat\tubes$ are the images of the coaxial lines of the tubes $T\in\tubes$ under the map $\phi\colon\RR^n\to\RR^n$, defined as follows. If $\tilde\ell$ is the coaxial line of $\tilde T$, then $\phi$ is the composition of a rigid transformation that sends the point $\tilde\ell\cap\{x_n=0\}$ to the origin and sends the line $\tilde\ell$ to the $x_n$-axis, with the anisotropic dilation
\begin{equation}\label{scaling}
 (x_1,\ldots,x_n)\mapsto \big( cx_1/\rho,\ cx_2/\rho,\ldots,cx_{n-1}/\rho,\ cx_n\big).
\end{equation}
Each shading $\hat Y(\hat T)$ is the union of $\delta/\rho$ cubes that intersect $\phi(Y(T))$. The constant $c=c(n)\sim 1$ is chosen so that the coaxial lines of the tubes in $\hat T$ are in $\mathcal{L}_n$, and $\hat Y(\hat T)\subset\hat T$. See Figure \ref{figUnitRescaling}.
\end{defn}
\begin{figure}[h]
\centering
\includegraphics[scale=0.3]{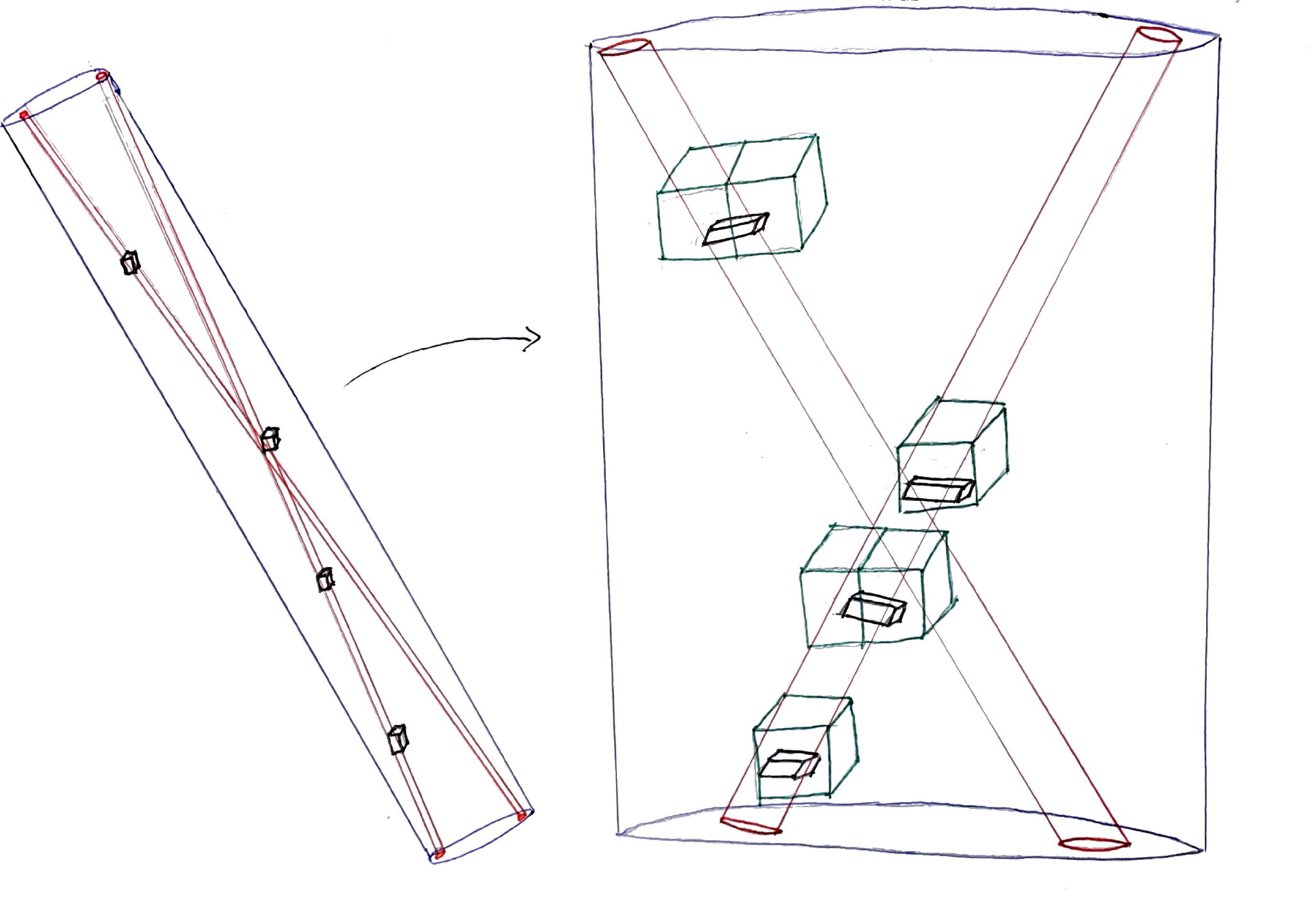}
\caption{The unit rescaling of $(\tubes,Y)_\delta$. For clarity, only a few cubes in $Y(T)$ have been drawn.}
\label{figUnitRescaling}
\end{figure}

We can now state the main result of this section. It says that extremal collections of tubes look self-similar at all scales.
\begin{prop}\label{multiScaleStructureExtremal}
For all $\eps>0$, there exists $\eta>0$ and $\delta_0>0$ so that the following holds for all $\delta\in (0, \delta_0]$. Let $(\tubes,Y)_\delta$ be a $\eta$-extremal collection of tubes, and let $\rho\in [\delta^{1-\eps}, \delta^{\eps}]$. Then there is a refinement $(\tubes',Y')_\delta$ of $(\tubes,Y)_\delta$ and a balanced cover $(\tilde\tubes,\tilde Y)_\rho$ of $(\tubes',Y')_\delta$ with the following properties.
\begin{enumerate}[label=(\roman*)]
\item\label{muCoarseBdItem} $(\tilde\tubes,\tilde Y)_\rho$ is $\eps$-extremal.

\item\label{fineExtremalItem} For each $\tilde T\in\tilde\tubes$, the unit rescaling of  $(\tubes'[\tilde T],Y')_\delta$ relative to $\tilde T$ is $\eps$-extremal.
\item\label{coarseScaleEstimateItem} For each $p\in\RR^n$, $\#\tilde\tubes(p)\leq \rho^{-\sigma_n-\eps}$
\item\label{fineEstimateItem} For each $p\in\RR^n$ and each $\tilde T\in\tilde\tubes$, $\#\tubes'[\tilde T](p)\leq (\delta/\rho)^{-\sigma_n-\eps}.$
\end{enumerate}

\end{prop}

Before proving Proposition \ref{multiScaleStructureExtremal}, we will analyze the unit rescaling of extremal collections of tubes that are covered by a single $\rho$-tube. A precise formulation is given below. In the arguments that follow, we will make frequent use of the fact that if $(\tubes,Y)_\delta$ is $\eps$-extremal, then $\delta^{1-n+\eps}\lesssim \#\tubes\lesssim \delta^{1-n-\eps},$ and if $(\tubes',Y')_\delta$ is a sub-collection with $\sum_{T\in\tubes'}|Y'(T)|\geq\delta^{\eps'}$ for some $\eps'\geq\eps$, then $(\tubes',Y')$ is $\eps'$-extremal.

\begin{lem}\label{rescaledCoveredTube}
For all $\eps>0$, there exists $\eta,\delta_0>0$ so that the following holds for all $\delta\in (0,\delta_0].$ Let $(\tubes,Y)_\delta$ be a set of tubes that are covered by a $\rho$-tube $\tilde T$, with $\rho\in [\delta^{1-\eps},1]$. 

Suppose that $(\tubes,Y)_\delta$ satisfies Items \ref{tubesDistinct} and \ref{DefnOfMGoodCoveringItem} from Definition \ref{defnOfM} with $\eta$ in place of $s$, and that$(\tubes,Y)_\delta$ satisfies the following analogue of Item \ref{DefnOfMBigVolItem}: $\sum_{T\in\tubes}|Y(T)|\geq\delta^\eta\rho^{n-1}$. 

Then there is a refinement $(\tubes',Y')_\delta$ so that the unit rescaling of $(\tubes',Y')_\delta$ relative to $\tilde T$ satisfies Items \ref{tubesDistinct}, \ref{DefnOfMGoodCoveringItem}, and \ref{DefnOfMBigVolItem} from Definition \ref{defnOfM} with $\eps$ in place of $s$ and $t$. Furthermore, 
\begin{equation}\label{pointwiseBdRescaledTubes}
\#\tubes'(p)\leq(\rho/\delta)^{\sigma+\eps}\quad\textrm{for all}\ p\in\RR^n.
\end{equation}
\end{lem}
\begin{proof}
To begin, observe that the transformation $\phi$ from Definition \ref{unitRescalingDefn} has the following properties. 
\begin{itemize}

\item If $\ell_1,\ell_2$ are lines with $d(\tilde\ell, \ell_i)\leq \rho$ for $i=1,2$. Then
\begin{equation}\label{phiDistortion}
\begin{split}
\rho^{-1}d(\ell_1, \ell_2)\lesssim d(\phi(\ell_1),&\ \phi(\ell_2))\leq \rho^{-1}d(\ell_1, \ell_2),\\
\rho^{-1}\angle\big(\dir(\ell_1),\ \dir(\ell_2)\big)\lesssim  \angle\big(\dir(\phi(\ell_1)),\ &\dir(\phi(\ell_2))\big)\leq \rho^{-1}\angle\big(\dir(\ell_1),\ \dir(\ell_2)\big).
\end{split}
\end{equation}

\item If $\tau\in[\delta,\rho]$ and $T^\dag$ is a $\tau$-tube covered by $2\tilde T$, then $\phi(T^\dag)$ is contained in the $\tau/\rho$ tube with coaxial line $\phi(\ell^\dag)$.

\item For every measurable $X\subset\RR^n$ we have
\begin{equation}\label{PhiDistortsVolume}
|\phi(X)| \sim  \rho^{1-n}|X|.
\end{equation} 

\end{itemize}
By \eqref{phiDistortion}, there is a refinement of $(\tubes,Y)_\delta$ whose unit rescaling relative to  $\tilde T$ consists of essentially distinct tubes. Let $(\tubes',Y')_\delta$ be a refinement with this property, which is also constant multiplicity, i.e.~there is a number $\mu$ so that $\#\tubes'(p)=\mu$ for all $p\in E_{\tubes'}$. Let $(\hat\tubes, \hat Y)_{\delta/\rho}$ be the unit rescaling of $(\tubes',Y')_\delta$ relative to $\tilde T$. In particular, $(\hat\tubes, \hat Y)_{\delta/\rho}$ satisfies Property \ref{tubesDistinct} from Definition \ref{defnOfM}.

By \eqref{PhiDistortsVolume}, we have
\begin{equation}\label{boundDilatedTubes}
\Big|\bigcup_{T\in\tubes}Y(T)\Big|\sim \rho^{n-1}\Big|\bigcup_{\hat T \in\hat \tubes}\hat Y(\hat T)\Big|.
\end{equation}

Our next task is to obtain a lower bound on the RHS of \eqref{boundDilatedTubes}. Let $s,t>0$ be chosen so that 
\begin{equation}\label{defnSTInProof}
N(s,t) \leq \sigma_n + \eps/4.
\end{equation} 
Decreasing $s$ and $t$ if necessary, we may suppose that $s,t\leq\eps/2$. We will show that $(\hat\tubes, \hat Y)_{\delta/\rho}$ satisfies Properties \ref{DefnOfMGoodCoveringItem} and \ref{DefnOfMBigVolItem} from Definition \ref{defnOfM} for this value of $s$ and $t$, and all sufficiently small $\delta$. 

Property \ref{DefnOfMBigVolItem} is straightforward: by \eqref{PhiDistortsVolume} we have 
\begin{equation}\label{volumeStarredTubes}
\sum_{\hat T\in\hat\tubes}|\hat Y(\hat T)|\gtrsim \rho^{1-n}\sum_{T\in\tubes}|Y(T)|\geq \delta^\eta \geq (\rho/\delta)^{\eta/\eps}.
\end{equation}
If $\eta,\delta_0>0$ are selected sufficiently small, then the LHS of \eqref{volumeStarredTubes} is at most $(\rho/\delta)^s$. 

For property \ref{DefnOfMGoodCoveringItem}, let $\tau\in (\delta/\rho, 1)$. Let $\tubes^*$ be a collection of $\tau\rho$-tubes that cover $\tubes$, at most $\delta^{-\eta}$ of which are essentially parallel to a common $\tau\rho$-tube (such a collection exists by hypothesis, since $\delta\leq \tau\rho\leq1$). We can suppose that each $T^*\in\tubes^*$ covers at least one tube from $\tubes$, and thus by the triangle inequality, each $T^*$ is covered by $2\tilde T$. Define $\hat{\tubes}^*$ to be the set of $\tau$ tubes whose coaxial lines are given by $\{ \phi(\ell^*)\colon \ell^*\in\tubes^*\}$. By \eqref{phiDistortion}, if $T\in\tubes$ is covered by $T^*\in\tubes^*$, then the corresponding tube $\hat T$ is covered by $\hat{T}^*$; in particular, $\hat{\tubes}^*$ covers $\hat\tubes$. By \eqref{phiDistortion}, if $\delta_0$ is chosen sufficiently small then at most $(\delta/\rho)^{-2\eta /\eps}$ tubes from $\hat{\tubes}^*$ are essentially parallel to a common tube. Thus if $\eta>0$ is chosen sufficiently small, then $(\hat \tubes,\hat Y)_{\delta/\rho}$ satisfies property \ref{DefnOfMGoodCoveringItem} for the value of $t$ specified in \eqref{defnSTInProof}.

We have shown that $(\hat \tubes, \hat Y)_{\delta/\rho}$ satisfies Properties \ref{tubesDistinct}, \ref{DefnOfMGoodCoveringItem}, and \ref{DefnOfMBigVolItem} from Definition \ref{defnOfM} with $\eps$ in place of $s$ and $t$.

All that remains is to establish \eqref{pointwiseBdRescaledTubes}, i.e. to bound $\mu\leq (\rho/\delta)^{\sigma_n+\eps}$. We have also shown that $(\hat \tubes, \hat Y)_{\delta/\rho}$ satisfies Properties \ref{tubesDistinct}, \ref{DefnOfMGoodCoveringItem}, and \ref{DefnOfMBigVolItem} from Definition \ref{defnOfM} for the values of $s$ and $t$ specified by \eqref{defnSTInProof}. Thus if $\delta/\rho$ is sufficiently small, then 

\begin{equation}\label{lowerBdYStar}
\Big|\bigcup_{\hat T\in\hat \tubes}\hat Y(\hat T)\Big|\geq (\delta/\rho)^{\sigma_n + \eps/2}. 
\end{equation}
We can ensure that $\delta/\rho$ is sufficiently small for \eqref{lowerBdYStar} to hold by selecting $\delta_0>0$ appropriately, since the hypothesis $\rho\geq\delta^{1-\eps}$ forces $\delta/\rho\searrow 0$ as $\delta\searrow 0$. 

On the other hand, since $(\hat \tubes, \hat Y)_{\delta/\rho}$ satisfies Property \ref{DefnOfMGoodCoveringItem} from Definition \ref{defnOfM} with $s\leq\eps/2$, we have 
\[
\mu= \Big(\sum_{\hat T\in\hat\tubes}|\hat Y(\hat T)|\Big)\Big(\Big|\bigcup_{\hat T \in\hat \tubes}\hat Y(\hat T)\Big|\Big)^{-1}\lesssim (\delta/\rho)^{\eps/2}\Big(\Big|\bigcup_{\hat T \in\hat \tubes}\hat Y(\hat T)\Big|\Big)^{-1}\leq (\delta/\rho)^{\sigma_n + \eps}.
\qedhere
\]

\end{proof}

\begin{proof}[Proof of Proposition \ref{multiScaleStructureExtremal}]
$\phantom{1}$\\
\noindent{\bf Step 1: Constructing $\tilde \tubes$}.\\
To begin, we will construct a set of $\rho$-tubes that covers a refinement of $\tubes$, with the property that each $T\in\tubes$ is contained in exactly one $\rho$-tube, and $T$ is close to the coaxial line of this $\rho$-tube. Here are the details. Let $\tilde\tubes_0$ be a set of $\rho/(2 n)$ tubes that covers $\tubes$, so that the coaxial line of each $\tilde T\in\tilde\tubes_0$ coincides with the coaxial line of a tube from $\tubes$, and at most $O(\delta^{-\eta})$ tubes from $\tilde\tubes_0$ are essentially pairwise parallel to a common tube. $\tilde\tubes_0$ can be constructed by first selecting a cover of $\tubes$ as described in Item \ref{DefnOfMGoodCoveringItem} of Definition \ref{defnM}, and then replacing each tube $\tilde T$ in this cover by $O(1)$ tubes $\tilde T'$ with $d(\tilde T, \tilde T')\lesssim\rho$.

Next, let $\tilde\tubes_1\subset\tilde\tubes_0$ be chosen so that the tubes in $2  n\tilde\tubes_1$ are essentially distinct; each $T\in\tubes$ is covered by at most one tube in $2 n \tilde\tubes_1$; and 
\[
\sum_{\tilde T\in\tilde\tubes_1}\sum_{T\in\tubes[\tilde T]}|Y(T)|\sim \sum_{T\in\tubes}|Y(T)|.
\]
Define $\tilde\tubes_2 =2  n\tilde\tubes_1$. Then $\tilde\tubes_2$ is a set of essentially distinct $\rho$-tubes, and hence they satisfies Property \ref{tubesDistinct} from Definition \ref{defnOfM}. Our final set $\tilde\tubes$ will be a subset of $\tilde\tubes_2$, and thus it will continue to satisfy Property \ref{tubesDistinct} from Definition \ref{defnOfM}.

Next, we will show that $\tilde\tubes_2$ satisfies Property \ref{DefnOfMGoodCoveringItem} from Definition \ref{defnOfM} with $t = \eta/\eps$. To do this, let $\tau\in[\rho, 1].$ Then there is a set $\tubes^\dag$ of $\tau$ tubes that covers $\tubes$, at most $\delta^{-\eta}\leq \rho^{-\eta/\eps}$ of which are essentially parallel to a common tube. But since each $\tilde T\in\tilde\tubes_2$ shares a coaxial line with some $T\in\tubes$, and $\tau\geq\rho$, $\tubes^\dag$ also covers $\tilde\tubes_2.$ Our final set $\tilde\tubes$ will be a subset of $\tilde\tubes_2$, and thus it will continue to satisfy Property \ref{DefnOfMGoodCoveringItem} from Definition \ref{defnOfM} with this same value of $t$.

Define $\tubes_2\subset\tubes$ to be the set of tubes that are covered by some tube from $\tilde\tubes_1$, and let $Y_2(T) = Y(T)$ for each $T\in\tubes_2$. Then $(\tubes_2,Y_2)$ is a refinement of $(\tubes,Y)$. Furthermore, each $T\in\tubes_2$ is covered by exactly one tube from $\tilde\tubes_2$; and if $\tilde T$ covers $T$, then $T$ is close to the coaxial line of $\tilde T$, in the sense that $\frac{1}{2 n}\tilde T$ also covers $T$.

\medskip
\noindent{\bf Step 2: Fine scale structure}.
In this step we will analyze the structure of the sets $\tubes_2[\tilde T]$. By dyadic pigeonholing, we can select a set $\tilde\tubes_3\subset \tilde\tubes_2$ and a refinement $(\tubes_3,Y_3)_\delta$ of $(\tubes_2,Y_2)_\delta$ so that for each $\tilde T\in\tilde\tubes_3$, we have
\begin{equation}\label{avgWeight}
\sum_{T\in\tubes_3[\tilde T]}|Y_3(T)| =(\# \tilde\tubes_3)^{-1} \sum_{T\in\tubes_3}|Y_3(T)|.
\end{equation}
Since $O(\delta^{-\eta})$ tubes from $\tilde\tubes_3$ are essentially parallel to a common tube, we have $\# \tilde\tubes_3\lesssim \delta^{-\eta}\rho^{1-n}$ and thus $\sum_{T\in\tubes_3[\tilde T]}|Y_3(T)| \gtrapproxdelta \delta^{2\eta}\rho^{n-1}$.

If $\eta$ and $\delta_0$ are chosen sufficiently small, then we can apply Lemma \ref{rescaledCoveredTube} to each set $(\tubes_3[\tilde T], Y_3)_\delta$ with $\eps^2/3$ in place of $\eps$. Let $(\tubes_4, Y_4)_\delta$ be the union of the refined collections $\{(\tubes_3'[\tilde T], Y_3')_\delta\colon \tilde T\in\tilde\tubes_3\}$ that are the output of Lemma \ref{rescaledCoveredTube}. Let $\tilde \tubes _4=\tilde \tubes_3$. Then for each $\tilde T\in\tilde\tubes_4$, we have
\begin{equation}\label{veryGoodFineEstimate}
\#\tubes_4[\tilde T](p)\leq (\rho/\delta)^{\sigma_n+\eps^2/3},
\end{equation}
so in particular $(\tubes_4,Y_4)_\delta$ and $\tilde\tubes_4$ satisfy Item \ref{fineEstimateItem} from the statement of Proposition \ref{multiScaleStructureExtremal}. 

Furthermore, for each $\tilde T\in\tilde\tubes_4$, the unit rescaling of $(\tubes_4[\tilde T],Y_4)_\delta$ relative to $\tilde T$ satisfies Items \ref{tubesDistinct}, \ref{DefnOfMGoodCoveringItem}, and \ref{DefnOfMBigVolItem} from Definition \ref{defnOfM} with $s=t=\eps^2$. Items \ref{tubesDistinct} and \ref{DefnOfMGoodCoveringItem} will continue to hold if $\tubes_4$ is replaced by a refinement.

\medskip
\noindent{\bf Step 3: Coarse scale structure}.\\
Thus far, we have analyzed the fine-scale behavior of the ``thin'' tubes $\tubes_4[\tilde T]$ for each $\tilde T\in\tilde\tubes_4$. Our next task is to analyze the coarse-scale behavior of the ``fat'' tubes $\tilde\tubes_4$. 

First, we will construct a shading on the fat tubes. Observe that for each $\tilde T\in\tilde\tubes_4$ and each $\rho$-cube $Q$, we have that 
\begin{equation}\label{QSubsetT}
Q\ \cap\!\bigcup_{T\in\tubes_4[\tilde T]}Y_4(T)\neq\emptyset\ \ \textrm{implies}\ \ Q\subset \tilde T.
\end{equation} 
Indeed, as noted at the end of Step 1, if $T\in\tubes_4[\tilde T]$, then $T$ is covered by $\frac{1}{2 n}\tilde T$, and thus $Q$ intersects $\frac{1}{2 n}\tilde T$. But this implies $Q\subset \tilde T$. 

Each $\tilde T\in\tilde\tubes_4$ contains $\sim \rho^{-1}$ $\rho$-cubes, and for each such cube we have 
\[
\sum_{T\in\tubes_4[\tilde T]}|Y_4(T)\cap Q|\lesssim \rho \sum_{T\in\tubes_4[\tilde T]}|T|.
\]
The directions $\{\dir(T)\colon T\in\tubes_4[\tilde T]\}$ are contained in a disk of radius $\rho$ in $S^{n-1}$. Since at most $\delta^{-\eta}$ tubes from $\tubes_4$ can be essentially parallel to a common tube, we have 
\begin{equation}\label{upperBdTubes3TildeT}
\#\tubes_4[\tilde T]\lesssim \delta^{-\eta}(\rho/\delta)^{n-1},
\end{equation}
and thus
\begin{equation}\label{upperBdWeightOfQ}
\sum_{T\in\tubes_4[\tilde T]}|Y_4(T)\cap Q|\lesssim \delta^{-\eta}\rho^n.
\end{equation}

Let $c_1\sim 1$ be a constant to be chosen below. We say a $\rho$-cube is \emph{heavy} for the tube $\tilde T$ if
\[
\sum_{T\in\tubes_4[\tilde T]}|Y_4(T)\cap Q| \geq c_1 \rho \sum_{T\in\tubes_4[\tilde T]}|Y_4(T)|.
\]
Denote the set of heavy cubes by $\mathcal{Q}_0(\tilde T)$. If the constant $c_1$ is chosen appropriately, then for each $\tilde T\in\tilde\tubes_4$ we have
\[
\sum_{Q\in\mathcal{Q}_0(\tilde T)} \sum_{T\in\tubes_4[\tilde T]}|Y_4(T)\cap Q| \geq \frac{1}{2}\sum_{T\in\tubes_4[\tilde T]}|Y_4(T)|.
\]
By dyadic pigeonholing, there exists a weight $w$ so that if we define
\begin{equation}\label{defnMathcalQ}
\mathcal{Q}_1(\tilde T)=\Big\{Q\in\mathcal{Q}_0(\tilde T)\colon w \leq \sum_{T\in\tubes_4[\tilde T]}|Y_4(T)\cap Q| < 2w \Big\},
\end{equation}
then
\begin{equation}\label{sumOverQandTubes}
\sum_{\tilde T\in\tilde\tubes_4}\sum_{Q\in\mathcal{Q}_1(\tilde T)} \sum_{T\in\tubes_4[\tilde T]}|Y_4(T)\cap Q| \gtrapproxdelta\sum_{\tilde T\in\tilde\tubes_4}\sum_{T\in\tubes_4[\tilde T]}|Y_4(T)|=\sum_{T\in\tubes_4}|Y_4(T)|.
\end{equation}
Comparing \eqref{upperBdWeightOfQ} and \eqref{sumOverQandTubes}, we have
\begin{equation}\label{numerOfRhoCubes}
\sum_{\tilde T\in\tilde\tubes_4}\#\mathcal{Q}_1(\tilde T) \gtrapproxdelta \delta^{2\eta} \rho^{-n}.
\end{equation}

For each $\tilde T\in\tilde\tubes_4$, define $\tilde Y_4(\tilde T)=\bigcup_{Q\in\mathcal{Q}_1(\tilde T)}Q$. By \eqref{QSubsetT}, $Y_4$ is a shading of $\tubes_4$. \eqref{numerOfRhoCubes} is precisely the statement that
\begin{equation}\label{lowerBdWeightShading}
\sum_{\tilde T\in\tilde \tubes_4}|\tilde Y_4(\tilde T)|\gtrapproxdelta \delta^{2\eta}.
\end{equation}

Note that $(\tilde\tubes_4, \tilde Y_4)_\rho$ satisfies Items \ref{tubesDistinct}, \ref{DefnOfMGoodCoveringItem}, and \ref{DefnOfMBigVolItem} from Definition \ref{defnOfM} with $t = \eta/\eps$ and $s = 3\eta/\eps$.

Let $(\tubes_5, Y_5)_\delta$ be a refinement of $(\tubes_4,Y_4)_\delta$ and let $(\tilde\tubes_5, \tilde Y_5)_\rho$ be a refinement of $(\tilde\tubes_4,\tilde Y_4)_\rho$, so that the following holds

\begin{itemize}
\item There is a number $\mu_{\operatorname{fine}}$ so that for each $\tilde T\in \tilde\tubes_5$ and each $p \in E_{\tubes_5[\tilde T]}$, we have $\#\tubes_5[\tilde T](p) = \mu_{\operatorname{fine}}$; we do this by refining the shading $Y_4$ to ensure that this property holds.
\item There is a number $w>0$ so that for each $\tilde T\in\tilde\tubes_5$ and each $\rho$-cube $Q\subset \tilde Y_5(\tilde T)$, we have $|Q\cap E_{\tubes_5[\tilde T]}| = w$; we do this by refining the shading $Y_4$, while preserving the previous property (each $\delta$-cube is either preserved or deleted from all shadings in $\tubes_5[\tilde T]$).
\item There is a number $\mu_{\operatorname{coarse}}$ so that for each $p\in E_{\tilde\tubes_5}$ we have $\mu_{\operatorname{coarse}}\leq \#\tilde\tubes_5(p)  < 2\mu_{\operatorname{coarse}}$; we do this by refining the shading $\tilde Y_4$, while preserving the previous properties.

\item $(\tilde\tubes_5,\tilde Y_5)_\rho$ covers $(\tubes_5,Y_5)_\delta$.
\end{itemize}

Since $(\tilde\tubes_5, \tilde Y_5)_\rho$ satisfies Items \ref{tubesDistinct}, \ref{DefnOfMGoodCoveringItem}, and \ref{DefnOfMBigVolItem} from Definition \ref{defnOfM} with $t = \eta/\eps$ and $s = 4\eta/\eps$, if $\eta$ and $\delta_0$ are chosen sufficiently small (which forces $\rho$ to be sufficiently small), then $|E_{\tilde\tubes_5}|\geq \rho^{\sigma_n+\eps/2}$, and thus 
\begin{equation}\label{bdMuCoarse}
\mu_{\operatorname{coarse}}\lesssim \rho^{-\sigma_n - \eps/2 - \eta/\eps}.
\end{equation}
Thus if $\eta$ is chosen sufficiently small, then Item \ref{coarseScaleEstimateItem} from the statement of Proposition \ref{multiScaleStructureExtremal} holds. 

At this point, $(\tilde\tubes_5,\tilde Y_5)_\rho$ is a balanced cover of $(\tubes_5,Y_5)_\delta$, and this pair satisfies Items \ref{fineEstimateItem} and \ref{coarseScaleEstimateItem}. In addition, $(\tilde\tubes_5,\tilde Y_5)_\rho$ satisfies Items \ref{tubesDistinct}, \ref{DefnOfMGoodCoveringItem}, and \ref{DefnOfMBigVolItem} from Definition \ref{defnOfM} with $t = \eta/\eps$ and $s = 4\eta/\eps$, and for each $\tilde T\in\tilde\tubes_5$, the unit rescaling of  $(\tubes_5[\tilde T], Y_5)_\delta$ relative to $\tilde T$ satisfies Items \ref{tubesDistinct} and \ref{DefnOfMGoodCoveringItem} from Definition \ref{defnOfM} with $t=\eps^2.$

Next we will analyze the multiplicities $\mu_{\operatorname{fine}}$ and $\mu_{\operatorname{coarse}}$ defined above. \eqref{bdMuCoarse} gives an upper bound on $\mu_{\operatorname{coarse}}$, while by \eqref{veryGoodFineEstimate} we have $\mu_{\operatorname{fine}}\leq (\rho/\delta)^{\sigma_n+\eps^2/3}\leq \rho^{-\eps/3}(\rho/\delta)^{\sigma_n}$. On the other hand,

\begin{equation}\label{mostMassPreserved}
\int \sum_{T\in\tubes_5}\chi_{Y_5(T)}(p)dp \gtrapproxdelta \delta^{\eta}.
\end{equation}
Since the original collection of tubes $(\tubes,Y)_\delta$ is $\eta$-extremal, the integrand in \eqref{mostMassPreserved} is supported on a set of size at most $\delta^{\sigma_n - \eta}$, and it is pointwise bounded by $\mu_{\operatorname{fine}}\mu_{\operatorname{coarse}}$. We conclude that
\begin{equation}\label{lowerBdMuCoarse}
\mu_{\operatorname{fine}} \gtrapprox_\delta \delta^{2\eta}\rho^{\eps/2+\eta/\eps}(\rho/\delta)^{\sigma_n},\quad \mu_{\operatorname{coarse}}\gtrapprox_\delta \delta^{2\eta}\rho^{-\sigma_n+\eps/3}.
\end{equation}
Selecting $\eta>0$ sufficiently small, we conclude that $|E_{\tilde \tubes_5}|\leq\rho^{\sigma_n-\eps}$, and for each $\tilde T\in\tilde\tubes_5$, the unit-rescaling $(\hat\tubes,\hat Y)_{\delta/\rho}$ of each set $(\tubes_5[\tilde T], Y_5)_\delta$ satisfies $|E_{\hat T}|\leq(\delta/\rho)^{\sigma_n-\eps}$. 

We are almost done, except it is possible that the unit-rescaling described above might fail to satisfy 
\begin{equation}\label{largeMass}
\sum_{\hat T\in\hat \tubes}|\hat Y(\hat T)|\geq(\rho/\delta)^{\eps}
\end{equation}
 for some choices of $\tilde T\in\tilde\tubes_5$. To fix this problem, we define $\tilde\tubes_6\subset\tilde\tubes_5$ to be the set of tubes for which \eqref{largeMass} holds. We define $\tilde Y_6 = \tilde Y_5$, $Y_6 = Y_5$, and $\tubes_6 = \bigcup_{\tilde T\in\tilde\tubes_6}\tubes_5[\tilde T]$. Then $(\tubes_6,Y_6)$ is a refinement of $(\tubes,Y)$ which, together with $(\tilde\tubes_6,\tilde Y_6)_\rho$, satisfies the conclusions of Proposition \ref{multiScaleStructureExtremal}.
\end{proof}


\begin{rem}
The conclusions of Lemma \ref{existenceOfExtremalFamilies} and Proposition \ref{multiScaleStructureExtremal} are the only consequences of stickiness that we will use to prove Theorem \ref{mainThm}. In particular, if Lemma \ref{existenceOfExtremalFamilies} and Proposition \ref{multiScaleStructureExtremal} hold for some other class of Kakeya sets, then it should be possible to prove the analogue of Theorem \ref{mainThm} in that setting as well. 
\end{rem}

We conclude this section with a few final observations about extremal collections of tubes. First, if $(\tubes,Y)_\delta$ is an extremal collection of tubes, then the directions of the tubes passing through a typical point cannot focus too tightly. This phenomena is sometimes called ``robust transversality.'' We will record a precise version below.

\begin{lem}\label{robustTransLem}
For all $\eps>0$, there exists $\eta>0$ and $\delta_0>0$ so that the following holds for all $\delta\in (0, \delta_0]$. Let $(\tubes,Y)_\delta$ be a $\eta$-extremal collection of tubes. Then after replacing $(\tubes,Y)_\delta$ by a refinement, for each $p\in\RR^n$ and each $v\in S^{n-1}$, we have
\begin{equation}\label{focussedCount}
\#\{ T\in\tubes(p)\colon \angle(v, \dir(T) < \delta^{\eps}\} \leq \delta^{-\sigma_n +\eps\sigma_n/2}.
\end{equation}
\end{lem}
\begin{proof}
Apply Proposition \ref{multiScaleStructureExtremal} with $\eps\sigma_n/2$ in place of $\eps$, and $\rho = \delta^{\eps}$. Since the resulting collection $(\tilde\tubes, \tilde Y)_{\rho}$ of tubes are essentially distinct, the $\delta$-tubes contributing to \eqref{focussedCount} must be covered by $O(1)$ $\rho$-tubes from  $\tilde\tubes$. \eqref{focussedCount} now follows from Item \ref{muCoarseBdItem} from Proposition \ref{multiScaleStructureExtremal}.
\end{proof}

The next lemma says that extremal collections of tubes remain extremal after a mild re-scaling. The proof is straightforward (though somewhat tedious), and is omitted.
\begin{lem}\label{rescalingPreservesExtremality}
For all $\eps>0$, there exists $\eta>0$ and $\delta_0>0$ so that the following holds for all $\delta\in (0, \delta_0]$. Let $(\tubes,Y)_{\delta}$ be a $\eta$-extremal collection of tubes. Let $Q\subset[-1,1]^n$ be an axis-parallel rectangular prism, and suppose $\sum_{T\in\tubes}|Y(T)\cap Q|\geq\delta^{\eta}$. Let $\phi$ be a translation composed with a dilation of the form $(x_1,\ldots,x_n)\mapsto (r_1x_1,\ldots,r_nx_n)$, where $\delta^{\eta}\leq r_i\leq \delta^{-\eta}$, and suppose $\phi(Q)\subset[-1,1]^n$.

Then the image of $(\tubes,Y\cap Q)_{\delta}$ under $\phi$ is $\eps$-extremal. More precisely, there exists an $\eps$-extremal collection of tubes $(\tilde\tubes, \tilde Y)_{\rho}$ for some $\delta \leq \rho \leq \delta^{1-\eta}$ so that each $\tilde T\in \tilde\tubes$ has a coaxial line of the form $\phi(\ell)$, where $\ell$ is the coaxial line of a tube $T\in\tubes$. Furthermore, for each such pair $T, \tilde T$, we have that $\tilde Y(\tilde T)$ is the set of $\rho$-cubes that intersect $\phi(Y(T)\cap Q).$
\end{lem}


\section{Planiness and graininess}\label{planinessAndGraininessSection}
For the remainder of the paper we will specialize to the case $n=3$, and we will define $\sigma=\sigma_3$. Our goal is to prove that $\sigma=0$. In this section we will establish the existence of extremal collections of tubes $(\tubes,Y)_\delta$ with additional structural properties. The first of these is the existence of a Lipschitz plane map.
\begin{defn}
Let $(\tubes,Y)_\delta$ be a set of $\delta$-tubes. A \emph{plane map} for $(\tubes,Y)_\delta$ is a function $V\colon E_\tubes \to S^2$ so that
\[
|\dir(T)\cdot V(p)|\leq\delta\quad\textrm{for all}\ p\in E_\tubes,\ T\in \tubes(p).
\]
\end{defn}
Next, we will show that each $\{z=z_0\}$ slice of the Kakeya set is contained in a union of long thin rectangles, which we will call global grains. These global grains will be arranged like an Ahlfors-David regular set. The specific property we need is the following.
\begin{defn}\label{ADSetDefn}
For $n\geq 1,\ \alpha\in (0, n]$, $C\geq 1$ and $\delta>0$, We say a set $E\subset\RR^n$ is a $(\delta,\alpha, C)_n$-ADset if for all $x\in\RR^n$, all $\rho\geq\delta$, and all $r\geq\rho$, we have
\begin{equation}\label{ERhoBoundAD}
\mathcal{E}_\rho( E \cap B(x,r) ) \leq C (r/\rho)^\alpha.
\end{equation}
\end{defn}
\begin{rem}
Note that if $E$ is a $(\delta,\alpha, C)_n$-ADset, then so is every subset of $E$. This observation will be used frequently in the arguments that follow. Definition \ref{ADSetDefn} only imposes upper bounds on the size of $\mathcal{E}_\rho( E \cap B(x,r) )$, while Ahlfors regularity usually requires a matching lower bound. However, if $E$ is a $(\delta,\alpha, C)_n$-ADset that has bounded diameter and $\mathcal{E}_\delta(E)$ has size roughly $\delta^{-\alpha}$, then a large subset of $E$ will satisfy a lower bound analogue of \eqref{ERhoBoundAD}. 
\end{rem}

With these definitions, we can now state the main result of this section.

\begin{prop}\label{plainyGrainyProp}
For all $\eps,\delta_0>0$, there exists $\delta\in (0,\delta_0]$ and a $\eps$-extremal set of tubes $(\tubes,Y)_\delta$ with the following two properties.

\begin{enumerate}
\item\label{plainyGrainyPropItem1} {\bf $E_{\tubes}$ is a union of global grains with Lipschitz slope function.}\\
There is a 1-Lipschitz function $f\colon [-1,1]\to\RR$ so that $\big( E_{\tubes} \cap \{z = z_0\}\big) \cdot (1, f(z_0), 0)$ is a $(\delta, 1-\sigma, \delta^{-\eps})_1$-ADset for each $z_0\in [-1,1]$. 

\item\label{plainyGrainyPropItem2} {\bf $E_{\tubes}$ is a union of local grains with Lipschitz plane map.} \\
$(\tubes,Y)_{\delta}$ has a 1-Lipschitz plane map $V$. For all $\rho\in[\delta,1]$ and all $p\in \RR^3$, $V(p) \cdot \big( B(p,\rho^{1/2}) \cap E_\tubes\big)$ is a $(\rho, 1-\sigma, \delta^{-\eps})_1$-ADset. 

\end{enumerate}
\end{prop}

\begin{rem}
Item \ref{plainyGrainyPropItem1} implies that each slice $E_\tubes\cap \{z=z_0\}$ is contained in a union of at most $\delta^{\sigma-1-\epsilon}$ rectangles of dimensions roughly $1\times\delta$, and thus by Fubini, $|E_\tubes|\lesssim \delta^{\sigma-\epsilon}$. On the other hand, this containment is nearly sharp, since $|E_\tubes|\geq\delta^{\sigma+\eps}$.

Item \ref{plainyGrainyPropItem2} implies that each ball $B(p,\rho^{1/2})\cap N_\rho(E_\tubes)$ is a union of at most $\rho^{\sigma/2 - 1/2}$ parallel rectangular slabs of dimensions roughly $\rho^{1/2} \times \rho^{1/2} \times \rho$. Since $E_\tubes$ can be covered by $\delta^{-O(\eps)}\rho^{\sigma/2 - 3/2}$ balls of radius $\rho^{1/2}$, this implies that $|N_\rho(E_\tubes)|\lesssim \delta^{-O(\eps)}\rho^{\sigma}$. On the other hand, this containment is nearly sharp, since $|N_\rho(E_\tubes)| \gtrsim \delta^{-O(\eps)}\rho^{\sigma}$. 
\end{rem}

\begin{rem}
All of the arguments in this and previous sections also apply to the Heisenberg group example \eqref{defnOfH} (this would involve working in $\CC^3$ rather than $\RR^3$, but all of the arguments thus far could be translated to that setting). For the Heisenberg group example, we would have that the global slope function $f\colon \CC\to\CC$ is given by $f(z) = \bar z$, which is 1-Lipschitz.
\end{rem}

Our main tool for proving Proposition \ref{plainyGrainyProp} will be the multilinear Kakeya theorem of Bennett-Carbery-Tao \cite{BCT}. The version stated below can be found in \cite{CV}.
\begin{thm}[Multilinear Kakeya]\label{multilinearKakeyaThm}
Let $\tubes$ be a set of $\delta$-tubes in $\RR^3$. Then
\begin{equation}\label{trilinearEstimate}
\int \Big(\sum_{T_1,T_2,T_3\in\tubes } \chi_{T_1}\chi_{T_2}\chi_{T_3}|v_1\wedge v_2\wedge v_3|\Big)^{1/2}
\lesssim  \big(\delta^2 \#\tubes)^{3/2},
\end{equation}
where in the above expression, $v_i=\dir(T_i)$ is a unit vector pointing in the direction of the tube $T_i$, and $|v_1\wedge v_2\wedge v_3|$ is the area of the parallelepiped spanned by $v_1,v_2$, and $v_3$.
\end{thm}


\subsection{Finding a plane map}\label{findingPlaneMapSec}
In this section, we will show that every extremal set of tubes has a refinement with a plane map. First, we will show that every extremal set of tubes has a refinement that ``weakly'' agrees with a plane map.

\begin{lem}\label{existenceOfWeakPlaneMap}
There exists $\eta,\delta_0>0$ so that the following holds for all $\delta\in (0, \delta_0]$. Let $(\tubes,Y)_\delta$ be a $\eta$-extremal set of $\delta$-tubes. Then there is a refinement $(\tubes',Y')_\delta$ and a function $V\colon E_{\tubes}\to S^2$ with the following property: for each $p\in E_{\tubes'}$, we have
\begin{equation}\label{vectorAngleCondition}
|V(p)\cdot \dir (T)|\leq\delta^{\sigma/2}\quad\textrm{for all}\ T\in\tubes'(p).
\end{equation}
\end{lem}
\begin{proof}
Define $\eps =\sigma/4$. We will select $\eta$ sufficiently small to ensure that $\eta\leq \min(\eps/40,\ \eps\sigma/24)$. Apply Lemma \ref{robustTransLem} to $(\tubes,Y)_\delta$ with $\eps/4$ in place of $\eps$ to obtain a refinement $(\tubes_1, Y_1)_{\delta}$ so that for each $p\in\RR^3$ and each unit vector $v\in S^2$, we have
\begin{equation}\label{transEstimate}
\# \{ T\in\tubes_1(p)\colon \angle(v, \dir(T))\leq \delta^{\eps/4}\}\leq  \delta^{-\sigma+(\eps/4)\sigma/2}\leq \delta^{-\sigma+3\eta}.
\end{equation}
Let $(\tubes_2,Y_2)$ be a constant multiplicity refinement of $(\tubes_1,Y_1)$, with multiplicity $\delta^{-\sigma+2\eta}\lessapprox\mu\lesssim\delta^{-\sigma-2\eta}$. 
By Theorem \ref{multilinearKakeyaThm} we have
\begin{equation}\label{trilinearEstimateT1}
\int \Big(\sum_{T_1,T_2,T_3\in\tubes_2 } \chi_{Y_2(T_1)}\chi_{Y_2(T_1)}\chi_{Y_2(T_3)}|v_1\wedge v_2\wedge v_3|\Big)^{1/2}
\lesssim  \big(\delta^2 \#\tubes_2)^{3/2}\lesssim \delta^{-(3/2)\eta}.
\end{equation}

We say a point $p\in E_{\tubes_2}$ is \emph{broad} if
\[
\#\{(T_1,T_2,T_3)\in (\tubes_2(p))^3 \colon |v_1\wedge v_2\wedge v_3| \geq \delta^{\sigma - 10\eta} \} \geq \delta^{-3\sigma+3\eta},
\]
otherwise we say it is narrow (in the above expression and the expressions to follow, we set $v_i = \dir(T_i)$). Both the set of broad and the set of narrow points are unions of $\delta$-cubes. Observe that each broad point contributes at least $\delta^{-\sigma-(7/2)\eta}$ to the integrand on the LHS of \eqref{trilinearEstimateT1}, and thus by \eqref{trilinearEstimateT1} the measure of the set of broad points $\lesssim \delta^{\sigma+2\eta}$. On the other hand, $E_{\tubes_2}$ has measure $\gtrapprox \delta^{\sigma+\eta}$. 

Define $Y_3(T) = \{p \in Y_2(T)\colon p\ \textrm{narrow}\}$, and define $\tubes_3=\tubes_2$. Then it is still true that $\#\tubes_3(p)=\mu$ for each $p\in E_{\tubes_3}$, and $|E_{\tubes_3}|\geq\frac{1}{2}|E_{\tubes_2}|$.

For each $p\in E_{\tubes_3}$, we have 
\begin{equation}\label{lotsOfLowTrilin}
\#\{(T_1,T_2,T_3)\in \tubes_3(p)^3 \colon |v_1\wedge v_2\wedge v_3| < \delta^{\sigma - 10\eta} \} \geq \mu^3 - \delta^{-3\sigma+3\eta} \geq \frac{1}{2}\mu^3.
\end{equation}
On the other hand, by \eqref{transEstimate}, for each $i,j\in \{1,2,3\}$ with $i\neq j$ we have 
\begin{equation}\label{lotsOfBilin}
\#\{(T_1,T_2,T_3)\in \tubes_3(p)^3 \colon |v_i\wedge v_j| < \delta^{\eps/4} \} \leq \mu^2 \delta^{-\sigma+3\eta} \leq \frac{1}{100}\mu^3.
\end{equation}
Comparing \eqref{lotsOfLowTrilin} and \eqref{lotsOfBilin}, we conclude that there are at least $\mu^3/4$ triples $(T_1,T_2,T_3)$ with $|v_1\wedge v_2\wedge v_3| < \delta^{\sigma - 2\eta}\leq\delta^{\sigma-\eps/4}$ and $|v_i\wedge v_j| \geq \delta^{\eps/4}$ for each $i\neq j$. By pigeonholing, there exists $T_1,T_2$ so that there are at least $\mu/4$ choices of $T_3$ so that $(T_1,T_2,T_3)$ is a triple of the above form. Define 
\[
V(p)=\frac{\dir(T_1)\times \dir(T_2)}{|\dir(T_1)\times \dir(T_2)|}.
\]
Then for each $T_3$ of the type described above, we have 
\begin{equation}\label{VpEqn}
|V(p)\cdot v(T_3)| = \frac{|\dir(T_1)\wedge \dir(T_2) \wedge \dir(T_3)|}{|\dir(T_1)\times \dir(T_2)|}\leq \frac{\delta^{\sigma - 10\eta}}{\delta^{\eps/4}} \leq\delta^{\sigma/2}.
\end{equation}
Finally, let $\tubes'=\tubes_3$, and define the shading 
\begin{equation}\label{defnY5}
Y'(T) = \{p \in Y_3(T)\colon |V(p)\cdot v(T)| \leq \delta^{\sigma/2} \}.
\end{equation}
$(\tubes',Y)_\delta$ is a refinement of $(\tubes,Y)_\delta$ that satisfies \eqref{vectorAngleCondition}.
\end{proof}

The next lemma shows that by moving to a coarser scale, we can ensure that a suitable thickening of the tubes from Lemma \ref{existenceOfWeakPlaneMap} strongly agrees with their plane map.
\begin{lem}\label{existenceOfPlaneMap}
For all $\eps>0$, there exists $\eta,\delta_0>0$ so that the following holds for all $\delta\in (0, \delta_0]$. Let $(\tubes,Y)_\delta$ be a $\eta$-extremal set of $\delta$-tubes, and let $\rho\in [\delta^{1-\eps}, \delta^{\eps}]$. Suppose there exists a function $V\colon E_{\tubes}\to S^2$ so that for each $p\in E_{\tubes}$, we have
\begin{equation}\label{vectorAngleConditionHypothesis}
|V(p)\cdot v(T)|\leq \rho/2\quad\textrm{for all}\ T\in\tubes(p).
\end{equation}
Then there exists a refinement $(\tubes',Y')_\delta$ of $(\tubes,Y)_\delta$ and a $\eps$-extremal set of $\rho$-tubes $(\tilde\tubes, \tilde Y)_\rho$ with plane map $\tilde V$ so that $(\tilde\tubes, \tilde Y)_\rho$ covers $(\tubes', Y')_\delta$. Furthermore, $\tilde V$ is constant on each $\rho$-cube, and $V$ and $\tilde V$ are consistent on each $\rho$-cube, in the sense that there exists $p\in Q$ with $\tilde V(Q) = V(p)$.
\end{lem}
\begin{proof}
Let $\eps_1>0$ be a small constant to be chosen below. We will select $\eta$ very small compared to $\eps_1$, and $\eps_1$ small compared to $\eps$. Apply Proposition \ref{multiScaleStructureExtremal} to $(\tubes, Y)_{\delta}$, with $\eps_1$ in place of $\eps$ and with $\rho$ as stated in the lemma. We obtain a refinement $(\tubes_1,  Y_1)_{\delta}$ of $(\tubes, Y)_{\delta}$ and a $\eps_1$-extremal collection of $\rho$-tubes $(\tilde\tubes_1,\tilde Y_1)_\rho$ that covers $(\tubes_1,Y_1)_\delta$. 

After refining $(\tilde\tubes_1,\tilde Y_1)_\rho$ (which induces a refinement on $(\tubes_1,  Y_1)_{\delta}$), we may suppose that for each $\rho$-cube $Q\subset E_{\tilde \tubes_1}$, there is a point $p_Q\in Q$ with $\#\tubes_1(p_Q)\gtrapprox \delta^{-\sigma+2\eta}$. We say the tube $\tilde T\in\tilde\tubes_1$ is \emph{associated} to the $\rho$-cube $Q$ if $p_Q \in \bigcup_{T\in\tubes_1[\tilde T]}Y_1(T)$; this necessarily implies $Q\subset\tilde Y_1(\tilde T)$. By Item \ref{fineEstimateItem} from Proposition \ref{multiScaleStructureExtremal}, each $\rho$-cube $Q\subset E_{\tilde T_1}$ satisfies
\[
\#\{\tilde T\in \tilde T_1\colon \tilde T\ \textrm{associated to}\ Q\} \gtrapprox \frac{\delta^{-\sigma+2\eta}}{(\rho/\delta)^{\sigma+\eps_1}}\gtrapprox \delta^{2\eta+\eps_1}\rho^{-\sigma}.
\]
On the other hand, by Item \ref{coarseScaleEstimateItem} from Proposition \ref{multiScaleStructureExtremal}, the number of tubes in $\tilde \tubes_1$ whose shading contains $Q$ is at most $\rho^{-\sigma-\eps_1}$. Thus if we define $\tilde Y_2(\tilde T)$ to be the union of $\rho$-cubes associated to $\tilde T$ and define $\tilde\tubes_2=\tilde\tubes_1$ then 
\begin{equation}\label{massT3}
\sum_{\tilde T\in\tilde\tubes_2}|\tilde Y_2(\tilde T)|\gtrapprox \delta^{2\eta+2\eps_1}\sum_{\tilde T\in\tilde\tubes_1}|\tilde Y_1(T)|\gtrapprox \delta^{2\eta+3\eps_1}.
\end{equation}
Selecting $\eps_1$ sufficiently small, we can ensure that the LHS of \eqref{massT3} is at least $\rho^{\eps}$. 

Next, we shall define the function $\tilde V(p)\colon E_{\tilde\tubes_2}\to S^2$ as follows. For each $\rho$-cube $Q\subset E_{\tilde\tubes_2}$, we define $\tilde V(p) = V(p_Q)$ for all $p\in Q$. Clearly $\tilde V$ is constant on $\rho$-cubes, and satisfies the consistency condition with $V$ stated in Lemma \ref{existenceOfPlaneMap}.
It remains to show that $\tilde V$ is a plane map for $(\tilde\tubes_2, \tilde Y_2)_\rho.$ Let $p\in E_{\tilde\tubes_2}$ be contained in a $\rho$-cube $Q$, and let $\tilde T\in \tilde\tubes_2(p)$. Then there exists a tube $T\in  \tubes_2(T)\cap \tubes(p)$. We have
\[
|\tilde V(p)\cdot \dir(\tilde T)| = |V(p_Q)\cdot \dir(\tilde T)|\leq |V(p_Q)\cdot \dir(T)| + | V(p_Q)\cdot (\dir(\tilde T) - \dir( T)| \leq \rho/2 + |\dir(\tilde T) - \dir( T)| \leq \rho,
\]
where the final inequality used \eqref{vectorAngleConditionHypothesis} and the observation that $|\dir(\tilde T) - \dir( T)|\leq\rho/2$ (since $\tilde T$ covers $T$).  
\end{proof}


\subsection{Lipschitz regularity of the plane map}\label{lipRegularityPlaneMapSec}
In this section we will show that if an extremal collection of tubes has a plane map, then after restricting to a sub-collection, this plane map must be Lipschitz with controlled Lipschitz norm. 

\begin{lem}\label{lipschitzContRhoCube}
For all $\eps>0,$ there exists $\eta>0,\delta_0>0$ so that the following holds for all $\delta\in (0, \delta_0]$. Let $(\tubes,Y)_\delta$ be a $\eta$-extremal collection of tubes with plane map $V$, and let $\rho\in [\delta^{1-\eps}, \delta^{\eps}]$. Then there exists a $\eps$-extremal sub-collection $(\tubes',Y')_\delta$, so that if $p,q\in E_{\tubes'}$ are contained in a common $\rho$-cube, then $|V(p)-V(q)|\leq \rho$. 
\end{lem}

\begin{proof}
Let $\eps_1,\eps_2>0$ be small constants to be chosen below. We will select $\eta$ very small compared to $\eps_1$, $\eps_1$ very small compared to $\eps_2$, and $\eps_2$ very small compared to $\eps$. 

Let $(\tubes_1,Y_1)_\delta$ be the refinement obtained by applying Lemma \ref{robustTransLem} to $(\tubes,Y)_\delta$ with $\eps_1$ in place of $\eps$. Next, apply Proposition \ref{multiScaleStructureExtremal} to $(\tubes_1,Y_1)_\delta$ with $\eps_2$ in place of $\eps$, and $\rho$ as specified in the lemma. We obtain a refinement $(\tubes_2,Y_2)_\delta$ of $(\tubes_1,Y_1)$ and a $\eps_2$-extremal set of $\rho$-tubes $(\tilde\tubes, \tilde Y)_{\rho}$.

After refining $(\tubes_2,Y_2)_\delta$ and $(\tilde\tubes, \tilde Y)_{\rho}$, we may suppose that $(\tubes_2,Y_2)_\delta$ has constant multiplicity $\mu\gtrapprox \delta^{-\sigma+2\eta}$, and $(\tilde\tubes, \tilde Y)_{\rho}$ is still a balanced cover of $(\tubes_2,Y_2)_\delta$. If $\eta>0$ is selected sufficiently small (depending on $\eps_1$ and $\sigma$), then for each $p\in E_{\tubes_2}$, there are at least $\mu^2/2$ pairs $T,T'\in\tubes_2(p)$ with $\angle(\dir(T),\dir(T'))\geq\delta^{\eps_1}$. Thus by Item \ref{fineEstimateItem} from Proposition \ref{multiScaleStructureExtremal}, there are at least $(\mu^2/2) / \big((\rho/\delta)^{\sigma+\eps_2}\big)^2\gtrapprox \delta^{4\eta+2\eps_2}\rho^{-2\sigma}$ pairs $\tilde T,\tilde T'\in\tilde\tubes$ with $p\in \bigcup_{T\in\tubes_2[\tilde T]}Y_2(T)$, $p\in \bigcup_{T\in\tubes_2[\tilde T']}Y_2(T)$, and $\angle\big(\dir(\tilde T),\dir(\tilde T')\big)\geq\delta^{\eps_1}-2\rho$; we will choose $\eps_1$ sufficiently small so that $\delta^{\eps_1}-2\rho\geq\delta^{\eps_1}/2$. In particular, for each $\rho$-cube $Q\subset E_{\tilde\tubes}$, we have 
\begin{equation}\label{numberOfTriples}
\begin{split}
\Big|\Big\{(p, \tilde T, \tilde T')\in (Q\cap E_{\tubes_2}) \times \tilde\tubes^2\colon p\in \bigcup_{T\in\tubes_2[\tilde T]}Y_2(T) \cap \bigcup_{T\in\tubes_2[\tilde T']}Y_2(T),\ &\angle\big(\dir(\tilde T),\dir(\tilde T')\big)\geq \delta^{\eps_1}/2\Big\} \Big|\\
&\gtrapprox |Q\cap E_{\tubes_2}| \delta^{4\eta+2\eps_2}\rho^{-2\sigma}\\
&\gtrapprox \delta^{4\eta+4\eps_2} |Q\cap E_{\tubes_2}| \big(\#\tilde\tubes(Q)\big)^2,
\end{split}
\end{equation}
where $|\cdot|$ on the LHS of \eqref{numberOfTriples} denotes the product of Lebesgue measure in $\RR^3$ and counting measure, and the final inequality used Item \ref{coarseScaleEstimateItem} from Proposition \ref{multiScaleStructureExtremal}. Thus by pigeonholing, we can select two tubes from $\tilde\tubes$, which we will denote by $\tilde T_Q$ and $\tilde T'_Q,$ with $\angle\big(\dir(\tilde T_Q),\dir(\tilde T_Q')\big)\geq\delta^{\eps_1}/2,$ so that
\begin{equation}\label{defnFQ}
\Big|Q\ \cap \Big(\bigcup_{T\in\tubes_2[\tilde T_Q]}Y_2(T)\Big) \cap \Big(\bigcup_{T\in\tubes_2[\tilde T_Q']}Y_2(T)\Big)\Big|\gtrapprox \delta^{4\eta+4\eps_2}|Q\cap E_{\tubes_2}|.
\end{equation}

\begin{figure}[h]
\centering
\begin{overpic}[scale=0.3]{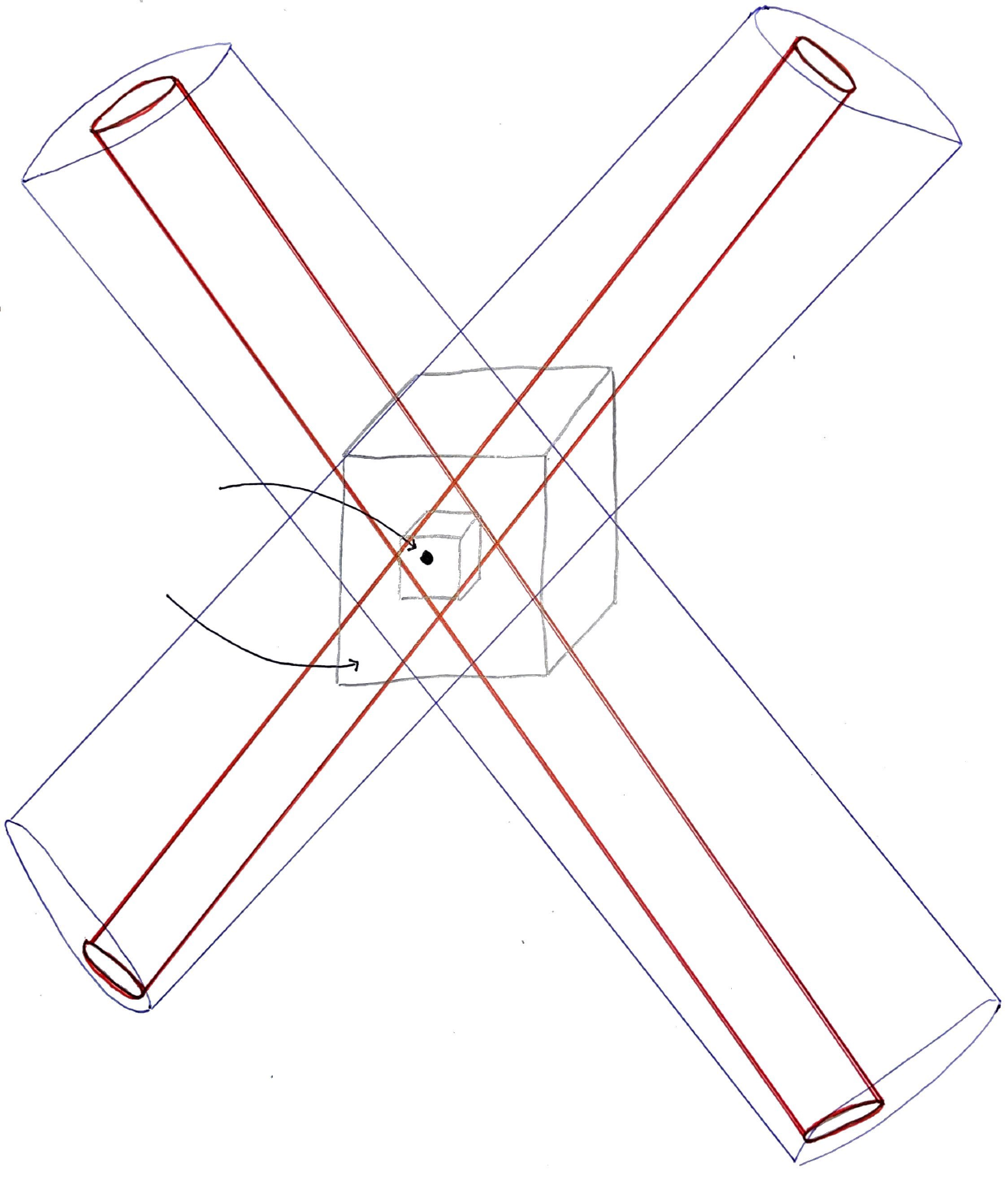}
 \put (15,57) {$p$}
 \put (10,52) {$Q$}
 \put (70,35) {$\tilde{T}_Q$}
 \put (67,68) {$\tilde{T}_Q'$}
 \put (55,24) {$\tilde{T}_p$}
 \put (55,78) {$\tilde{T}_p'$}
\end{overpic}
\caption{$V(p)$ is almost orthogonal to $\dir(\tilde T_Q)$ and $\dir(\tilde T_Q')$.}
\label{lipInCubeFig}
\end{figure}

Define $F_Q$ to be the set on the LHS of \eqref{defnFQ}. Let $\tubes_3=\tubes_2$, and for each $T\in\tubes_3$ define
\[
Y_3(T) = Y_2(T) \cap  \bigcup_Q F_Q.
\]
Then $\sum_{T\in\tubes_3}|Y_3(T)| \gtrapprox\delta^{5\eta+4\eps_2}$. Let $Q$ be a $\rho$-cube, and let $p\in E_{\tubes_3}\cap Q$. Then there exists $T_p\in\tubes_3[\tilde T_Q]$ and $T_p'\in\tubes_3[\tilde T_Q']$ so that $p\in Y_3(T_p)\cap Y_3(T_p')$, and hence
\[
|V(p)\cdot \dir(\tilde T_Q)|\leq |V(p)\cdot \dir(T_p)| + |V(p)\cdot (\dir(\tilde T_Q) - \dir(T_p))|\leq \delta + \rho,
\]
and similarly for $|V(p)\cdot \dir(\tilde T'_Q)|$. See Figure \ref{lipInCubeFig}. Since $\angle\big( \dir(\tilde T_Q),\ \dir(\tilde T'_Q)\big)\gtrsim\delta^{\eps_1}$, we conclude that for all $p\in Q$, we have
\begin{equation}\label{almostGoodAngle}
\angle\Big(V(p),\ \frac{\dir(\tilde T_Q)\times \dir(\tilde T'_Q)}{|\dir(\tilde T_Q)\times \dir(\tilde T'_Q)|}\Big)\lesssim \delta^{-\eps_1}\rho.
\end{equation}
\eqref{almostGoodAngle} says that $(\tubes_3,Y_3)$ almost satisfies the conclusion of Lemma \ref{lipschitzContRhoCube}, except the RHS of \eqref{almostGoodAngle} is  $\delta^{-\eps_1}\rho$ rather than $\rho$. To fix this, we will pigeonhole to select a set $E_4\subset E_{\tubes_3}$ that is a union of $\delta$-cubes so that (i): for each $\rho$-cube $Q$, $\angle(V(p),V(q))\leq\rho$ for all $p,q\in Q$, and (ii) if we define $Y_4(T) = Y_3(T)\cap E_4$ and $\tubes_4 = \tubes_3$, then $\sum_{T\in\tubes_4}|Y_4(T)|\gtrsim \delta^{2\eps_1}\sum_{T\in\tubes_3}|Y_3(T)|$. If we choose $\eps_1,\eps_2,\eta$ sufficiently small, then $(\tubes_4,Y_4)_\delta$ will be $\eps$-extremal. 
\end{proof}


Applying Lemma \ref{lipschitzContRhoCube} and then deleting a constant fraction of the $\rho$-cubes to obtain a $2\rho$-separated set of cubes, we obtain the following
\begin{cor}\label{lipschitzAtOneScaleCor}
For all $\eps>0$, there exists $\eta>0,\delta_0>0$ so that the following holds for all $\delta\in (0, \delta_0]$. Let $(\tubes,Y)_\delta$ be a $\eta$-extremal collection of tubes with plane map $V$, and let $\rho\in [\delta^{1-\eps}, \delta^{\eps}]$. Then there exists a $\eps$-extremal sub-collection $(\tubes',Y')_\delta$, so that if $p,q\in E_{\tubes'}$ satisfy $|p-q|\leq\rho$, then $|V(p)-V(q)|\leq 2\rho$. 
\end{cor}


Corollary \ref{lipschitzAtOneScaleCor} says that after restricting to a sub-collection, the plane map $V$ will be Lipschitz at a specific scale $\rho$. The next lemma iterates this result at many scales. 

\begin{lem}
For all $\eps>0$, there exists $\eta>0,\delta_0>0$ so that the following holds for all $\delta\in (0, \delta_0]$. Let $(\tubes,Y)_\delta$ be a $\eta$-extremal collection of tubes with plane map $V$ that is constant on $\delta$-cubes. Then there exists a $\eps$-extremal sub-collection $(\tubes',Y')_\delta$ so that $V$ is $\delta^{-\eps}$-Lipschitz on $E_{\tubes'}$. 
\end{lem}
\begin{proof}
After a refinement, we can suppose that if $p,q\in E_{\tubes}$ satisfy $|p-q|\leq\delta$, then $p$ and $q$ are contained in a common $\delta$-cube, and hence $V(p)=V(q)$. Let $N = 2/\eps$. If $\eta,\delta_0$ are selected sufficiently small, then we can apply Corollary \ref{lipschitzAtOneScaleCor} iteratively to $(\tubes,Y)_\delta$ for $\rho=\rho_j = \delta^{j/N},$ $j=1,\ldots,N-1$, and the resulting sub-collection $(\tubes',Y')_\delta$ will be $\eps$-extremal. We claim that $V$ is $2\delta^{-1/N}$-Lipschitz on $(\tubes',Y')_\delta$. To see this, let $p,q\in E_{\tubes'}$. If $|p-q|\leq\delta$ then $V(p)=V(q)$ and we are done. If $|p-q|> \delta^{1/N},$ then $|V(p)-V(q)|\leq 2$, and again we are done. Otherwise, there is an index $1\leq k\leq N-1$ so that $\delta^{(k+1)/N}<|p-q|\leq \delta^{k/N}$, and
\[
|V(p)-V(q)|\leq \delta^{k/N} \leq \delta^{-1/N}|p-q|.\qedhere
\] 
\end{proof}

Combining the results of this section with Lemma  \ref{existenceOfExtremalFamilies}, we have the following
\begin{lem}\label{existenceOfLipschitzPlaneMap}
For all $\eps,\delta_0>0$, there exists $\delta\in (0,\delta_0]$ and a $\eps$-extremal set of tubes $(\tubes,Y)_\delta$ that has a $\delta^{-\eps}$-Lipschitz plane map. 
\end{lem} 


\subsection{Grains}\label{grainSec}

In this section, we will show that if an extremal set of tubes has a Lipschitz plane map, then after restricting to a sub-collection, the tubes arrange themselves into parallel slabs, which are called grains. In what follows, $V$ will be a plane map for a collection of tubes. If $Q$ is a cube with center $p$, we define $V(Q) = V(p_Q)$ (this is a unit vector in $\RR^3$), and we define the projection $\pi_Q(w) = w\cdot V(Q)$.

Our goal is to show that the set of grains inside a cube arrange themselves into a AD-regular set, in the sense of Definition \ref{ADSetDefn}. As a first step, the next lemma controls how many grains can occur inside a cube at a given scale.
\begin{lem}\label{coveringNumBdOneScale}
For all $\eps>0,$ there exists $\eta>0,\delta_0>0$ so that the following holds for all $\delta\in (0, \delta_0]$. Let $(\tubes,Y)_\delta$ be a $\eta$-extremal collection of tubes with $\delta^{-\eta}$-Lipschitz plane map $V$. Let $\rho\in [\delta^{1-\eps}, \delta^{\eps}]$ and let $\tau\in [\rho^{1-\eps}, \rho^{1/2}]$. 

Then there exists a $\eps$-extremal sub-collection $(\tubes',Y')_\delta$, so that for each $\tau$-cube $Q$, we have
\begin{equation}\label{rhoNbhdProjBd}
\mathcal{E}_\rho\Big( V(Q) \cdot \big( E_{\tubes'} \cap Q\big) \Big) \leq \rho^{-\eps}(\tau/\rho)^{1-\sigma}.
\end{equation}
See Figure \ref{smallProjectionGrainDirFig}.
\end{lem}

\begin{figure}[h]
\centering
\begin{overpic}[scale=0.4]{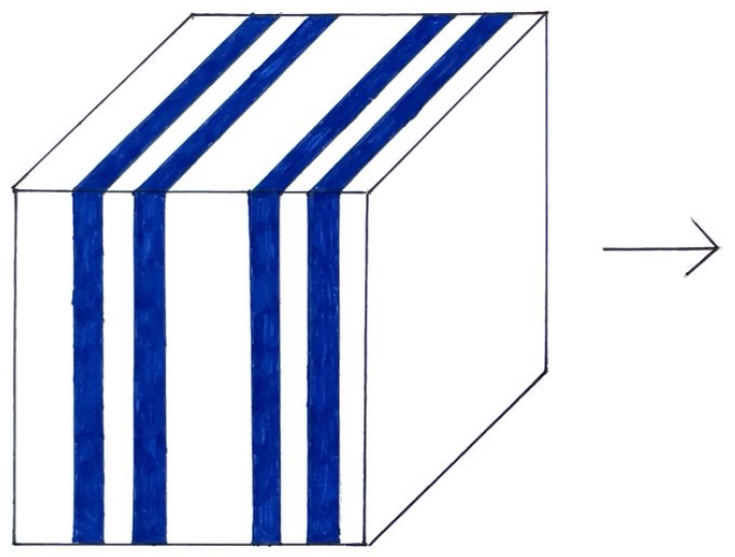}
  \put (60,40) {$Q$}
   \put (80,49) {$V(Q)$}

\end{overpic}
\caption{The set $E_{\tubes'} \cap Q$ (blue) at scale $\rho$, and the unit vector $V(Q)$}
\label{smallProjectionGrainDirFig}
\end{figure}

\begin{proof}
Let $\eps_1,\eps_2,\eps_3>0$ be small constants to be chosen below. We will select $\eta$ very small compared to $\eps_1$, $\eps_1$ very small compared to $\eps_2$, $\eps_2$ very small compared to $\eps_3$, and $\eps_3$ very small compared to $\eps$. 

Applying Proposition \ref{multiScaleStructureExtremal} to $(\tubes,Y)_\delta$ with $\eps_1/2$ in place of $\eps$, and then again to the output $(\tilde\tubes,\tilde Y)_\rho$ with $\tau$ in place of $\rho$ and $\eps_2$ in place of $\eps$, we obtain a refinement $(\tubes_1,Y_1)_\delta$ of $(\tubes,Y)_\delta$ that is covered by a $\eps_1$-extremal set $(\tilde\tubes, \tilde Y)_\rho$, which in turn is covered by a $\eps_2$-extremal set $(\dbtilde\tubes, \dbtilde Y)_\tau$. Note that $V$ is also a plane map for $(\tilde\tubes, \tilde Y)_\rho$ and $(\dbtilde\tubes, \dbtilde Y)_\tau$.

After refining $(\tubes_1,Y_1)_\delta$ and $(\tilde\tubes, \tilde Y)_\rho$, we can ensure that $(\tilde\tubes, \tilde Y)_\rho$ is a balanced cover of $(\tubes_1,Y_1)_\delta$. In the arguments below, we will find a sub-collection of $(\tilde\tubes, \tilde Y)_\rho$ with certain favorable properties. Since $(\tilde\tubes, \tilde Y)_\rho$ is a balanced cover of $(\tubes_1,Y_1)_\delta$, this sub-collection of $(\tilde\tubes, \tilde Y)_\rho$ will induce a sub-collection of $(\tubes_1,Y_1)_\delta$ of the same relative density. 

Apply Lemma \ref{robustTransLem} to $(\tilde\tubes,\tilde Y)_\rho$, with $\eps_3$ in place of $\eps$. After further refinement, we can find a collection $(\tilde\tubes_1,\tilde Y_1)_\rho$ with the following property (\emph{P}): For each $\tilde T\in \tilde\tubes_1$ and each $\tau$-cube $Q$ that intersects $\tilde Y_1(\tilde T)$, we have $|Q\cap \tilde Y_1(\tilde T)|\geq \rho^{2+2\eps_1}\tau$ (for comparison, the maximal size of an intersection of the form $\tilde T\cap Q$ is bounded by $100\rho^2\tau$).

Let $(\tilde\tubes_2,\tilde Y_2)_\rho$ be a constant multiplicity refinement of $(\tilde\tubes_1,\tilde Y_1)_\rho$, and let $(\tilde\tubes_3,\tilde Y_3)_\rho$ be a refinement of $(\tilde\tubes_2,\tilde Y_2)_\rho$ that again has property (P).  The reason for the above refinements is as follows: If $\tilde T\in\tilde \tubes_3$ and if $Q$ is a $\tau$-cube that intersects $\tilde Y_3(\tilde T)$, then $|Q\cap \tilde Y_3(\tilde T)|\geq \rho^{2+2\eps_1}\tau$. Thus we can choose a set of points $p_1,\ldots p_k \in Q\cap \tilde Y_3(\tilde T)$ that are $\geq \rho^{1-\eps_3}$ separated, with $k\gtrsim \rho^{2\eps_1+\eps_3}\rho^{-1}\tau$. For each such point $p_i$, we have $\#\tilde\tubes_2(p_i)=\mu\gtrapprox\rho^{-\sigma+\eps_1}$. But since we refined $(\tilde\tubes,\tilde Y)_\rho$ using Lemma \ref{robustTransLem}, if $\eps_1$ is sufficiently small compared to $\eps_3$, then there is a tube $\tilde T_i\in \tilde\tubes_1(p_i)$ with $\angle(\dir(\tilde T), \dir(\tilde T_i))\geq\rho^{\eps_3}$.

We claim that the union of the sets $\{Q\cap \tilde Y(\tilde T_i)\}_{i=1}^k$ has large volume. Since the points $p_i$ are $\rho^{1-\eps_3}$ separated and the tubes $\tilde T_i$ make angle $\geq\rho^{\eps_3}$ with $\tilde T$, we can label the points $p_1,\ldots,p_k$ so that $\dist\big(\tilde T_i\cap\tilde T,\ \tilde T_j\cap\tilde T\big)\gtrsim \rho^{1-\eps_3}|i-j|$. Then if distinct tubes $\tilde T_i$ and $\tilde T_j$ intersect inside $Q$, we must have $\angle(v (\tilde T_i), v(\tilde T_j))\gtrsim \rho^{1-\eps_3}|i-j|/\tau$. But this implies that for distinct $\tilde T_i$ and $\tilde T_j$ we have
\[
|Q\cap\tilde T_i\cap\tilde T_j|\lesssim \frac{\rho^3}{\angle(v (\tilde T_i), v(\tilde T_j))+\rho}\lesssim \frac{\rho^{2+\eps_3}\tau}{|i-j|}.
\]
Thus we have the following Cordoba-type $L^2$ estimate.
\begin{equation}\label{CordobaEstimate}
\Big| Q \cap \bigcup_{i=1}^k \tilde Y_1(\tilde T_i)\Big| 
\geq 
\frac{\Big(\sum_{i=1}^k |Q\cap \tilde Y_1(\tilde T_i)|\Big)^2}
{\big\Vert\sum_{i=1}^k \chi_{\tilde T_i}\big\Vert_{L^2(Q)}^2}
\gtrapprox \frac{k^2 \big(\rho^{2\eps_1+2}\tau\big)^2}
{ (k\rho^{2+\eps_3}\tau)}
\geq \rho^{6\eps_1+\eps_3}\tau^2\rho.
\end{equation}

\begin{figure}[h]
\centering
\begin{overpic}[scale=0.4]{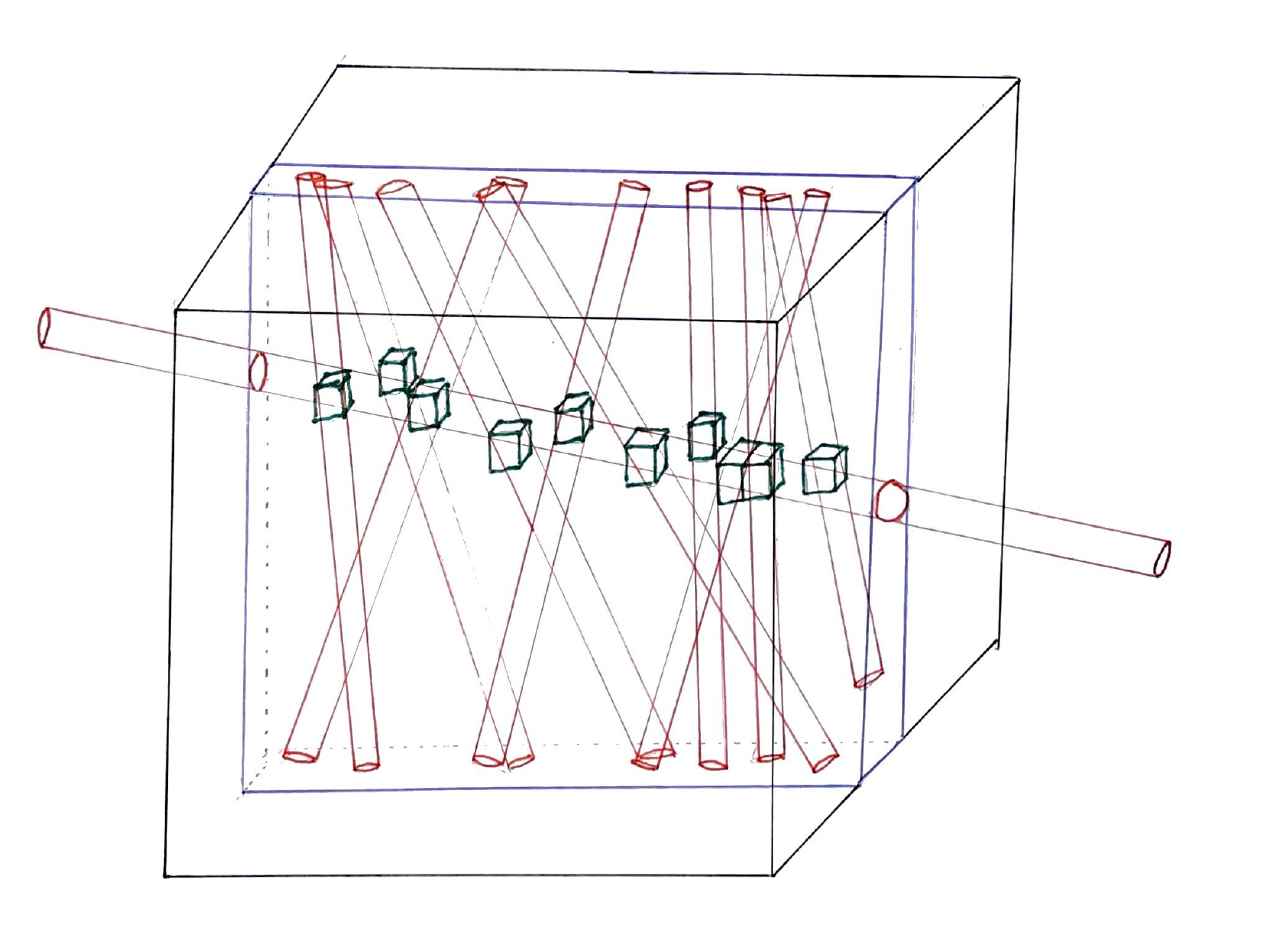}
 \put (26,35) {{\tiny $\tilde T_1$}}
 \put (33,35) {{\tiny $\tilde T_2$}}
 \put (65,33) {{\tiny $\tilde T_k$}}

\end{overpic}
\caption{The union of the sets $Q\cap \tilde Y_1(T_i),\ i=1,\ldots,k$ fill out most of the slab (blue) inside $Q$.}
\label{fullGrainFig}
\end{figure}

\eqref{CordobaEstimate} says that $\bigcup_{i=1}^k \tilde Y_1(\tilde T_i)$ has large volume inside of $Q$. On the other hand, we claim that this set is contained inside a rectangular prism of dimensions roughly $\delta^{-\eta}\rho\times\tau\times\tau$. Indeed, since the plane map $V$ is $\delta^{-\eta}$-Lipschitz and $Q$ has diameter $\leq 2\tau$, if $q\in Q\cap \tilde Y_3(\tilde T)$ then we have
\[
Q \cap\bigcup_{i=1}^k \tilde T_i\ \subset\ Q\cap N_{4\delta^{-\eta}\rho}(q + V(Q)^\perp). 
\]
This is illustrated in Figure \ref{fullGrainFig}. Thus
\begin{equation}\label{unionContainedSlab}
V(q) \cdot \Big( Q \cap\bigcup_{i=1}^k \tilde T_i\Big) \subset [V(Q)\cdot q - 4\delta^{-\eta}\rho,\ V(Q)\cdot q + 4\delta^{-\eta}\rho].
\end{equation}
Note that $V(Q)\cdot (Q \cap \tilde T)$ is also contained in this same interval.

Combining \eqref{CordobaEstimate} and \eqref{unionContainedSlab}, we conclude that if $t\in V(Q) \cdot \big(Q \cap E_{\tilde\tubes_3}\big)$, then
\begin{equation}\label{lowerBdInverseImageT}
\Big| Q \cap E_{\tilde\tubes_1} \cap \pi_Q^{-1}( [t - 4\delta^{-\eta}\rho,\ t + 4\delta^{-\eta}\rho] )\Big| \gtrapprox \rho^{6\eps_1+\eps_3}\tau^2\rho.
\end{equation}
Informally, \eqref{lowerBdInverseImageT} says that each $\rho$-separated point $t\in V(Q) \cdot \big(Q \cap E_{\tilde\tubes_3}\big)$ corresponds to a $\rho\times \tau\times\tau$ grain (i.e.~rectangular prism whose short axis is parallel to $V(Q)$) that has nearly full intersection with $ Q\cap E_{\tilde\tubes_1}$. 

If $|t-t'|>8\delta^{-\eta}\rho$ then the sets $\pi_Q^{-1}( [t - 4\delta^{-\eta}\rho,\ t + 4\delta^{-\eta}\rho]$ and $\pi_Q^{-1}( [t'- 4\delta^{-\eta}\rho,\ t' + 4\delta^{-\eta}\rho]$ are disjoint, and hence \eqref{lowerBdInverseImageT} implies that
\[
\mathcal{E}_{8\delta^{-\eta}\rho}\Big(V(Q) \cdot \big(Q \cap E_{\tilde\tubes_3}\big)\Big) \lessapprox \frac{|Q\cap E_{\tilde\tubes_1}|}{\rho^{6\eps_1+\eps_3}\tau^2\rho}\leq \frac{\rho^{\sigma-\eps_1}\tau^{3-\sigma-\eps_2}}{\rho^{6\eps_1+\eps_3}\tau^2\rho}\lesssim\rho^{-7\eps_1-\eps_2-\eps_3}(\tau/\rho)^{1-\sigma},
\]
and thus
\begin{equation}\label{cubeEstimate}
\mathcal{E}_{\rho}\Big(V(Q) \cdot \big(Q \cap E_{\tilde\tubes_3}\big)\Big) \leq 8\delta^{-\eta}\mathcal{E}_{8\delta^{-\eta}\rho}\Big(V(Q) \cdot \big(Q \cap E_{\tilde\tubes_3}\big)\Big) \lessapprox \delta^{-\eta}\rho^{-7\eps_1-\eps_2-\eps_3}(\tau/\rho)^{1-\sigma}.
\end{equation}
To conclude the proof, select $\eps_1,\eps_2,\eps_3,\eta$ sufficiently small (depending on $\eps$) so that the LHS of \eqref{cubeEstimate} is at most $\rho^{-\eps}(\tau/\rho)^{1-\sigma}$ (this is possible since $\rho\leq\delta^{\eps}$). Let $(\tubes_3,Y_3)_\delta$ be the refinement of $(\tubes_1,Y_1)_\delta$ induced by the refinement $(\tilde\tubes_3, \tilde Y_3)_\rho$. Since $E_{\tubes_3}\subset E_{\tilde\tubes_3}$, \eqref{cubeEstimate} implies that $(\tubes_3,Y_3)_\delta$ satisfies \eqref{rhoNbhdProjBd}.
\end{proof}

Lemma \ref{coveringNumBdOneScale} says that when restricted to a $\tau$-cube and examined at scale $\rho$, the set $E_{\tubes'}$ looks like a union of about $(\tau/\rho)^{1-\sigma}$ parallel grains, whose normal directions all coincide and are given by the plane map $V$. The next lemma asserts that a similar statement continues to hold if we replace a $\tau$-cube with a (potentially much larger) $\tau\times\rho^{1/2}\times\rho^{1/2}$ rectangular prism, where the two ``long'' directions of the rectangular prism are orthogonal to the plane map $V$. To prove this lemma, we need to show that at scale $\rho$, the grains inside a $\tau$-cube are pieces of the grains inside a $\rho^{1/2}$-cube. 

\begin{lem}\label{coveringNumBdOneScaleInterval}
For all $\eps>0,$ there exists $\eta>0,\delta_0>0$ so that the following holds for all $\delta\in (0, \delta_0]$. Let $(\tubes,Y)_\delta$ be a $\eta$-extremal collection of tubes with $\delta^{-\eta}$-Lipschitz plane map $V$. Let $\rho\in [\delta^{1-\eps}, \delta^{\eps}]$ and let $\tau\in [\rho^{1-\eps}, \rho^{1/2}]$. 

Then there exists a $\eps$-extremal sub-collection $(\tubes',Y')_\delta$, so that the following holds. Let $Q$ be a $\rho^{1/2}$-cube and let $F = V(Q) \cdot ( E_{\tubes'} \cap Q)$. Then for each interval $J$ of length $\tau$, we have
\begin{equation}\label{rhoNbhdProjBdInterval}
\mathcal{E}_\rho(J \cap F) \leq \rho^{-\eps}(\tau/\rho)^{1-\sigma}.
\end{equation}
\end{lem}

\begin{proof}
For notational convenience, we will suppose that $\rho^{1/2}/\tau$ is an integer, and thus every $\tau$-cube is contained in a single $\rho^{1/2}$ cube. The case where $\rho^{1/2}/\tau$ is not an integer can be handled similarly. Let $\eps_1,\ldots,\eps_4>0$ be small constants to be chosen below. We will select $\eta$ very small compared to $\eps_1$, and $\eps_i$ very small compared to $\eps_{i+1}$. We will select $\eps_4$ small compared to $\eps$. 

Apply Proposition \ref{multiScaleStructureExtremal} to $(\tubes,Y)_\delta$ with $\eps_1$ in place of $\eps$, and $\rho$ as above. We obtain a refinement of $(\tubes,Y)_\delta$ (which we will continue to denote by $(\tubes,Y)_\delta$), and a $\eps_1$-extremal collection of $\rho$-tubes, $(\tilde\tubes_1, \tilde Y_1)_\rho$. In what follows, we will describe various sub-collections of $(\tilde\tubes, \tilde Y)_\rho$. Since $(\tilde\tubes, \tilde Y)_\rho$ is a balanced cover of $(\tubes,Y)_\delta$, each sub-collection of $(\tilde\tubes, \tilde Y)_\rho$ will induce an analogous sub-collection of $(\tubes,Y)_\delta$, of the same relative density.

Applying Lemma \ref{coveringNumBdOneScale} to $(\tilde\tubes_1,\tilde Y_1)_\rho$ with $\eps_2$ in place of $\eps$ and $\tau$ as above, and then again to the output with $\eps_3$ in place of $\eps$ and $\rho^{1/2}$ in place of $\tau$, we obtain a $\eps_3$-extremal sub-collection $(\tilde\tubes_2,\tilde Y_2)_\rho$ of $(\tilde\tubes_1,\tilde Y_1)_\rho$ with the following two properties: 
\begin{itemize}
\item For each $\rho^{1/2}$-cube $Q$, we have 
\begin{equation}
\label{rhoCoveringInsideRho12Cube}
\mathcal{E}_{\rho}\big( V(Q) \cdot (E_{\tilde\tubes_2}\cap Q) \big)\leq\rho^{-\eps_3}\rho^{\frac{\sigma-1}{2}}.
\end{equation}
\item For each $\rho^{1/2}$-cube $Q$, and each $\tau$-cube $W\subset Q$, we have 
\begin{equation}\label{rhoCoveringVqInW}
\mathcal{E}_{\rho}\big( V(Q) \cdot (E_{\tilde\tubes_2}\cap W) \big)\lesssim \delta^{-\eta}\rho^{-\eps_2}(\tau/\rho)^{\sigma-1}.
\end{equation}
\end{itemize}
Note that in \eqref{rhoCoveringVqInW} we have the unit vector $V(Q)$ rather than $V(W)$. This requires justification. By Lemma \ref{coveringNumBdOneScale} we get the estimate $\mathcal{E}_{\rho}\big( V(W) \cdot (E_{\tilde\tubes_2}\cap W) \big)\lesssim \rho^{-\eps_2}(\tau/\rho)^{\sigma-1}$. To obtain \eqref{rhoCoveringVqInW}, we bound
\begin{equation}\label{coarseningCoveringNumberArgument}
\begin{split}
\mathcal{E}_{\rho}\big( V(Q) \cdot (E_{\tilde\tubes_2}\cap W) \big)& \leq 2\delta^{-\eta}\mathcal{E}_{2\delta^{-\eta}\rho}\big( V(Q) \cdot (E_{\tilde\tubes_2}\cap W)\big)\\
&\lesssim \delta^{-\eta}\mathcal{E}_{\delta^{-\eta}\rho}\big( V(W) \cdot (E_{\tilde\tubes_2}\cap W)\big)\\
&\leq \delta^{-\eta}\mathcal{E}_{\rho}\big( V(W) \cdot (E_{\tilde\tubes_2}\cap W)\big)\\
&\leq \delta^{-\eta}\rho^{-\eps_2}(\tau/\rho)^{\sigma-1}.
\end{split}
\end{equation}
The key step is the second inequality, which used the fact that $\operatorname{diam}(W)\lesssim \rho^{1/2}$ and $\angle(V(Q),V(W))\leq 2\delta^{-\eta}\rho^{1/2}$; the latter statement follows from the fact that $W\subset Q$, $\operatorname{diam}(Q)\lesssim\rho^{1/2},$ and $V$ is $\delta^{-\eta}$-Lipschitz.

Apply Proposition \ref{multiScaleStructureExtremal} to $(\tilde\tubes_2,\tilde Y_2)_\rho$ with $\eps_4$ in place of $\eps$ and $\rho^{1/2}$ in place of $\rho$. We obtain a refinement $(\tilde\tubes_3,\tilde Y_3)_\rho$ of $(\tilde\tubes_2,\tilde Y_2)_\rho$, which is contained in a union of $\rho^{1/2}$-cubes of cardinality $\lesssim \rho^{-\eps_4}\rho^{\frac{\sigma-3}{2}}$.

After a further refinement, we can suppose that each $\rho^{1/2}$-cube $Q$ that intersects $E_{\tilde\tubes_3}$ satisfies $|E_{\tilde\tubes_3}\cap Q|\gtrapprox \rho^{2\eps_4}\rho^{\frac{\sigma+3}{2}}.$ 

The two properties described above continue to hold for $(\tilde\tubes_3,\tilde Y_3)_\rho$. Therefore, if $Q$ is a $\rho^{1/2}$-cube that intersects $E_{\tilde\tubes_3}$, then by \eqref{rhoCoveringInsideRho12Cube}, $E_{\tilde\tubes_3} \cap Q$ is contained in a union of at most $\rho^{-\eps_3}\rho^{\frac{\sigma-1}{2}}$ ``grains,'' which are sets of the form $Q \cap N_{\rho}(\pi_Q^{-1}(t))$, where $t\in\RR$. The union of these grains has volume $\lesssim \rho^{-\eps_3}\rho^{\frac{\sigma+3}{2}}$, while on the other hand $E_{\tilde\tubes_3}\cap Q$ has volume $\gtrapprox \rho^{2\eps_4}\rho^{\frac{\sigma+3}{2}}$. This means that after further refining $(\tilde\tubes_3,\tilde Y_3)_\rho$ to obtain a pair $(\tilde\tubes_4,\tilde Y_4)_\rho$, we can suppose that for each $\rho^{1/2}$-cube $Q$ and each $t \in V(p_Q) \cdot ( E_{\tilde\tubes_4}\cap Q)$, we have
\begin{equation}\label{grainsFull}
|E_{\tilde\tubes_4}\cap Q\ \cap\ \pi_Q^{-1}(N_{\rho}(t))| \gtrapprox \rho^{2\eps_4+\eps_3} \rho^2.
\end{equation}
(Note that $\rho^2$ is the volume of a $\rho\times\rho^{1/2}\times\rho^{1/2}$ grain).

Next, let $Q$ be one of the $\rho^{1/2}$ cubes described above. By Fubini, there must be a point $q\in Q$ so that the line $L = q + \RR V(p_Q)$ satisfies 
\begin{equation}\label{goodLine}
|L \cap E_{\tilde\tubes_4} \cap Q|\gtrapprox \rho^{2\eps_4}\rho^{\frac{\sigma+3}{2}} \gtrsim \rho^{\eps_2 + 2\eps_4}|E_{\tilde\tubes_4} \cap Q|.
\end{equation}
(The $|\cdot|$ on the LHS of \eqref{goodLine} denotes 1-dimensional Lebesgue measure, while the $|\cdot|$ on the RHS of \eqref{goodLine} denotes 3-dimensional Lebesgue measure).

For each cube $Q$, let $G_Q$ be the union of grains that intersect $L \cap E_{\tilde\tubes_4} \cap Q$, i.e.
\[
G_Q = Q \cap \Big[  \pi_Q^{-1}\Big(N_{\rho}\big(  \pi(L \cap E_{\tilde\tubes_4} \cap Q) \big)  \Big)    \Big].
\]
Then by \eqref{grainsFull} and \eqref{goodLine}, we have
\begin{equation}\label{muchMassPreserved}
|E_{\tilde\tubes_4} \cap G_Q| \gtrapprox \rho^{\eps_2+\eps_3 + 4\eps_4}|E_{\tilde\tubes_4} \cap Q|.
\end{equation}
Let $\tilde\tubes_5 =\tilde\tubes_4$, and define the shading $\tilde Y_5$ by $\tilde Y_5(\tilde T) = \tilde Y_4(\tilde T)\cap \bigcup_{Q}G_Q$. By \eqref{muchMassPreserved} and the bound $\eps_2+\eps_3 + 4\eps_4\leq 5\eps_4$, we have $\sum_{\tilde T\in\tilde\tubes_5}|\tilde Y_5(\tilde T)|\gtrapprox \rho^{5\eps_4}\sum_{\tilde T\in\tilde\tubes_4}|\tilde Y_4(\tilde T)|$. 

Unwinding definitions, we see that if $Q$ is a $\rho^{1/2}$-cube and if $J\subset\RR$ is an interval, then 
\[
\mathcal{E}_{\rho}\Big(J \cap \big(V(p_Q) \cdot (E_{\tilde\tubes_5}\cap Q)\big)\Big)\sim \mathcal{E}_{\rho} \Big( V(p_Q) \cdot ( L_J \cap E_{\tilde\tubes_3})\Big),
\]
where $L_J = L \cap \pi_Q^{-1}(J)$, and $L$ is the line from \eqref{goodLine}. In particular, if $J$ is an interval of length $\tau$, then $L_J$ is contained in a union of $O(1)$ $\tau$-cubes $W$, and hence
\[
\mathcal{E}_{\rho}\Big(J \cap \big(V(p_Q) \cdot (E_{\tilde\tubes_5}\cap Q)\big)\Big) \lesssim \sum_W \mathcal{E}_{\rho} \big( V(p) \cdot (E_{\tilde\tubes_5}\cap W)\big),
\]
where the sum is taken over the $O(1)$ $\tau$-cubes $W$ described above. Since the contribution from each of these $\tau$-cubes is $\leq \rho^{-\eps_2}(\tau/\rho)^{1-\sigma},$ we conclude that
\begin{equation}\label{boundScaleRhoOverInterval}
\mathcal{E}_{\rho}\Big(J \cap \big(V(p_Q) \cdot (E_{\tilde\tubes_5}\cap Q)\big)\Big) \lesssim\rho^{-\eps_2}(\tau/\rho)^{1-\sigma}.
\end{equation}

We claim that the sub-collection $(\tubes', Y')_\delta$ of $(\tubes, Y)_\delta$ induced by $(\tilde \tubes_5,\tilde Y_5)_\rho$ satisfies the conclusions of Lemma \ref{coveringNumBdOneScaleInterval}. First, if $\eps_1,\ldots,\eps_4$ are chosen appropriately, then $\sum_{T\in\tubes'}|Y'(T)|\geq\delta^{\eps}$, so $(\tubes', Y')_\delta$  is $\eps$-extremal. Next, let $Q$ be a $\rho^{1/2}$ cube, let $p\in Q$, and let $J$ be an interval of length $\tau$. Then since $E_{\tubes'}\subset E_{\tilde\tubes_5}$, we have
\begin{equation}\label{rhoControlsDelta}
\mathcal{E}_{\rho}\Big(J \cap \big( V(p) \cdot (E_{\tubes'}\cap Q) \big) \Big) 
\lesssim \delta^{-\eta}\mathcal{E}_{\rho}\Big(J \cap \big( V(p_Q) \cdot (E_{\tilde\tubes_5}\cap Q) \big) \Big),
\end{equation}
where we used the same argument as in \eqref{coarseningCoveringNumberArgument} to replace $V(p)$ by $V(p_Q)$, at the cost of a multiplicative factor of $\delta^{-\eta}$.
\eqref{rhoNbhdProjBdInterval} now follows from \eqref{boundScaleRhoOverInterval}.
\end{proof}

For a fixed $\rho$, the next lemma records what happens when we iterate Lemma \ref{coveringNumBdOneScaleInterval} for many different scales $\tau$.
\begin{lem}\label{ADSetOneScale}
For all $\eps>0,$ there exists $\eta>0,\delta_0>0$ so that the following holds for all $\delta\in (0, \delta_0]$. Let $(\tubes,Y)_\delta$ be a $\eta$-extremal collection of tubes with $\delta^{-\eta}$-Lipschitz plane map $V$. Let $\rho\in [\delta^{1-\eps}, \delta^{\eps}]$. 

Then there exists a $\eps$-extremal sub-collection $(\tubes',Y')_\delta$, so that the following holds. Let $Q$ be a $\rho^{1/2}$-cube. Then $V(Q) \cdot ( E_{\tubes'} \cap Q)$ is a $(\rho, 1-\sigma; \delta^{-\eps})_1$-ADset. 
\end{lem}
\begin{proof}
Let $N\geq 1/\eps$ be an integer. We will select a sequence of numbers $\eps/2 = \eps_N > \eps_{N-1} > \ldots >\eps_0>0$ so that the following holds.  Define $\eta = \eps_0$ and define $(\tubes_0,Y_0)_\delta = (\tubes,Y)_\delta$; this is a $\eps_0$-extremal collection of tubes with a $\delta^{-\eps_0}$-Lipschitz plane map. For each $j = 1,\ldots,N$, apply Lemma \ref{coveringNumBdOneScaleInterval} to $(\tubes_{j-1},Y_{j-1})_\delta$ with $\eps_j$ in place of $\eps$, $\rho$ as above, and $\tau = \tau_j = \rho^{(j+N)/2N}$. Denote the output by $(\tubes_j,Y_j)_\delta$; this pair is $\eps_j$-extremal and has a $\delta^{-\eps_j}$-Lipschitz plane map. We may apply Lemma \ref{coveringNumBdOneScaleInterval} provided that $\eps_{j-1}$ is chosen sufficiently small depending on $\eps_{j}$. Define $(\tubes',Y')_\delta = (\tubes_N, Y_N)_\delta$; this set is $\eps/2$-extremal. 

We will verify that $(\tubes',Y')_\delta$ satisfies the conclusion of Lemma \ref{ADSetOneScale}. Let $Q$ be a $\rho^{1/2}$-cube and let $J\subset\RR$ be an interval of length at least $\rho$. Shrinking $J$ if needed, we may suppose $J\subset V(Q)\cdot Q$, so in particular $|J|\leq 2\rho^{1/2}$. Thus there is an integer $0\leq j <N$ so that $\rho^{(j+N+1)/2N} < |J| \leq 2 \rho^{(j+N)/2N}$. Let $\tilde J$ be an interval of length $\tau_{j}=\rho^{(j+N)/2N}$ that contains $J$ (if $|J|>\rho^{1/2}$ then a slight modification is needed; we need to cover $J$ by a union of two intervals of length $\rho^{1/2}$).  Let $F = V(Q) \cdot ( E_{\tubes'} \cap Q)$. By Lemma \ref{coveringNumBdOneScaleInterval}, we have
\[
\mathcal{E}_\rho(J\cap F) \leq \mathcal{E}_\rho(\tilde J\cap F) \leq  \delta^{-\eps/2}(\rho^{(j+N)/2N}/\rho)^{1-\sigma}
\leq \delta^{-\eps/2}\rho^{-1/2N}(|J|/\rho)^{1-\sigma}\leq \delta^{-\eps}(|J|/\rho)^{1-\sigma}.\qedhere
\]
\end{proof}

The next lemma records what happens when we iterate Lemma \ref{ADSetOneScale} for many different scales $\rho$.
\begin{lem}\label{grainStructureLem}
For all $\eps>0$, there exists $\delta_0,\eta>0$ so that the following holds. Let $\delta\in (0,\delta_0]$ and let $(\tubes,Y)_\delta$ be an $\eta$-extremal collection of tubes with a $\delta^{-\eta}$-Lipschitz plane map $V$. 

Then there exists a $\eps$-extremal sub-collection $(\tubes',Y')_\delta$, so that 
\begin{equation}\label{grainStructureStatement}
V(p)\cdot\big(E_{\tubes'}\cap B(p, \rho^{1/2}\big)\ \textrm{is a}\ (\rho, 1-\sigma; \delta^{-\eps})_1\textrm{-ADset}\quad\textrm{for all}\ \rho\in[\delta,1],\ p\in \RR^3.
\end{equation}
\end{lem}
\begin{proof}
Let $N\geq 2/\eps$ be an integer. We will select a sequence of numbers $\eps/3 = \eps_N > \eps_{N-1} > \ldots >\eps_0>0$ so that the following holds.  Define $\eta = \eps_0$ and define $(\tubes_0,Y_0)_\delta = (\tubes,Y)_\delta$; this is a $\eps_0$-extremal collection of tubes with a $\delta^{-\eps_0}$-Lipschitz plane map. For each $j = 1,\ldots,N$, apply Lemma \ref{ADSetOneScale} to $(\tubes_{j-1},Y_{j-1})_\delta$ with $\eps_j$ in place of $\eps$ and $\rho=\rho_j=\delta^{j/N}$. Denote the output by $(\tubes_j,Y_j)_\delta$; this pair is $\eps_j$-extremal and has a $\delta^{-\eps_j}$-Lipschitz plane map. We may apply Lemma \ref{ADSetOneScale} provided that $\eps_{j-1}$ is chosen sufficiently small depending on $\eps_{j}$. Define $(\tubes',Y')_\delta = (\tubes_N, Y_N)_\delta$; this set is $\eps/3$-extremal.

Next, let $\rho\in [\delta,1]$, let $p\in\RR^3$, and let $J$ be an interval of length at least $\rho$. Our goal is to show that
\begin{equation}\label{RhoCoveringBdJ}
\mathcal{E}_{\rho}\Big(J\ \cap\  \Big[V(p)\cdot\big(E_{\tubes'}\cap B(p, \rho^{1/2}\big) \Big]  \Big)\leq\delta^{-\eps}(|J|/\rho)^{1-\sigma}.
\end{equation}
\medskip
\noindent {\bf Case 1: $\rho\geq\delta^{1/N}$}. In this case, \eqref{RhoCoveringBdJ} follows from the simple observation that $E_{\tubes'}\subset B(0,1)$, and $\mathcal{E}_{\rho}([-1,1])\leq 2\rho^{-1} \leq \delta^{-\eps}$. 

\medskip
\noindent {\bf Case 2: $\rho\in[\delta^{1-1/N}, \delta^{1/N})$}. In this case, there is an integer $1\leq j\leq N-2$ so that $\delta^{(j+1)/N}\leq \rho<\delta^{j/N}$.  $B(p, \rho^{1/2})$ can be covered by a union of $O(1)$ $\rho_{j}^{1/2}$-cubes. By Lemma \ref{ADSetOneScale}, for each such cube $Q$ we have
\begin{equation}\label{rhoJCoveringForCenterQ}
\mathcal{E}_{\rho_j}\Big(J\ \cap\  \big(V(Q)\cdot(E_{\tubes'}\cap Q) \big)\Big)\leq\delta^{-\eps/3}(|J|/\rho_j)^{1-\sigma}.
\end{equation}
Recall that $V(Q) = V(p_Q)$, where $p_Q$ is the center of $Q$. Since $Q$ intersects $B(p, \rho^{1/2})$, we have that $|p_Q-p|\lesssim \rho_j^{1/2}$. Since $V$ is $\delta^{-\eta}$-Lipschitz, this implies $\angle(V(Q), V(p))\lesssim \delta^{-\eta}\rho_j^{1/2}$. Finally, since $Q$ has diameter $\leq 2\rho_j^{1/2}$, we have
\begin{equation*}
\begin{split}
\mathcal{E}_{\delta^{-\eta}\rho_j}\Big(J\ \cap\  \big(V(p)\cdot(E_{\tubes'}\cap Q) \big)\Big)&\sim \mathcal{E}_{\delta^{-\eta}\rho_j}\Big(J\ \cap\  \big(V(Q)\cdot(E_{\tubes'}\cap Q) \big)\Big)\\
&\leq \delta^{-\eps/3}(|J|/\rho_j)^{1-\sigma}, 
\end{split}
\end{equation*}
where the final inequality used \eqref{rhoJCoveringForCenterQ}. 

Since \eqref{rhoJCoveringForCenterp} holds for each of the $O(1)$ cubes whose union covers $B(p, \rho^{1/2})$, we have 
\begin{equation}\label{rhoJCoveringForCenterp}
\mathcal{E}_{\delta^{-\eta}\rho_j}\Big(J\ \cap\  \Big[V(p)\cdot\big(E_{\tubes'}\cap B(p, \rho^{1/2}\big) \Big]  \Big)\lesssim\delta^{-\eps/3}(|J|/\rho_j)^{1-\sigma}.
\end{equation}
This is almost what we want, except we have bounded the $\delta^{-\eta}\rho_j$-covering number, rather than the $\rho$-covering number. But since each interval of length $\delta^{-\eta}\rho_j$ can be covered by $\leq \delta^{-1/N-\eta}$ intervals of length $\rho$, \eqref{rhoJCoveringForCenterp} implies
\[
\mathcal{E}_{\rho}\Big(J\ \cap\  \Big[V(p)\cdot\big(E_{\tubes'}\cap B(p, \rho^{1/2}\big) \Big]  \Big)\lesssim\delta^{-\eps/3-1/N-\eta}(|J|/\rho_j)^{1-\sigma}\leq \delta^{-\eps}(|J|/\rho_j)^{1-\sigma}.
\]
\eqref{RhoCoveringBdJ} now follows, provided $\delta_0$ and $\eta$ are chosen sufficiently small.

\medskip
\noindent {\bf Case 3: $\rho\in[\delta,\delta^{1-1/N})$}. This reduces to Case 2, using the observation that for every set $X\subset\RR$, $\mathcal{E}_\rho(X) \leq \delta^{-1/N}\mathcal{E}_{\delta^{1-1/N}}(X)$.

\end{proof}

\subsection{Proof of Proposition \ref{plainyGrainyProp}}
We are now ready to prove Proposition \ref{plainyGrainyProp}. In brief, we use the results from Sections \ref{findingPlaneMapSec} and \ref{lipRegularityPlaneMapSec} to find an extremal set of $\delta$-tubes $(\tubes,Y)_\delta$ with a plane map. Next, we use the results from Section \ref{grainSec} to show that this set of tubes arranges into grains of dimensions $\sqrt\delta\times\sqrt\delta\times\delta$. We select a tube $T_0$ whose $\sqrt\delta$ neighborhood covers about $\delta^{-1}$ tubes from $(\tubes,Y)_\delta$, and consider the unit rescaling (recall Definition \ref{unitRescalingDefn}) of this collection of tubes with respect to the $\sqrt\delta$ neighborhood $T_0$. This is a set of $\sqrt\delta$ tubes, which we will denote by $(\tilde T, \tilde Y)_{\sqrt\delta}$. 

In what follows, we define $\rho = \sqrt\delta$. Note that the $\sqrt\delta$ neighborhood of $T_0$ can be covered by about $\delta^{-1/2}$ balls of radius $\sqrt\delta$, and the restriction of $E_{\tubes}$ to each of these balls consists of a union of parallel grains of dimensions $\sqrt\delta\times\sqrt\delta\times\delta$. After re-scaling, the situation is as follows: The unit ball can be covered by $\rho^{-1}$ ``slabs'' (rectangular prisms) of dimensions $1\times 1\times\rho$, and the restriction of $E_{\tilde\tubes}$ to each slab consists of a union of parallel (re-scaled) grains of dimensions $1\times\rho\times\rho$; these will be called \emph{global grains}, and they will satisfy the properties described in Item \ref{plainyGrainyPropItem1} from the statement of Proposition \ref{plainyGrainyProp}.

Finally, we will apply the results from Section \ref{grainSec} to show that $(\tilde\tubes,\tilde Y)_{\rho}$ arranges into grains of dimensions $\sqrt\rho\times\sqrt\rho\times\rho$; these will be called \emph{local grains}, and they will satisfy the properties described in Item \ref{plainyGrainyPropItem2} from the statement of Proposition \ref{plainyGrainyProp}. We now turn to the details. 
\begin{proof}[Proof of Proposition \ref{plainyGrainyProp}]
Let $\eps_1,\ldots,\eps_4>0$ be small constants to be chosen below. We will select $\eps_i$ very small compared to $\eps_{i+1},$ and $\eps_4$ very small compared to $\eps$. 

\medskip
\noindent{\bf Step 1: Global grains.}\\
Define $\tau_0 = \delta_0^{\frac{2+\sigma}{\sigma}}$. By Lemma \ref{existenceOfLipschitzPlaneMap} followed by Lemma \ref{grainStructureLem}, there exists a $\eps_1$-extremal collection of tubes $(\tubes,Y)_\tau$ for some $\tau\in (0, \tau_0]$ that has a $\tau^{-\eps_1}$-Lipschitz plane map $V$, and satisfies \eqref{grainStructureStatement} with $\eps_1$ in place of $\eps$. 

Define $\delta = \tau^{\frac{\sigma}{2+\sigma}}$, so $\delta\in (0,\delta_0]$. Apply Proposition \ref{multiScaleStructureExtremal} to $(\tubes,Y)_\tau$, with $\tau$ in place of $\delta$, $\delta$ in place of $\rho$, and $\eps_2$ in place of $\eps$. We obtain a refinement $(\tubes_1,Y_1)_\tau$ of $(\tubes,Y)_\tau$, and a $\eps_2$-extremal collection of $\delta$-tubes $(\tilde\tubes_1,\tilde Y_1)_{\delta}$. 

Fix a choice of $\delta$-tube $\tilde T$; define $\tubes^\dag = \tubes_1[\tilde T]$, and let $Y^\dag$ be the restriction of $Y_1$ to $\tubes^\dag$. Since the unit re-scaling of $(\tubes^\dag,Y^\dag)_\tau$ relative to $\tilde T$ is $\eps_2$-extremal, we can find a sub-collection $(\tubes_1^\dag,Y_1^\dag)_\tau$ of $(\tubes^\dag,Y^\dag)_\tau$, with $\sum|Y_1^\dag(T)|\gtrapprox \tau^{\eps_2} \sum|Y^\dag(T)|,$ and a tube $T_0\in\tubes_1^\dag$ with the following property (P): for each $\delta$-cube $Q$ that intersects $E_{\tubes_1^\dag}$, we have $Y_1^\dag(T_0)\cap Q\neq\emptyset$. After enlarging $\tilde T$ slightly (we might have to replace $\delta$ by $2\delta$, but this is harmless), we may suppose that $T_0$ and $\tilde T$ have the same coaxial line.

Let $(\hat\tubes, \hat Y)_{\tau/\delta}$ be the unit re-scaling of $(\tubes_1^\dag,Y_1^\dag)_\tau$ relative to $\tilde T$. This re-scaling sends $T_0$ to a tube $\hat T_0$ whose coaxial line is the $z$-axis. Let 
\[
Z = \{z \in \delta\ZZ\colon (0,0,z) \in \hat Y(\hat T_0)\}.
\]
Since $T_0$ has Property (P), for each $p = (x,y,z) \in E_{\hat \tubes}$ there is $z_0\in Z$ with $|z-z_0|\lesssim\delta$.

Next, we will analyze what \eqref{grainStructureStatement} says about $(\hat\tubes, \hat Y)_{\tau/\delta}$. For each $z\in Z$, let $\phi^{-1}(z)$ be the inverse image of $z$ under the unit-rescaling associated to $\tilde T$. Then $\phi^{-1}(z)$ is contained in the coaxial line of $T_0$, and there is a point $p_z\in Y_1^\dag(T_0)$ with $|\phi^{-1}(z)-p_z|\lesssim \delta$. By \eqref{grainStructureStatement}, we have
\begin{equation}\label{dotted}
V(p_z)\cdot\big(E_{\tubes_1^\dag}\cap B(p_z, \delta)\big)\ \textrm{is a}\ (\delta^2, 1-\sigma; \tau^{-\eps_1})_1\textrm{-ADset}.
\end{equation}
But if we consider the image of $B(p_z, \delta)$ under $\phi$, \eqref{dotted} says that there is a unit vector $\hat V(z)\in S^2$ so that
\begin{equation}\label{rescaledDotted}
\hat V(z)\cdot\Big(E_{\hat\tubes}\cap \big(\RR^2\times [z, z+\delta)\big)\Big)\ \textrm{is a}\ (\delta, 1-\sigma; O(\tau^{-\eps_1}))_1\textrm{-ADset}.
\end{equation}

Since $|V(p_z)\cdot \dir(L_0)|\leq\tau$, we have $|\hat V(z)\cdot(0,0,1)|\leq \tau/\delta\leq 1/10$ for all $z\in Z$. In particular, since $V$ is $\tau^{-\eps_1}$-Lipschitz, we have that that $\hat V$ is also $\tau^{-\eps_1}$-Lipschitz, and thus for each interval $I\subset[-1,1]$ of length $\sim \tau^{\eps_1}$, we must have either $|\hat V(z)\cdot (1,0,0)|\geq \frac{1}{3}$ or $|\hat V(z)\cdot (0,1,0)|\geq \frac{1}{3}$ for all $z\in I\cap Z$. Interchanging the $x$ and $y$ axes if necessary, we may replace $Z$ with $Z\cap I$, and define $\hat\tubes_1=\hat\tubes$, $\hat Y_1(\hat T) = \hat Y(\hat T) \cap (\RR^2\times I)$. If we choose the interval $I$ appropriately, then $\sum_{\hat T\in\hat\tubes}|\hat Y_1(\hat T)|\gtrsim \tau^{\eps_1}\sum_{\hat T\in\hat\tubes}|\hat Y(\hat T)|$, and $|\hat V(z)\cdot (1,0,0)|\geq\frac{1}{3}$ for all $z\in Z$. In particular, there are $3\tau^{-2\eps_1}$-Lipschitz functions $f,g\colon Z\to\RR$ so that for all $z\in Z$, we have that $\hat V(z)$ is parallel to $(1, f(z), g(z))$ ($f$ is given by dividing the second component of $\hat V$ by the first, and $g$ is defined by dividing the third component of $\hat V$ by the first). Currently $f(z)$ is defined on $Z$, but by the Kirszbraun-Valentine Lipschitz extension theorem \cite{Kir}, we can extend $f$ to a $3\tau^{-2\eps_1}$-Lipschitz function with domain $[-1,1]$. By \eqref{rescaledDotted}, for each $p = (x,y,z)\in E_{\hat\tubes_1}$ we have
\begin{equation}\label{fDot}
(1, f(z), 0) \cdot\Big(E_{\hat\tubes}\cap \big(\RR^2\times [z, z+\delta)\big)\Big)\ \textrm{is a}\ (\delta, 1-\sigma; O(\tau^{-\eps_1}))_1\textrm{-ADset}.
\end{equation}

\medskip
\noindent{\bf Step 2: Local grains.}\\
Apply Lemma \ref{existenceOfWeakPlaneMap} to $(\hat\tubes_1, \hat Y_1)_{\tau/\delta}$ with $\tau/\delta$ in place of $\delta$ and $\eps_3$ in place of $\eps$; we obtain a refinement  $(\hat\tubes_2, \hat Y_2)_{\tau/\delta}$ and a map $U\colon E_{\hat\tubes_2}\to S^2$ that satisfies
\[
|U(p)\cdot \dir (\hat T)|\leq(\tau/\delta)^{\sigma/2}\quad\textrm{for all}\ p\in E_{\hat\tubes_2},\ \hat T\in\hat\tubes_2(p).
\]
Next, apply Lemma \ref{existenceOfPlaneMap} with $\tau/\delta$ in place of $\delta$, $\delta$ in place of $\rho$ (note that this is a valid choice, since $\delta = (\tau/\delta)^{\sigma/2}$), and $\eps_3$ in place of $\eps$. We obtain a $\eps_3$-extremal collection of $\delta$-tubes $(\tilde T, \tilde Y)_\delta$ that covers $(\hat\tubes_2, \hat Y_2)_{\tau/\delta}$ and a plane map $W\colon E_{\tilde \tubes}\to S^2$. Each $\delta$-cube $Q\subset E_{\tilde \tubes}$ intersects $E_{\hat\tubes_2}$, and thus by \eqref{fDot}, for each $z\in Z$ we have
\begin{equation}\label{intersectionScaleDelta}
(1, f(z), 0) \cdot\Big(E_{\tilde\tubes}\cap \big(\RR^2\times [z, z+\delta)\big)\Big)\ \textrm{is a}\ (\delta, 1-\sigma; O(\tau^{-\eps_1}))_1\textrm{-ADset}.
\end{equation}
Indeed, to verify \eqref{intersectionScaleDelta} let $\tilde X = \big(E_{\tilde\tubes}\cap \big(\RR^2\times [z, z+\delta)\big)\big)$ and let $\hat X = \big(E_{\hat\tubes}\cap \big(\RR^2\times [z, z+\delta)\big)\big)$. It suffices to show that each $w \in (1, f(z), 0) \cdot \tilde X$ is contained in the $O(\delta)$-neighborhood of $(1, f(z), 0) \cdot \hat X$. For such a $w$, there exists a $\delta$-cube $Q$ that intersects $\tilde X$ with $w \in (1, f(z), 0)\cap Q$. But since $Q$ is a $\delta$-cube, it must be contained in $\tilde X$. Since $Q\cap E_{\hat\tubes}\neq\emptyset,$ there must be a point $p \in Q\cap \hat X$. Thus $|(1, f(z), 0)\cdot w - (1,f(z),0)\cdot p|\lesssim |w-p|\lesssim\delta$. 

We will re-write \eqref{intersectionScaleDelta} as
\begin{equation}\label{intersectionScaleDeltaSlice}
(1, f(z_0), 0) \cdot\Big(E_{\tilde\tubes}\cap \{z=z_0\}\Big)\ \textrm{is a}\ (\delta, 1-\sigma; O(\tau^{-\eps_1}))_1\textrm{-ADset},
\end{equation}
for all $z_0\in Z$. In fact, \eqref{intersectionScaleDeltaSlice} holds for all $z\in[-1,1]$, since the set on the LHS of \eqref{intersectionScaleDeltaSlice} is empty if $z_0\in[-1,1]\backslash Z$.

To complete the proof, apply Lemma \ref{grainStructureLem} to $(\tilde T, \tilde Y)_\delta$ with $\eps_4$ in place of $\eps$. If $\eps_1,\ldots,\eps_4$ are chosen sufficiently small, then the resulting refinement of $(\tilde T, \tilde Y)_\delta$ almost satisfies the conclusions of Proposition \ref{plainyGrainyProp}, except the function $f$ and the plane map $W$ are not $1$-Lipschitz. This can be fixed via a mild rescaling, using Lemma \ref{rescalingPreservesExtremality}.
\end{proof}


\section{The global grains slope function is $C^2$}\label{globalGrainsIsC2Sec}
Our goal in this section is to strength the conclusions of Proposition \ref{plainyGrainyProp} by increasing the regularity of the global grains slope function. More precisely, we will prove the following.
\begin{prop}\label{C2GrainsProp}
For all $\eps,\delta_0>0$, there exists $\delta\in (0,\delta_0]$ and a $\eps$-extremal set of tubes $(\tubes,Y)_\delta$ with the following properties:
\begin{enumerate}
\item[$1^\prime.$] {\bf Global grains with $C^2$ slope function.}\\
There is a function $f\colon [-1,1]\to\RR$ with $\Vert f\Vert_{C^2}\leq 1$ so that $\big( E_{\tubes} \cap \{z = z_0\}\big) \cdot (1, f(z_0), 0)$ is a $(\delta, 1-\sigma, \delta^{-\eps})_1$-ADset for each $z_0\in [-1,1]$.

\item[2.] {\bf Local grains with Lipschitz plane map.} \\
$(\tubes,Y)_{\delta}$ has a 1-Lipschitz plane map $V$. For all $\rho\in[\delta,1]$ and all $p\in \RR^3$, $V(p) \cdot \big( B(p,\rho^{1/2}) \cap E_\tubes\big)$ is a $(\rho, 1-\sigma, \delta^{-\eps})_1$-ADset. 
\end{enumerate}
\end{prop}

We refer the reader to Section \ref{proofSketchSection}, Step 2 for an overview of the main ideas used to prove Proposition \ref{C2GrainsProp}. Our main tool for proving Proposition \ref{C2GrainsProp} will be the following new result from projection theory.

\begin{thm}\label{SWThm}
For all $\eps>0$, there exists $\eta,\delta_0>0$ so that the following holds for all $\delta\in (0, \delta_0]$. Let $F,G\subset[0,1]^2$ be $(\delta, 1, \delta^{-\eta})$-ADsets. Then at least one of the following must hold.
\begin{enumerate}[label=(\Alph*)]
	\item\label{SWThmItemA} There exist two lines $\ell$  and  $\ell^\perp$ (where $\ell^\perp$ passes through the origin and is orthogonal to $\ell$)  whose $\delta$-neighborhoods cover a significant fraction of $F$ and $G$ respectively. In particular, 
\begin{equation}\label{largeIntersectionRec}
\mathcal{E}_\delta(N_\delta(\ell)\cap F)\geq \delta^{\eps-1} \ \textrm{and}\ \ 
\mathcal{E}_\delta(N_\delta(\ell^\perp)\cap G)\geq \delta^{\eps-1}.
\end{equation}

\item\label{SWThmItemB} Let $\mathcal{H}\subset F\times G\times G$ be any set satisfying $\mathcal{E}_{\delta}(\mathcal{H})\geq \delta^{\eta-3}$. 
Then there exists $\rho\geq\delta$ and an interval $I$ of length at least $\delta^{-\eta}\rho$ such that 
\begin{equation}\label{bigRCoveringNumber}
\mathcal{E}_\rho\big(I \cap \{a\cdot(b_1-b_2)\colon (a,b_1,b_2)\in \mathcal{H}\} \big) \geq \big(|I|/\rho\big)^{1-\eps}.
\end{equation}
\end{enumerate}
\end{thm}

\begin{rem}
Note that the statement of Theorem \ref{SWThm} makes sense over $\CC$ (with all exponents doubled), but the theorem is false. To construct a counter-example, let $\Omega \subset\CC^2$ be a neighbourhood of the origin and let $F=G=\{(z,\bar z)\colon z\in\CC\}\cap \Omega$; these are $(\delta, 2, O(1))$-sets. Let $\mathcal{H} = F\times G\times G$; then $\mathcal{E}_\delta(\mathcal{H})\gtrsim\delta^{-6}$. Conclusion (A) fails because the $\delta$-neighbourhood of every line $\ell=\{(a+bt, c+dt)\colon t\in \CC\}$ intersects only a small fraction of $F$; indeed, since $(b,d)\neq (0,0)$, we have $\mathcal{E}_{\delta} (N_{\delta}(\ell) \cap F) \leq \delta^{-1}$, which is much smaller than $\delta^{-2}$.  Conclusion (B) fails because if $\omega=(z_0, \bar z_0)\in F$ and $\zeta_1=(\bar z_1, z_1),\zeta_2=(\bar z_2, z_2)\in G$, then
\[
\omega\cdot(\zeta_1-\zeta_2) = z_0(z_1-z_2) + \bar z_0(\bar z_1-\bar z_2) = z_0(z_1-z_2) + \overline{z_0(z_1-z_2)} = \overline{\omega\cdot(\zeta_1-\zeta_2)},
\]
and thus
\[
\{\omega\cdot(\zeta_1-\zeta_2) \colon (\omega,\zeta_1,\zeta_2)\in\mathcal H\} \subset \{z\in\CC\colon \operatorname{Im}(z) = 0\}.
\]
The latter set fails to satisfy the analogue of \eqref{bigRCoveringNumber} (with exponent $2-\eps$ in place of $1-\eps$) for any choice of scale $\rho$.
\end{rem}

The proof of Theorem \ref{SWThm} uses ideas and results that are different in character from the rest of the proof of Theorem \ref{mainThm}, so we will defer the proof to Section \ref{SWProofSec}. In Section \ref{SWProofSec} we will prove a mild generalization of Theorem \ref{SWThm} in which the non-concentration condition on $F$ and $G$ is weakened. 


\subsection{The slope function is locally linear}
In this section we will use Theorem \ref{SWThm} to show that if $(\tubes,Y)_\delta$ is an extremal collection of tubes that satisfies the conclusions of Proposition \ref{plainyGrainyProp}, then the restriction of the global grains slope function to an interval of length $\rho^{1/2}$ looks linear at scale $\rho$. 

To begin, we will study the interplay between the ``global grains'' and ``local grains'' structures from Proposition \ref{plainyGrainyProp}. If we select a scale $\rho\in[\delta,\delta^{\eps}]$ and restrict $E_{\tubes}$ to the $\rho^{1/2}$ neighborhood of a global grain, then we obtain a set $E\subset E_{\tubes}$ that has the following two structures: The global grains property (Item \ref{plainyGrainyPropItem1} from Proposition \ref{plainyGrainyProp}) says that each ``horizontal'' (i.e. $z=z_0$) slice of $E$ is a union of $1\times\rho$ rectangles pointing in the direction $(1, f(z_0), 0)$, and the spacing of these rectangles forms a ADset in the sense of Definition \ref{ADSetDefn}. The local grains property (Item \ref{plainyGrainyPropItem2} from Proposition \ref{plainyGrainyProp}) says that the restriction of $E$ to a ball of radius $\rho^{1/2}$ is a union of $\rho^{1/2}\times\rho^{1/2}\times\rho$ grains, whose normal direction is given by the plane map $V$. We will use these two structures to construct sets $F$, $G$, and $\mathcal{H}\subset F\times G\times G$ that satisfy the hypotheses of Theorem \ref{SWThm}. We will show that Conclusion \ref{SWThmItemB} from Theorem \ref{SWThm} would contradict the ADset spacing condition from the global grains structure, and hence the sets $F$ and $G$ must satisfy Conclusion \ref{SWThmItemA}. This in turn implies that the slope function $f$ is linear at scale $\rho$ over intervals of length $\rho^{1/2}$. More precisely, the slope function $f$ agrees with a linear function at scale $\rho$ on a substantial portion of most intervals of length $\rho^{1/2}$. The precise version is stated in Lemma \ref{locallyLinearLem} below. The precise hypotheses of this lemma are slightly technical; they are formulated in this fashion to allow the lemma to be applied at many scales and locations. 

In Lemma \ref{locallyLinearEverywhereLem}, we will apply Lemma \ref{locallyLinearLem} at many scales and locations to the (re-scaled) slices of an extremal Kakeya set. We will conclude that for many different scales $\rho$, the slope function $f$ agrees with a linear function at scale $\rho$ on a substantial portion of most intervals of length $\rho^{1/2}$. This in turn implies that on a substantial portion of its domain, the slope function $f$ agrees with a function that has controlled $C^2$-norm. 

\begin{lem}\label{locallyLinearLem}
For all $\eps>0$, there exists $\eta,\rho_0>0$ so that the following holds for all $\rho\in (0, \rho_0]$. Let $I\subset[-1,1]$ be an interval of length $\rho^{1/2}$ with left endpoint $z^*$, and let $Z\subset I$. Let $f\colon I \to[0,1]$ be 1-Lipschitz. Let $E\subset[-1,1]^2\times Z$ be a union of $\rho$-cubes contained in the $\rho^{1/2}$-neighborhood of a line pointing in direction $(-f(z^*), 1, 0)$, with $|E|\geq\rho^{1+\sigma/2+\eta}$. Let $V\colon E\to S^2$ be 1-Lipschitz, with $|V(p)\cdot(0,0,1)|\leq 1/2$ for all $p\in E$. Suppose that for each $z_0\in Z$, 
\begin{equation}\label{fewGlobalGrainValues}
\big(E \cap \{z=z_0\}\big)\cdot\big(1, f(z_0),0\big)\quad\textrm{is a}\ (\rho, 1-\sigma,\rho^{-\eta})_1\textrm{-ADset},
\end{equation}
and for each $p\in E$, 
\begin{equation}\label{fewLocalGrainValues}
\big|\big(E\cap B(p, \rho^{1/2})\big)\cdot V(p)\big|\leq\rho^{\frac{1+\sigma}{2}-\eta}.
\end{equation}

\medskip

Then there exists a set $Z_{\operatorname{lin}}\subset Z$ with 
$
|Z_{\operatorname{lin}}|\geq \rho^{1/2+\eps},
$ 
and a linear function  $L\colon [z^*, z^*+\rho^{1/2}]\to [0,1]$ so that $|f(z)-L(z)|\leq\rho$ for all $z\in Z_{\operatorname{lin}}$. 
\end{lem}
\begin{proof}
Without loss of generality we may suppose that $I = [0, \rho^{1/2}]$ (i.e.~$z^*=0$), $f(0)=0$ (and hence $f(Z)\subset[-\rho^{1/2},\rho^{1/2}]$), and that $E$ is contained in the $\rho^{1/2}$-neighborhood of the $y$-axis.  If not, then we can apply a (harmless) translation so that $E$ is contained in the $\rho^{1/2}$-neighborhood of a line passing through the origin, and then we can apply a rotation around the $z$-axis so that $f(0) = 0$. 

Note that the image $E'$ of $E$ under this rotation is no longer a union of $\rho$-cubes (since by definition, $\rho$-cubes are aligned with the standard grid). However, we can replace $E'$ by the union $E''$ of $\rho$ cubes that intersect $E'$, and \eqref{fewGlobalGrainValues} and \eqref{fewLocalGrainValues} will continue to hold for $E''$ (the quantities $\rho^{-\eta}$ and $\rho^{\frac{1+\sigma}{2}-\eta}$ in \eqref{fewGlobalGrainValues} and \eqref{fewLocalGrainValues} are respectively weakened to $100\rho^{-\eta}$ and $100\rho^{\frac{1+\sigma}{2}-\eta},$ but this is harmless). Note that our rotation fixes each plane $\{z=z_0\}$, so after the initial shift $Z\to Z-z^*$, the set $Z$ remains unchanged by this procedure. 

\medskip

\noindent {\bf Step 1: $V$ has large $x$-coordinate}\\
Divide $[0,\rho^{1/2}]\times [0,1]\times [0,\rho^{1/2}]$ into $\rho^{-1/2}$ cubes. After pigeonholing and refining $E$, we may suppose that each $\rho^{1/2}$-cube $Q$ that intersects $E$ satisfies $|Q\cap E|\gtrsim \rho^{3/2+\sigma/2+\eta}$, and these $\rho^{1/2}$ cubes are $\rho^{1/2}$ separated. For each such cube $Q$, select a point $p_Q\in Q\cap E$; since the cubes are $\rho^{1/2}$ separated, so are the points $p_Q$. For each such cube $Q$, define $\pi_Q(q) = q\cdot V(p_Q)$. 

Comparing with \eqref{fewLocalGrainValues}, we may further refine $E$ so that each local grain is almost full: if $Q$ is $\rho^{1/2}$-cube that intersects $E,$ and $w \in V(p_Q)\cdot (Q \cap E)$, then 
\begin{equation}\label{localGrainAlmostFull}
|E\cap Q \cap N_{\rho}(\pi_Q^{-1}(w))|\gtrsim \rho^{2+2\eta}.
\end{equation} 
Fix such a cube $Q$. By \eqref{localGrainAlmostFull} and Fubini, there is some $z_1\in [0,\rho^{1/2}]$ so that  
\begin{equation}\label{rectangleFilledOut}
\Big|\{z=z_1\} \cap \big( E\cap Q \cap N_{\rho}\big(\pi_Q^{-1}(w)\big)\big)\Big|\gtrsim \rho^{3/2+2\eta}
\end{equation}
(note that $|\cdot|$ in \eqref{localGrainAlmostFull} refers to 3-dimensional Lebesgue measure, while $|\cdot|$ in \eqref{rectangleFilledOut} refers to 2-dimensional Lebesgue measure).

Since $|V(p_Q)\cdot (0,0,1)|\leq 1/2$, the set on the LHS of \eqref{rectangleFilledOut} is contained in a rectangle $R$ (in the $\{z=z_1\}$ plane) of dimensions roughly $\rho^{1/2}\times\rho$. Thus $R\cdot (1, f(z_1))$ is an interval $I$, whose length depends on the angle between the vectors $V(p_Q)$ and $(1, f(z_1), 0)$. More precisely, we have 
\begin{equation}\label{lengthOfI}
|I|\sim \rho^{1/2} |V(p_Q)\cdot (-f(z_1), 1, 0)|.
\end{equation}
By \eqref{rectangleFilledOut}, we have
\begin{equation}\label{largeDot}
\Big|(1, f(z_1), 0)\cdot \Big(\{z=z_1\} \cap \big( E\cap Q \cap N_{\rho}(\pi_Q^{-1}(w))\big)\Big)\Big|\gtrsim \rho^{2\eta}|I|,
\end{equation}
but the set on the LHS of \eqref{largeDot} is contained in $I \cap \big(E \cap \{z=z_1\}\big)\cdot\big(1, f(z_1),0\big)$, and \eqref{fewGlobalGrainValues} says that 
\[
\Big|I \cap \big(E \cap \{z=z_1\}\big)\cdot\big(1, f(z_1),0\big)\Big|\leq \rho^{1-\sigma-\eta}|I|^{\sigma}.
\]
Re-arranging, we conclude that $|I|\leq\rho^{1-3\eta/(1-\sigma)}$, and thus by \eqref{lengthOfI}, 
\begin{equation}\label{dotWithFz10}
|V(p_Q)\cdot (-f(z_1), 1, 0)|\lesssim  \rho^{1/2-3\eta/(1-\sigma)}. 
\end{equation}
Since $f(z_1)\in [0, \rho^{1/2}]$ and $V(p_Q)$ is a unit vector, if $\eta$ is selected sufficiently small then \eqref{dotWithFz10} implies that $|V(p_Q)\cdot (0, 1, 0)|\leq 1/10.$ Since $V(p_Q)$ is a unit vector and $|V(p_Q)\cdot(0,0,1)|\leq 1/2$, we conclude that $|V(p_Q)\cdot (1,0,0)|\geq 1/4$.

\medskip

\noindent {\bf Step 2: Finding the function $g$.}\\
Let $E_{\rho^{1/2}}$ be the union of $\rho^{1/2}$-cubes that intersect $E$, and let $J\subset[0,1]$ be the projection of $E_{\rho^{1/2}}$ to the $y$-axis. Each $y\in J$ is in the image of precisely one $\rho^{1/2}$-cube that intersects $E$; we will denote this cube by $Q(y)$. The set $J$ is a union of $\rho^{1/2}$ intervals, which are $\rho^{1/2}$ separated. 

Define the functions $g(y), h(y)\colon J\to\RR$ as follows: For each $y\in J$, $(1, h(y), g(y))$ is parallel to $V(p_Q)$, where $Q=Q(y)$. The functions $g$ and $h$ are constant on each $\rho^{1/2}$ interval of $J$. Since $V$ is 1-Lipschitz and $|V(p_Q)\cdot(1,0,0)|\geq 1/4$ for each cube $Q$, we have that $g$ and $h$ are $4$-Lipschitz. Abusing notation, we will replace $g$ and $h$ by their $4$-Lipschitz extension to $[0,1]$; the image of $g$ and $h$ are contained in $[-5,5]$.

If $p = (x,y,z)\in E$ is contained in a $\rho^{1/2}$-cube $Q$, then $|V(p)-V(p_Q)|\leq \sqrt 3\rho^{1/2}$, and hence by \eqref{fewLocalGrainValues} we have
\[
\big|\big(E\cap B(p, \rho^{1/2})\big)\cdot (1, h(y), g(y))\big|\lesssim\rho^{\frac{1+\sigma}{2}-\eta}.
\]
In particular, for each $y_0\in [0,1]$ we have
\begin{equation}\label{ySliceLocalGrains}
\Big|\big(E \cap \{y=y_0\}\big)\cdot\big(1, 0,g(y_0)\big)\Big|\lesssim \rho^{\frac{1+\sigma}{2}-\eta}.
\end{equation}

We remark that all of the arguments thus far also apply to the Heisenberg group example \eqref{defnOfH} (this would involve working in $\CC^3$ rather than $\RR^3$, but all of the arguments thus far could be translated to that setting). For the Heisenberg group example, we would have that the function $g\colon \CC\to\CC$ is given by $g(z) = \bar z$.

\medskip

\noindent{\bf Step 3: Paths of length 3.}\\
After refining the set $E$, we can select a set $Z_{\operatorname{popular}}\subset\rho\ZZ\cap[0, \rho^{1/2}]$ so that each $\rho$-cube $Q\subset E$ has center with $z$-coordinate in $Z_{\operatorname{popular}}$, and each $z\in Z_{\operatorname{popular}}$ is the center of  $\gtrsim \rho^{-3/2+\sigma/2+\eta}$ such cubes. After this refinement $E$ is still a union of $\rho$-cubes, and its volume has decreased by at most a factor of $1/2$. 

After pigeonholing and applying a shift of the form $(x,y,z)\mapsto(x,y-y_0,z)$ for some $y_0\in [0, \rho^{1/2}]$, we can find a set $E_1\subset E$ that is a union of $\rho$-cubes, with $|E_1|\geq \rho^{1/2+2\eta}|E|$, so that each $\rho$-cube $Q\subset E_1$ has center $p_Q$ whose $y$-coordinate is contained in $\rho^{1/2}\ZZ$. Indeed, for each local grain, we choose  a $\rho\times \rho\times \rho^{1/2}$-segment orthogonal to  the $y$-direction and include the intersection of this segment with $E$ in $E_1$.   Note that \eqref{fewGlobalGrainValues} and \eqref{ySliceLocalGrains} continue to hold with $E_1$ in place of $E$.

We say that two $\rho$-cubes $Q$ and $Q'$ contained in $E_1$ are \emph{in the same global grain} if their centers $p_Q$ and $p_{Q'}$ have the same $z$-coordinate, which we will denote by $z_Q$, and 
\[
\big| p_Q\cdot(1, f(z_Q), 0) - p_{Q'}\cdot(1, f(z_Q), 0)\big| \leq \rho.
\]

We say that two $\rho$-cubes $Q$ and $Q'$ contained in $E_1$ are \emph{in the same local grain} if their centers $p_Q$ and $p_{Q'}$ have the same $y$-coordinate, which we will denote by $y_Q$, and

\[
\big| p_Q\cdot(1, 0, g(y_Q)) - p_{Q'}\cdot(1, 0, g(y_{Q}))\big| \leq \rho.
\]

Recall that $E_1$ is a union of at least $\rho^{-3/2+\sigma/2+3\eta}$ $\rho$-cubes. For each such cube $Q$, the $y$-coordinate of its center is an element of $\rho^{1/2}\ZZ\cap[0,1]$, and thus takes one of $\rho^{-1/2}$ values. Once the $y$-coordinate has been fixed, by \eqref{fewLocalGrainValues} there are at most $\rho^{-1/2+\sigma/2-\eta}$ choices for $p_Q\cdot \big(1, 0, g(y)\big)$. Thus we can partition $E_1$ into $O(\rho^{-1 + \sigma/2 - \eta})$ sets, each of which are unions of $\rho$-cubes, so that any two $\rho$-cubes from a common set are in the same local grain. Since $E_1$ contains at least $\rho^{-3/2+\sigma/2+3\eta}$ $\rho$-cubes, by Cauchy-Schwarz there are $\gtrsim \big(\rho^{-3/2+\sigma/2+3\eta}\big)^2/\rho^{-1 + \sigma/2 - \eta} = \rho^{-2+\sigma/2+7\eta}$ pairs $(Q,Q')$, where $Q$ and $Q'$ are $\rho$-cubes in $E$ that are contained in a common local grain. 

Consider the map $s\colon (Q,Q')\mapsto (z_Q, z_{Q'}, r_{Q'})$, where $z_Q$ is the $z$-coordinate of the center of $Q$ (similarly for $z_{Q'}$), and $r_{Q'}$ is the number of the form $\rho\ZZ$ closest to $p_{Q'}\cdot(1, f(z_{Q'}), 0)\big)$. There are $\leq \rho^{-1/2}$ possible values for $z_Q$ and similarly for $z_{Q'}$. By \eqref{fewGlobalGrainValues}, for each such value of $z$ there are at most $\rho^{-1/2 +\sigma/2-\eta}$ values for $r_{Q'}$. Thus $s$ takes at most $\rho^{-3/2 +\sigma/2 - \eta}$ values.

Thus by Cauchy-Schwarz, there are $\gtrsim \big(\rho^{-2+\sigma/2+7\eta}\big)^2 / \rho^{-3/2 +\sigma/2 - \eta} = \rho^{-5/2 + \sigma/2 + 15\eta}$ pairs $((Q,Q'),(\tilde Q, \tilde Q'))$ where $Q$ and $Q'$ (resp.~$\tilde Q$ and $\tilde Q'$) are contain in a common local grain, and $s(Q,Q') = s(\tilde Q,\tilde Q')$.

Let $\mathcal{Q}$ denote the set of quadruples $Q_1,\ldots,Q_4$ of $\rho$-cubes contained in $E_1$ that have the following relations:
\begin{itemize}
\item $Q_1$ and $Q_2$ are in the same local grain.
\item $Q_2$ and $Q_3$ are in the same global grain.
\item $Q_3$ and $Q_4$ are in the same local grain.
\item The midpoints of $Q_4$ and $Q_1$ have the same $z$-coordinate.
\end{itemize}
After renaming the tuple $((Q,Q'),(\tilde Q, \tilde Q'))$ as $(Q_1,\ldots,Q_4)$, we conclude that 
\begin{equation}\label{numberOfQuadruples}
\#\mathcal{Q} \gtrsim \rho^{-5/2+\sigma/2+15\eta}.
\end{equation}

\medskip

\noindent{\bf Step 4: A skew transform.}\\
In this step we will perform a minor coordinate change to reduce to the case that $f(0) = g(0) = 0$. By \eqref{numberOfQuadruples}, pigeonholing, and \eqref{fewGlobalGrainValues}, we can select a choice $(z_0,w)$ so that there are at least $\gtrsim \rho^{-3/2+8\eta}$ quadruples $(Q_1,\ldots,Q_4)$ where the center of $Q_1$ has $z$-coordinate $z_0$, and 
\begin{equation}\label{wGrain}
\big| p_{Q_1}\cdot(1, f(z_0), 0)-w|\leq\rho.
\end{equation}
Denote this set of quadruples by $\mathcal{Q}'$.
Define
\begin{equation}\label{defnA}
A =  \big(E \cap \{z=z_0\}\big)\cdot\big(1, f(z_0),0\big) - w.
\end{equation}
By \eqref{fewGlobalGrainValues}, $A$ is a $(\rho, 1-\sigma,\rho^{-\eta})_1$-ADset.

Define
\[
\phi(x,y,z) = \big(x + f(z_0)y + g(0)(z-z_0),\ y,\ z-z_0\big).
\]

Define $f^\dag(z) = f(z) - f(z_0)$ and $g^\dag(y) = g(y)-g(0)$. Then $f^{\dag}(0) = g^{\dag}(0)=0$. Let $E^{\dag}= \phi(E)$. Then for each $z_1\in [-z_0,\ \rho^{1/2}-z_0]$, we have 
\begin{equation}\label{EDagZSlice}
\big(E^\dag \cap \{z=z_1\}\big)\cdot\big(1, f^\dag(z_1), 0\big) = \big(E \cap \{z=z_1+z_0\}\big)\cdot\big(1, f(z_1), 0\big)+g(0)(z_1-z_0),
\end{equation}
so by \eqref{fewGlobalGrainValues}, the LHS of \eqref{EDagZSlice} is a $(\rho, 1-\sigma,\rho^{-\eta})_1$-ADset. Similarly, for each $y_1\in [0,1]$,
\begin{equation}\label{EDagYSlice}
\big(E^\dag \cap \{y=y_1\}\big)\cdot\big(1, 0, g^\dag(y_1)\big) = \big(E \cap \{y=y_1\}\big)\cdot\big(1, 0, g(y_1)\big)+f(z_0)y_1-g(0)z_0,
\end{equation}
so by \eqref{fewLocalGrainValues}, the LHS of \eqref{EDagYSlice} has $\rho$-covering number $\lesssim \rho^{\frac{-1+\sigma}{2}-\eta}$.

\medskip

\noindent{\bf Step 5: Few dot products.}\\
Let $\mathcal{W}^\dag$ be the images of the centers of the cubes in $\mathcal{Q}'$ under $\phi$. Note that for each quadruple $(p_1,\ldots,p_4)\in\mathcal{W}^\dag$, $p_1$ and $p_4$ have $z$-coordinate 0, and by \eqref{wGrain}, the $x$-value of $p_1$ (which we will denote by $x_1$) satisfies $|x_1 - w| \leq\rho$.

For each $(p_1,\ldots,p_4)\in\mathcal{W}^\dag$, we will associate a tuple $(y_1, z_2, y_3)$ as follows: $y_1$ will be the $y$-coordinate of $p_1$; $z_2$ will be the $z$-coordinate of $p_2;$ and $y_3$ will be the $y$-coordinate of $p_3$. We will show that
\begin{equation}\label{dotInA}
z_2\big(g^{\dag}(y_1)-g^{\dag}(y_3)\big) + \big(y_1-y_3\big)f^{\dag}(z_2) \in N_{4\rho}(A).
\end{equation}
To establish \eqref{dotInA}, we will argue as follows. For each index $i=1,\ldots,4$, write $p_i = (x_i, y_i, z_i)$. First, we have $|x_1-w|\leq\rho$ and $z_1=0$. Since $Q_1$ and $Q_2$ are in the same local grain, $y_2=y_1$, and $x_1,x_2,z_2$ are related by $|(x_1,y_1,0)\cdot(1, 0, g^{\dag}(y_1)) - (x_2,y_2,z_2)\cdot(1, 0, g^{\dag}(y_1))|\leq \rho$ i.e.
\[
\big|x_2 - \big(w -z_2g^{\dag}(y_1)\big)\big|\leq 2\rho.
\]
Next, since $Q_2$ and $Q_3$ are in the same global grain, $z_2=z_3$, and $x_2,x_3,y_2,y_3$ are related by $|(x_2,y_2,z_2)\cdot(1, f^{\dag}(z_2),0)-(x_3,y_3,z_2)\cdot(1, f^{\dag}(z_2),0)|\leq\rho,$ i.e. (since $y_1=y_2$)
\[
\Big| x_3 -\big(w -z_2g^{\dag}(y_1) + (y_1-y_3)f^{\dag}(z_2)\big)\Big| \leq 3\rho.
\]
Finally, we have $z_4=0$. since $Q_3$ and $Q_4$ are in the same local grain, $y_4=y_3$, and  $x_3,x_4,z_3$ are related by $|(x_3,y_3,z_3)\cdot(1, 0, g^{\dag}(y_3)) - (x_4,y_4,0)\cdot(1, 0, g^{\dag}(y_3))|\leq \rho$ i.e. (since $z_2=z_3$)
\begin{equation}\label{x4Explicit}
\Big| x_4 -\Big(w - z_2\big(g^{\dag}(y_1)-g^{\dag}(y_3)\big) + \big(y_1-y_3\big)f^{\dag}(z_2)\Big)\Big| \leq 4\rho.
\end{equation}
Since $p_4=\phi(p_Q)$ for some $\rho$-cube $Q\subset E$ whose center has $z$-coordinate $z_0$, after unwinding definitions we see that $x_4\in \big(E \cap \{z=z_0\}\big)\cdot\big(1, f(z_0),0\big)$ and hence \eqref{dotInA} follows from \eqref{x4Explicit}.

Define
\begin{equation*}
\begin{split}
F &= \{\big(\rho^{-1/2}z,\ \rho^{-1/2}f^{\dag}(z)\big)\colon z-z_0 \in Z_{\operatorname{popular}}\},\\
G &= \{\big(y,\ g^{\dag}(y)\big)\colon y\in [0,1]\cap \rho^{1/2}\ZZ\}.
\end{split}
\end{equation*}
Since $f^{\dag}$ is 1-Lipschitz and $g^{\dag}$ is 4-Lipschitz, $F$ and $G$ are $(\rho^{1/2}, 1, C)_{2}$-ADsets for some absolute constant $C$. 
 Let $\mathcal{H}$ be the set of $6$-tuples 
\[
\big(y_1,\ g^{\dag}(y_1),\ \rho^{-1/2}z_2,\ \rho^{-1/2} f^{\dag}(z_2),\ y_3,\ g^{\dag}(y_3) \big), 
\]
where $(y_1,z_2,y_3)$ corresponds to a triple from $\mathcal{W}^{\dag}$. By construction $\mathcal{H}\subset F\times G\times G$, and 
\[
\#\mathcal{H}=\#\mathcal{W}^{\dag}=\#\mathcal{Q}' \gtrsim (\rho^{1/2})^{16\eta-3}.
\]

Let $\tilde A = N_{4\rho^{1/2}}\big(\{\rho^{-1/2}a\colon a\in A\}\big)$. Then $\tilde A\subset[0,1]$ is a $(\rho^{1/2}, 1-\sigma,8\rho^{-\eta})_1$-ADset, and \eqref{dotInA} says that
\[
\{p\cdot (q-q')\colon (p,q,q')\in\mathcal{H}\} \subset \tilde A.
\] 

Let $\eta_1$ be the output of Theorem \ref{SWThm} with $\eps$ as above and $\rho^{1/2}$ in place of $\rho$. If $\eta>0$ is selected sufficiently small, then $F,G,$ and $\mathcal{H}$ satisfy the hypotheses of Theorem \ref{SWThm}. If $\eta>0$ is selected sufficiently small depending on $\eta_1$, then the conclusion of Item \ref{SWThmItemB} cannot hold, since \eqref{bigRCoveringNumber} violates the property that $\tilde A$ is a $(\rho^{1/2}, 1-\sigma,8\rho^{-\eta})_1$-ADset. 

Hence Conclusion \ref{SWThmItemA} from Theorem \ref{SWThm} must hold. Let $\ell$ and $\ell^\perp$ be the orthogonal lines from Conclusion \ref{SWThmItemA} of Theorem \ref{SWThm}. Then by \eqref{largeIntersectionRec}, we have
\[
\mathcal{E}_{\rho}\big( \rho^{1/2}(F\cap N_{\rho^{1/2}}(\ell) ) \big)\geq(\rho^{1/2})^{\eps-1}. 
\]

Define $Z_{\operatorname{lin}}$ to be the union of $\rho$-intervals that intersect the set
\[
 \big\{z \in Z_{\operatorname{popular}}\colon \big(\rho^{-1/2}(z+z_0), \rho^{-1/2}f^{\dag}(z+z_0)\big)\in F\cap N_{\rho^{1/2}}(\ell) \big\}.
\]

We have $\rho^{-1}|Z_{\operatorname{lin}}|\geq (\rho^{1/2})^{\eps}\#Z_{\operatorname{popular}}$, and thus
\[
|E\cap ([-1,1]^2\times Z_{\operatorname{lin}})|\gtrapprox_\rho \rho^{\eps/2}\rho^3 \rho^{-3/2+\sigma/2+3\eta} = \rho^{\eps/2+3\eta+3/2+\sigma/2}.
\]
But by \eqref{fewGlobalGrainValues}, this implies that $|Z_{\operatorname{lin}}|\gtrapprox \rho^{1/2+\eps/2+4\eta}$. We will select $\rho_0>0$ sufficiently small so that this quantity is at least $\rho^{1/2+\eps}$. 

To conclude the proof, the graph of the linear function $L$ from the statement of Lemma \ref{locallyLinearLem} is the obvious re-scaling of $\ell$ by a factor of $\rho^{1/2}$. Then for each $z\in Z_{\operatorname{lin}},$ we have $|f(z)-L(z)|\leq\rho$. 
\end{proof}

The next lemma records what happens when we slice an extremal Kakeya set $(\tubes,Y)_\delta$ into horizontal slabs of thickness $\rho^{1/2}$, select the $\rho^{1/2}$ neighborhood of a global grain from each of these slices, and apply Lemma \ref{locallyLinearLem}: we conclude that the slope function $f$ is linear at scale $\rho$ on each interval of length $\rho^{1/2}$. In what follows, we define $\pi_z\colon\RR^3\to\RR$ to be the projection to the third coordinate.

\begin{lem}\label{locallyLinearEverywhereLem}
For all $\eps>0$, there exists $\eta>0,\delta_0>0$ so that the following holds for all $\delta\in (0, \delta_0]$. Let $(\tubes,Y)_\delta$ be a $\eta$-extremal collection of tubes, and let $\rho\in [\delta^{1-\eps}, \delta^{\eps}]$. Suppose that there are functions $f\colon [-1,1]\to \RR$ and $V\colon E_{\tubes}\to S^2$ that satisfy Items \ref{plainyGrainyPropItem1} and \ref{plainyGrainyPropItem2} from the statement of Proposition \ref{plainyGrainyProp}, with $\eta$ in place of $\eps$. 

Then there is a $\eps$-extremal sub-collection $(\tubes',Y')_\delta$ and a function $L\colon [-1, 1]\to\RR$ with the following properties:
\begin{itemize}
\item $L$ is 2-Lipschitz, and linear on each interval of the form $[n\rho^{1/2}, (n+1)\rho^{1/2}]$, $n\in\ZZ$. 
\item For each $z\in \pi_z(E_{\tubes'})$, we have $|f(z) - L(z)|\leq\rho$. 
\end{itemize}
\end{lem}
\begin{proof}
Let $\eps_1,\eps_2,\eps_3>0$ be small constants to be chosen below. We will select $\eta$ very small compared to $\eps_1$, $\eps_1$ very small compared to $\eps_2$, $\eps_2$ very small compared to $\eps_3$, and $\eps_3$ very small compared to $\eps$.

Apply Proposition \ref{multiScaleStructureExtremal} to $(\tubes,Y)_\delta$ with $\eps_1$ in place of $\eps$, and $\rho$ as above. We obtain a refinement of $(\tubes,Y)_\delta$, and a $\eps_1$-extremal collection of $\rho$-tubes, $(\tilde\tubes, \tilde Y)_\rho$. In what follows, we will describe various sub-collections of $(\tilde\tubes, \tilde Y)_\rho$. Since $(\tilde\tubes, \tilde Y)_\rho$ is a balanced cover of $(\tubes,Y)_\delta$, each of these sub-collections of $(\tilde\tubes, \tilde Y)_\rho$ will induce an analogous sub-collection of $(\tubes,Y)_\delta$, of the same relative density. 

Apply Proposition \ref{multiScaleStructureExtremal} to $(\tilde\tubes, \tilde Y)_\rho$ with $\eps_2$ in place of $\eps_1$, and $\rho^{1/2}$ in place of $\rho$. We obtain a refinement of $(\tilde\tubes,\tilde Y)_\rho$, and a $\eps_2$-extremal collection of $\rho^{1/2}$-tubes, $(\dbtilde\tubes, \dbtilde Y)_{\rho^{1/2}}$. After further refining each of these three collections, we may suppose that $(\tubes,Y)_\delta$ has constant multiplicity; $(\tilde\tubes,\tilde Y)_\rho$ is a balanced cover of $(\tubes,Y)_\delta$, and $(\dbtilde\tubes, \dbtilde Y)_{\rho^{1/2}}$ is a balanced cover of $(\tilde\tubes,\tilde Y)_\rho$  (and hence also a balanced cover of $(\tubes,Y)_\delta$).

Recall that each shading $\dbtilde Y(\dbtilde \tubes)$ is a union of $\rho^{1/2}$-cubes. After further refining $(\dbtilde\tubes, \dbtilde Y)_{\rho^{1/2}}$ (which again induces a refinement of $(\tubes,Y)_\delta$ and $(\tilde\tubes,\tilde Y)_\rho$), we may suppose that $E_{\dbtilde\tubes}$ is contained in a union of $\rho^{1/2}$-separated ``slabs'' of the form $\RR^2\times [n\rho^{1/2}, (n+1)\rho^{1/2}]$, $n\in \ZZ$, and for each such slab $S$ that meets $E_{\dbtilde\tubes}$, we have $|S\cap E_{\tilde\tubes}|\gtrapprox_{\rho} \rho^{1/2+\eps_1}|E_{\tilde\tubes}|$ and hence $|S\cap E_{\tubes}|\gtrapprox_{\delta} \rho^{1/2+\eps_1}|E_{\tubes}|\gtrapprox \rho^{1/2+\eps_1}\delta^{\sigma+\eta}$.

We shall consider each slab $S$ in turn. Let $Z_S\subset [n\rho^{1/2}, (n+1)\rho^{1/2}]$ be the set of values $z_0$ for which $|E_{\tubes}\cap \{z=z_0\}|\gtrapprox_{\delta}\rho^{2\eps_1}\delta^{\sigma + \eta}$. By Fubini, $|Z_S|\geq \rho^{1/2+\eps_1}\delta^{\eta}$. Select a choice of $z_0\in Z_S$. Since $(\tubes,Y)_\delta$ satisfies Item \ref{plainyGrainyPropItem1}  from the statement of Proposition \ref{plainyGrainyProp}, we have $|(E_\tubes\cap\{z=z_0\})\cdot(1, f(z_0),0)|\leq \delta^{\sigma-\eta}$, i.e.~$(E_\tubes\cap\{z=z_0\})\cdot(1, f(z_0),0)$ is contained in a union of at most $\delta^{\sigma-\eta-1}$  $1\times\delta$ rectangles (in the $\{z=z_0\}$ plane), each of which have normal direction $(1, f(z_0), 0)$. By pigeonholing, we can select one of these rectangles that intersects $E_{\tubes}\cap\{z=z_0\}$ in a set of (two-dimensional) Lebesgue measure $\gtrapprox_\delta \delta^{1+2\eta}\rho^{2\eps_1}$, i.e. there exists $w\in \RR$ so that
\begin{equation}\label{dottedPs}
|\{ p \in E_{\tubes}\cap \{z=z_0\}\colon |p\cdot (1, f(z_0), 0)-w|\leq\delta\}| \gtrapprox_\delta \delta^{1+2\eta}\rho^{2\eps_1},
\end{equation}
where $|\cdot|$ denotes 2-dimensional Lebesgue measure. The set on the LHS of \eqref{dottedPs} is contained in the $\delta$-neighborhood of a unit line segment in the $\{z=z_0\}$ plane, that points in the direction $(1, f(z_0),0)$; denote this line segment by $\ell$. Then $\ell$ must intersect $\gtrapprox \delta^{2\eta}\rho^{-1/2+2\eps_1}$ $\rho^{1/2}$-cubes from $E_{\dbtilde\tubes}$, and hence
\[
\big|N_{\rho^{1/2}}(\ell) \cap E_{\tilde\tubes}\big| \gtrapprox \big(\delta^{2\eta}\rho^{-1/2+\eps_1}\big)\big(\rho^{\eps_1}\rho^{\frac{\sigma-3}{2}}\big)
=\delta^{2\eta}\rho^{\sigma/2 - 1 + 2\eps_1}.
\]

We can now verify that if $\eta$ and $\eps_1$ are chosen sufficiently small compared to $\eps_3$, then the set $N_{\rho^{1/2}}(L) \cap E_{\tilde\tubes},$ along with the function $f$ (restricted to $[n\rho^{1/2}, (n+1)\rho^{1/2}]$) and the function $V$ (restricted to $N_{\rho^{1/2}}(\ell) \cap E_{\tubes}$) satisfy the hypotheses of Lemma \ref{locallyLinearLem}, with $\eps_3$ in place of $\eps$. 

Applying Lemma \ref{locallyLinearLem}, we obtain a set $Z_{\operatorname{lin},S}\subset [n\rho^{1/2}, (n+1)\rho)$ and a linear function $L_S$ with domain $[n\rho^{1/2}, (n+1)\rho)$, so that $|f(z)-L_S(z)|\leq\rho$ for all $z\in Z_{\operatorname{lin},S}$, and $|Z_{\operatorname{lin},S}|\geq\rho^{1/2+\eps_3}$.  In particular, $Z_{\operatorname{lin}, S}$ contains two points $z_1, z_2$ that are at least $\rho^{1/2+\eps_3}$-separated. Since $f$ is 1-Lipschitz, $L_S$ is linear, and $|f(z_i)-L_S(z_i)|\leq\rho$ for $i=1,2$, we conclude that $L_S$ has slope at most 2.

Since $Z_{\operatorname{lin},S}\subset Z_S$ and $|Z_{\operatorname{lin},S}|\geq\rho^{1/2+\eps_3}$, we have 
\begin{equation}
\big| E_{\tubes} \cap (\RR^2\times Z_{\operatorname{lin},S}) \big| \geq \rho^{1/2+\eps_1+\eps_3}|E_\tubes|.
\end{equation}
Define $Z_{\operatorname{lin}}=\bigcup_S Z_{\operatorname{lin},S}$, where the union is taken over all horizontal slabs $S$ that meet $E_{\dbtilde\tubes}$, a union of $\rho^{1/2}$-cubes. Then 
\[
\big| E_{\tubes} \cap (\RR^2\times Z_{\operatorname{lin}}) \big| \gtrapprox \rho^{\eps_1+\eps_2+\eps_3}|E_\tubes|,
\]
and hence if we define $\tubes'=\tubes$ and $Y'(T) = Y(T) \cap (\RR^2\times Z_{\operatorname{lin}})$, then (since we replaced $(\tubes,Y)_\delta$ with a constant multiplicity refinement at an earlier step) we have $\sum_{T\in\tubes'}|Y'(T)|\gtrapprox_\delta \delta^\eta \rho^{\eps_1+\eps_2+\eps_3}$. If $\eta,\eps_1,\eps_2,\eps_3$ are chosen sufficiently small (depending on $\eps$), then $(\tubes',Y')$ is $\eps$-extremal. 

Finally, define $L$ by $L(z) = L_S(z)$ if $z\in I_S$ (here we write $S = [-1,1]^2\times I_S$). Note that for each slab $S$, there are points $p=(x,y,z), p'=(x',y',z')\in E_{\tubes'}$ with $|z-z'|\geq 4\rho$. Then $\big| \frac{L(z)-L(z')}{z-z'} \big| \leq \big|\frac{f(z) - f(z')}{z-z'}|+\frac{2\rho}{|z-z'|}\leq 3/2$, i.e. $L$ is $(3/2)$-Lipschitz on $I_S$. Finally, we can extend $L$ to $[-1,1]$ by defining $L$ to be piecewise linear; since the intervals $I_S$ on which $L$ was initially defined are $\rho^{1/2}$-separated (and on each such interval, $L$ agrees with  $f$, which is a $1$-Lipschitz function up to accuracy $\rho$), we have that the extension of $L$ to $[-1,1]$ is $2$-Lipschitz. 
\end{proof}

The next lemma records what happens when we iterate Lemma \ref{locallyLinearEverywhereLem} for many different scales. The proof is straightforward and is omitted. 

\begin{lem}\label{nestedRectanglesCoverGraphF}
For all $\eps>0$ and $N\geq 1$, there exists $\eta>0,\delta_0>0$ so that the following holds for all $\delta\in (0, \delta_0]$. Let $(\tubes,Y)_\delta$ be a $\eta$-extremal collection of tubes. Suppose that there are functions $f\colon [-1,1]\to \RR$ and $V\colon E_{\tubes}\to S^2$ that satisfy Items \ref{plainyGrainyPropItem1} and \ref{plainyGrainyPropItem2} from the statement of Proposition \ref{plainyGrainyProp}, with $\eta$ in place of $\eps$. 

For each $j=1,\ldots,N$, define $\rho_j = \delta^{j/N}$. Then there is a $\eps$-extremal sub-collection $(\tubes',Y')_\delta$ and functions $L_1,\ldots,L_N$ 
with the following properties:
\begin{enumerate}[label=(\alph*)]
\item\label{nestedRectanglesCoverGraphFConclusionA} Each $L_j$ is 2-Lipschitz, and linear on each interval of the form $[n\rho_j^{1/2}, (n+1)\rho_j^{1/2}]$, $n\in\ZZ$. 
%
%
\item\label{nestedRectanglesCoverGraphFConclusionD} For each $z\in \pi_z(E_{\tubes'})$ and each index $j$, we have $|f(z) - L_j(z)|\leq\rho_j$. 
\end{enumerate}
\end{lem}

Conclusion \ref{nestedRectanglesCoverGraphFConclusionD} enforces some consistency between the functions $L_j$ for different values of $j$---but we must be a bit careful. Let $I$ be an interval of the form $[n\rho_j^{1/2}, (n+1)\rho_j^{1/2}]$. Note that Conclusion \ref{nestedRectanglesCoverGraphFConclusionD} does \emph{not} imply that $|L_j(z) - L_{j-1}(z)|\leq \rho_{j}+\rho_{j-1}$ for all $z\in I$. However, if $z_1,z_2\in I\cap \pi_z(E_{\tubes'})$, then $|L_j(z) - L_{j-1}(z)|\leq \rho_{j}+\rho_{j-1}$ for all $z\in [z_1,z_2]$. Luckily, after a harmless refinement of $(\tubes',Y')_\delta$ by a factor of $4^{-N}$, can ensure that the output $(\tubes',Y')_\delta$ satisfies the following popularity condition: each interval $I=[n\rho_j^{1/2}, (n+1)\rho_j^{1/2}]$ that intersects $\pi_z(E_{\tubes})$ has large intersection with $I$, see \eqref{popularInsideIj} below.  We can also enforce a separation condition that will be useful later:

\emph{
For each index $j$, let $\mathcal{I}_j$ be the set of intervals of the form $I=[n\rho_j^{1/2}, (n+1)\rho_j^{1/2}]$ that intersect $\pi_z(E_{\tubes'})$. Then in addition to Items (a) and (b) above, for each $j=1,\ldots,N$ we also have:
\begin{enumerate}
\item[(c)] Each interval $I\in \mathcal{I}_j$  satisfies 
\begin{equation}\label{popularInsideIj}
|I\cap\pi_z(E_{\tubes'})|\geq 4^{-N} \delta^{\eps}|I|.
\end{equation}
\item[(d)] $\mathcal{I}_j$ is $\rho_j^{1/2}$-separated, i.e.~no two adjacent intervals are contained in $\mathcal{I}_j$.
\end{enumerate}
}

Indeed, for each $j=N,\ldots,1$, we obtain Items (c) and (d) at scale $\rho_j$ as follows. For Item (c), we discard all intervals $I$ for which \eqref{popularInsideIj} fails, and refine the shading $Y'$ accordingly. For Item (d), we either discard all intervals $[n\rho_j^{1/2}, (n+1)\rho_j^{1/2}]$ for which $n$ is odd, or all intervals for which $n$ is even. This step does not destroy Items (c) and (d) for previous (larger) values of $j$ (corresponding to smaller values of $\rho_j$).

Thus, for each $I\in\mathcal{I}_j$ there is an interval $I'\subset I$ of length $|I'|\geq 4^{-N}\delta^{\eps}\rho_j^{1/2}$ so that $\pi_z(E_{\tubes'})\cap I\subset I'$, and 
\begin{equation}\label{LjInsideLjm1}
|L_j(z) - L_{j-1}(z)|\leq \rho_{j}+\rho_{j-1}\quad\textrm{for all}\ z\in I'.
\end{equation}
It will be helpful to restate Lemma \ref{nestedRectanglesCoverGraphF} by replacing the piecewise-linear functions $L_j$ with ``vertical trapezoids,'' which are defined as follows:

\begin{defn}
A \emph{vertical trapezoid} is a set of the form
\[
R = \{(z,t)\in J\times \RR\colon |t - L(z)|\leq s\},
\]
where $J\subset[-1,1]$ is an interval and $L\colon J\to[-1,1]$ is linear. We define $w=|J|$ to be the \emph{length} of the trapezoid, and $s$ to be the \emph{height} of the trapezoid; we will refer to $R$ as a $w\times s$ trapezoid, and we define $\operatorname{slope}(R)$ to be the slope of $L$.
\end{defn}
With this definition, \eqref{LjInsideLjm1} can be rephrased as a set-containment statement about vertical trapezoids. This allows us to restate Lemma \ref{nestedRectanglesCoverGraphF} (with the additional Conclusions (c) and (d))  as follows:

\begin{cor}\label{nestedRectanglesCoverGraphFCor}
For all $\eps>0$ and $N\geq 1$, there exists $\eta>0,\delta_0>0$ so that the following holds for all $\delta\in (0, \delta_0]$. Let $(\tubes,Y)_\delta$ be a $\eta$-extremal collection of tubes. Suppose that there are functions $f\colon [-1,1]\to \RR$ and $V\colon E_{\tubes}\to S^2$ that satisfy Items \ref{plainyGrainyPropItem1} and \ref{plainyGrainyPropItem2} from the statement of Proposition \ref{plainyGrainyProp}, with $\eta$ in place of $\eps$. 

For each $j=1,\ldots,N$, define $\rho_j = \delta^{j/N}$. Then there is a $\eps$-extremal sub-collection $(\tubes',Y')_\delta$ and sets $\mathcal{R}_1,\ldots,\mathcal{R}_N$ of vertical trapezoids, with the following properties:

\begin{itemize}
\item Each trapezoid has slope in $[-2,2]$.
\item For each $j$, each trapezoid $R\in\mathcal{R}_j$ has height $\rho_j$ and length between $\rho_j^{1/2+\eps}$ and $\rho_j^{1/2}$. 
\item For each $j$, the projections of the trapezoids in $\mathcal{R}_j$ to the $z$-coordinate are $\rho_j^{1/2}$-separated.
\item For each $j>1$ and each $R\in\mathcal{R}_j$, there is a (necessarily unique) parent trapezoid $\tilde R\in\mathcal{R}_{j-1}$ that contains $R$.
\item Let $D = \pi_z(E_{\tubes'})$. Then $\operatorname{graph}(f|_D) \subset \bigcup_{R\in\mathcal{R}_N}R$. 
\end{itemize}
\end{cor}


\subsection{Locally linear at many scales implies $C^2$}
In this section, we will prove that the conclusions of Corollary \ref{nestedRectanglesCoverGraphFCor} imply that $f$ must have small $C^2$ norm (or more precisely, the restriction $f|_D$ agrees to accuracy $\delta$ with a function that has small $C^2$ norm). The precise statement is as follows.
\begin{prop}\label{rectVsC2Lem}
Let $\eta>0$, let $N\geq 2$, and let $\delta>0$. For each $j=1,\ldots,N$, define $\rho_j = \delta^{j/N}$ and let $\mathcal{R}_j$ be a set of vertical trapezoids with the following properties:
\begin{itemize}
\item Each trapezoid has slope in $[-2,2]$.
\item For each $j$, each trapezoid $R\in\mathcal{R}_j$ has height $\rho_j$ and length between $\rho_j^{1/2+\eta}$ and $\rho_j^{1/2}$. 
\item For each $j$, the projections of the trapezoids in $\mathcal{R}_j$ to the $z$-coordinate are $\rho_j^{1/2}$-separated.
\item For each $j>1$ and each $R\in\mathcal{R}_j$, there is a (necessarily unique) parent trapezoid $\tilde R\in\mathcal{R}_{j-1}$ that contains $R$.
\end{itemize}
Let $D\subset[-1,1]$ and let $f\colon D\to\RR$ with $\operatorname{graph}(f)\subset\bigcup_{R\in\mathcal{R}_N}R$.
 
Then there exists a function $g \colon [-1,1]\to\RR$ so that
\begin{enumerate}
\item $|f(x)-g(x)|\leq\delta$ for all $x\in D$.
\item $|g''(x)|\lesssim \delta^{-1/N - 2\eta}$ for all $x\in [0,1]$.
\item If $x\in\pi(R)$ for some $R\in\mathcal{R}_j$, then $|g'(x)-\operatorname{slope}(R)|\lesssim\rho_j /\rho_{j+1}^{1/2+\eta}$.
\end{enumerate} 
\end{prop}

\noindent The next lemma describes the basic building block of the function $g$.
\begin{lem}\label{functionAdaptedToLineSegment}
Let $x_1<x_2$ and let $L$ be the line segment from $(x_1,y_1)$ to $(x_2,y_2)$. Let $a>0$. Then there is a function $G\colon\RR\to\RR$ with the following properties.
\begin{itemize}
\item $L\subset \operatorname{graph}(G)$.
\item $G(x)=0$ for $x\in \RR\backslash[x_1-a, x_2+a]$.
\item Let $h = |y_1| + |y_2|$ and let $s = \frac{y_2-y_1}{x_2-x_1}$. Then $|G'(x)|\lesssim h/a + |s|$ and $|G''(x)|\lesssim h/a^2 + |s|/a$ for all $x\in\RR$. 
\end{itemize}
\end{lem}

\begin{proof}
Let $\phi\colon \RR \to [0,1]$ be a smooth (weakly) monotone increasing function with $\phi(x)=0$ for $0\leq x\leq 1/3$ and $\phi(x)=1$ for $2/3\leq x\leq 1$. We will choose $\phi$ so that $|\phi'(x)|\leq 10$ and $|\phi''(x)|\leq 10$ for $0\leq x\leq 1$. Define $\phi_a(x) =\phi(x/a)$. Then $\phi_a(0)=0$, $\phi_a(a)=1$, $(\phi_a)_+'(0) = (\phi_a)_-'(a)=0$, and $(\phi_a)_+''(0)=(\phi_a)_-''(a)=0$. We also have $|\phi_a'(x)|\lesssim a^{-1}$ and $|\phi_a''(x)|\lesssim a^{-2}$ for $x\in [0,a].$

Define  $\psi_a(x) = (x-a)\phi_a(x)$. Then $\psi_a(0)=\psi_a(a)=0$, $(\psi_a)_+'(0)=0$, $(\psi_a)_{-}'(a)=1$, and $(\psi_a)_+''(0)=(\psi_a)_-''(a)=0$. We also have $|\psi_a'(x)|\lesssim 1$ and $|\psi_a''(x)|\lesssim a^{-1}$ for $x\in [0,a]$. 

Let 
\[
G(x) = \left\{\begin{array}{ll}
y_1\phi_a(x-x_1+a)+s\psi_a(x-x_1+a),& \quad x\leq x_1,\\
s(x-x_1) +y_1, & \quad x_1<x<x_2, \\
y_2\phi_a(a+x_2-x)-s\psi_a(a+x_2-x),& \quad x\geq x_2.
 \end{array}\right. \qedhere
\]

\end{proof}
The function $g$ from \ref{rectVsC2Lem} will be a sum of the functions $G$ from Lemma \ref{functionAdaptedToLineSegment}, with one function $G$ for each trapezoid $R$. We turn to the details.
\begin{proof}[Proof of Lemma \ref{rectVsC2Lem}]
 For each $R\in\mathcal{R}_1$, let $\tilde L(R)$ be the line segment coaxial with $R$. Next, for $2\leq k\leq N$ and for each $R\in\mathcal{R}_k$, let $\tilde s(R)=\operatorname{slope}(R)-\operatorname{slope}(\tilde R)$, where $\tilde R\in\mathcal{R}_{k-1}$ is the parent of $R$. Let $(x_1,y_1)$, $(x_2,y_2)$ be the endpoints of the line segment coaxial with $R$. Let $\tilde y_1,\tilde y_2$ be selected so that $(x_1,\tilde y_1), (x_2,\tilde y_2)$ lie on the line segment coaxial with $\tilde R$. Let $\tilde L(R)$ be the line segment from $(x_1, y_1-\tilde y_1)$ to $(x_2, y_2-\tilde y_2)$. Finally, let $G_R(x)$ be the function defined in Lemma \ref{functionAdaptedToLineSegment} with $L = \tilde L(R)$ and $a = \rho_k^{1/2}/2$.

Define
\[
g(x) = \sum_{k=1}^N \sum_{R\in\mathcal{R}_k}G_R(x).
\]

We will verify that this function satisfies the three conclusions of Lemma \ref{rectVsC2Lem}. Conclusion 1 follows immediately from the definition of $g$ and the definition of the functions $\{G_R\}$. 

For conclusion 2, observe that for each $1\leq k\leq N$ and each $R\in\mathcal{R}_k$, we have $x_2-x_1 \gtrsim \rho_k^{1/2+\eta}$ and $|y_1-\tilde y_1|, |y_2-\tilde y_2|\leq \rho_{k-1}$. We also have $|\operatorname{slope}(\tilde L(R))|\leq \rho_{k-1}/\rho_k^{1/2+\eta}$. This implies that
\begin{equation}\label{boundGpp}
|G_R''(x)|\lesssim \frac{\rho_{k-1}}{\rho_k^{1+2\eta}}\leq \delta^{-\frac{1}{N}-2\eta \frac{k}{N}}\quad\textrm{for all}\ x\in\RR.
\end{equation}
We have 
\begin{equation}\label{computegpp}
|g''(x)| \leq \sum_{k=1}^N \Big| \sum_{R\in\mathcal{R}_k} G_R''(x)\Big|.
\end{equation}

Recall that for each index $k=1,\ldots,N$, the projections of the trapezoids in $\mathcal{R}_k$ to the $z$-axis are $\rho_k^{1/2}$-separated. This implies that the functions $\{G_R\colon R\in\mathcal{R}_k\}$ have disjoint support. Thus for each $k=1,\ldots,N$, at most one function from $\{G_R: R\in\mathcal{R}_k\}$ can contribute to the inner sum in \eqref{computegpp}. Thus by \eqref{boundGpp} we have 
\[
|g''(x)|\lesssim \sum_{k=1}^N \delta^{-\frac{1}{N}-2\eta \frac{k}{N}} \lesssim \delta^{-\frac{1}{N}-2\eta}.
\]

For conclusion 3, observe that if $R\in\mathcal{R}_k$, then
\begin{equation}\label{boundGp}
|G_R'(x)|\lesssim  \frac{\rho_{k-1}}{\rho_k^{1/2+\eta}}\quad\textrm{for all}\ x\in\RR.
\end{equation}
Next, if $R_0\in\mathcal{R}_k$ and $x\in \pi(R_0)$, then
\[
\operatorname{slope}(R_0) = \sum_{j=1}^k \sum_{R\in\mathcal{R}_j}G'_R(x).
\]
Thus
\[
|g'(x) - \operatorname{slope}(R_0)| \leq \sum_{j=k+1}^N \Big|\sum_{R\in\mathcal{R}_j}G'_R(x)\Big|.
\]
Again, for each index $k=1,\ldots,N$ the functions $\{G_R\colon R\in\mathcal{R}_k\}$ have disjoint support. Thus for each $j=k+1,\ldots,N$, at most one function $G_R,\ R\in\mathcal{R}_j$ can contribute to the above sum. Thus by \eqref{boundGp} we have 
\[
|g'(x) - \operatorname{slope}(R_0)| \leq \sum_{j=k+1}^N \frac{\rho_{j-1}}{\rho_j^{1/2+\eta}}\lesssim \frac{\rho_{k}}{\rho_{k+1}^{1/2+\eta}}.\qedhere
\]
\end{proof}


\subsection{Proof of Proposition \ref{C2GrainsProp}}
We will prove Proposition \ref{C2GrainsProp} by combining Corollary \ref{nestedRectanglesCoverGraphFCor} and Proposition \ref{rectVsC2Lem}. The details are as follows. 

\begin{proof}[Proof of Proposition \ref{C2GrainsProp}]
Let $\delta_0>0'$ and  $\eps_1,\eps_2>0$ be small constants to be determined below. We will select $\eps_1$ very small compared to $\eps_2$, and $\eps_2$ very small compared to $\eps$. Let $\delta\in (0,\delta_0']$ and let $(\tubes,Y)_\delta$ be a collection of $\delta$-tubes that satisfies the conclusions of Proposition \ref{plainyGrainyProp}, with $\eps_1$ in place of $\eps$. Apply Corollary \ref{nestedRectanglesCoverGraphFCor} to $(\tubes,Y)_\delta$ with $\eps_2$ in place of $\eps$ and $N = 1/\eps_2$, and let $(\tubes_1,Y_1)_\delta$ and $\{\mathcal{R}_j\}_{j=1}^N$ be the output of this lemma. 

Apply Proposition \ref{rectVsC2Lem} to the family of trapezoids $\{\mathcal{R}_j\}_{j=1}^N$. The resulting function $g$ almost satisfies Conclusion $1^\prime$ of Proposition \ref{C2GrainsProp}, except we only have the bound $\Vert g\Vert_{C^2}\lesssim \delta^{-3\eps_2}.$ This can be remedied by restricting $(\tubes_1,Y_1)_\delta$ to a slab of dimensions roughly $1\times 1\times \delta^{3\eps_2}$ and then applying a vertical re-scaling using Lemma \ref{rescalingPreservesExtremality}. After this re-scaling, the re-scaled function $g$ will satisfy $\Vert g\Vert_{C^2}\leq 1$, as desired, and the resulting collection of tubes $(\tubes',Y')_{\delta'}$ will continue to satisfy Conclusion 2 of Proposition \ref{C2GrainsProp}. If $\delta_0'$ is chosen sufficiently small (depending on $\delta_0$ and $\eps$) and if $\eps_2$ is chosen sufficiently small, then we will have $\delta'\leq\delta_0$, and $(\tubes',Y')_\delta'$ will be $\eps$-extremal. 
\end{proof}


\section{The slope function has large slope}\label{largeSlopeSec}
Our goal in this section is to strength the conclusions of Proposition \ref{C2GrainsProp} by showing that the global grains slope function has derivative bounded away from 0. In this step we make important use of the fact that the tubes in $\tubes$ point in different directions (though with a modification of our argument, a weaker condition specifying that the tubes in $\tubes$ do not concentrate too tightly into $\rho$ tubes is sufficient---see \cite{WZ}). More precisely, we will prove the following.

\begin{prop}\label{largeSlopeSlopeFunction}
For all $\eps,\delta_0>0$, there exists $\delta\in (0,\delta_0]$ and a $\eps$-extremal set of tubes $(\tubes,Y)_\delta$ with the following properties:
\begin{enumerate}
\item[$1^{\prime\prime}.$] {\bf Global grains with nonsingular $C^2$ slope function.}\\
There is a function $f\colon [-1,1]\to\RR$ with $1\leq |f'(z)|\leq 2$ and $|f''(z)|\leq 1/100$ for all $z\in[-1,1],$ so that $\big( E_{\tubes} \cap \{z = z_0\}\big) \cdot (1, f(z_0), 0)$ is a $(\delta, 1-\sigma, \delta^{-\eps})_1$-ADset for each $z_0\in [-1,1]$.

\item[2.] {\bf Local grains with Lipschitz plane map.} \\
$(\tubes,Y)_{\delta}$ has a 1-Lipschitz plane map $V$. For all $\rho\in[\delta,1]$ and all $p\in \RR^3$, $V(p) \cdot \big( B(p,\rho^{1/2}) \cap E_\tubes\big)$ is a $(\rho, 1-\sigma, \delta^{-\eps})_1$-ADset. 
\end{enumerate}
\end{prop}

In the remainder of the paper, only  Item $1''$  will be used. We keep Item $2$  for completeness. 
We will begin with a technical lemma, which says that if $U$ is a segment of a tube $T$, and if the image of a large subset of $U$ under a linear projection $p\mapsto p\cdot v$ is an ADset, then $v$ must be nearly orthogonal to $\dir(T)$. The precise statement is as follows.
\begin{lem}\label{dottedADSetLem}
Let $\eps>0$,  $0\leq\alpha<1,$ and $0<\delta\leq\rho\leq 1$. Let $T$ be a $\delta$-tube, let $U\subset T$ be a tube segment of length $\rho^{1/2}$, and let $F\subset U$ with 
\begin{equation}\label{FMostlyFull}
|F|\geq\delta^{\eps}|U|.
\end{equation} 
Let $v\in S^2$, and suppose $v\cdot F$ is a $(\rho, \alpha, \delta^{-\eps})_1$-ADset. Then
\begin{equation}
|v\cdot \dir(T)|\lesssim \delta^{\frac{-2\eps}{1-\alpha}}\rho^{1/2}.
\end{equation}
\end{lem}
\begin{proof}
We will suppose that $|v\cdot \dir(T)|\geq \rho^{1/2}$, since otherwise there is nothing to prove. Let $\tau = |v\cdot \dir(T)|$. Divide $U$ into sub-segments of length $\rho/\tau$ (by hypothesis $\rho/\tau\leq \rho^{1/2}$, so there is at least one such segment), and let $\mathcal{S}$ be this set of segments. For each $S\in\mathcal{S}$, $v\cdot S$ is an interval of length $\sim\rho$, and no point $x\in\RR$ is contained in more than 10 of the sets $\{v\cdot S\colon S\in\mathcal{S}\}$. By \eqref{FMostlyFull}, we have
\[
\#\{S\in\mathcal{S}\colon S\cap F\neq\emptyset\}\geq\delta^{\eps}\tau\rho^{-1/2},
\]
and thus 
\begin{equation}\label{rhoCoveringVFLowerBd}
\mathcal{E}_{\rho}(v\cdot F)\gtrsim \delta^{\eps}\tau\rho^{-1/2}.
\end{equation}
On the other hand, $v\cdot F\subset v\cdot U$ is contained in an interval $I\subset\RR$ of length $\sim\rho^{1/2}\tau$, and hence
\begin{equation}\label{rhoCoveringVFUpperBd}
\mathcal{E}_{\rho}(v\cdot F)\leq \mathcal{E}_{\rho}((v\cdot F) \cap I)\leq\delta^{-\eps}(|I|/\rho)^\alpha\lesssim \delta^{-\eps}(\rho^{-1/2}\tau)^\alpha.
\end{equation}
Comparing \eqref{rhoCoveringVFLowerBd} and \eqref{rhoCoveringVFUpperBd}, we conclude that $\tau\lesssim \delta^{\frac{-2\eps}{1-\alpha}}\rho^{1/2}$, as desired.
\end{proof}

The next lemma says that if $(\tubes,Y)_\delta$ is an extremal set of tubes pointing in different directions with $C^2$ slope function $f$, then $f$ cannot have small derivative on any ``thick'' slab that contains (at least) an average amount of the mass of $(\tubes,Y)_\delta$. This lemma contains the key geometric argument needed to prove Proposition \ref{largeSlopeSlopeFunction}. 

\begin{lem}\label{largeDerivativeOnInterval}
For all $\eps>0,$ there exists $\eta>0,\delta_0>0$ so that the following holds for all $\delta\in (0, \delta_0]$. Let $(\tubes,Y)_\delta$ be a set of tubes that satisfies the conclusions of Proposition \ref{C2GrainsProp}, with $\eta$ in place of $\eps$. Suppose in addition that the tubes in $\tubes$ point in $\delta$-separated directions. Let $J\subset[-1,1]$ be an interval, with $\delta^{1/2}\leq |J|\leq\delta^\eps$. Suppose that
\begin{equation}\label{JFull}
\sum_{T\in\tubes}|Y(T)\cap (\RR^2\times J)|\geq\delta^{\eta}|J|,
\end{equation}
and that $(\RR^2\times J)\cap E_{\tubes}$ can be covered by a union of $\leq\delta^{-\eta}|J|^{-2 + \sigma}$ cubes of side-length $J$. 

Then the slope function $f\colon [-1,1]\to\RR$ from Proposition \ref{C2GrainsProp} satisfies 
\begin{equation}\label{FPrimeLargeOnJ}
|f'(z)|\geq|J|\quad\textrm{for all}\ z\in J.
\end{equation}
\end{lem}
\begin{proof}
Suppose not, i.e.~there exists $z_0\in J$ with $|f'(z_0)|<|J|$. Since $\Vert f \Vert_{C^2}\leq 1$, we have $|f'(z)|\leq 2|J|$ for all $z\in J$. We will obtain a contradiction.

Define $\rho = |J|^2$, so $\delta\leq\rho\leq\delta^{2\eps}$. After a translation and rotation, we may suppose that $J = [0, \rho^{1/2}]$ and $f(0) = 0$ (note that after this rotation, our shadings $Y(T)$ will be unions of cubes that are no longer axis parallel, but this is harmless for the arguments that follow).

After a refinement, we may suppose there is a number $\delta^{-\sigma+\eta}\leq\mu\leq\delta^{-\sigma-\eta}$ so that for each point $p\in (\RR^2\times J)\cap E_{\tubes}$, we have $\#\tubes(p)=\mu$. Define $F = (\RR^2\times J)\cap E_{\tubes}$. By \eqref{JFull}, we have $
\delta^{\sigma+2\eta}|J| \lessapprox |F|\lessapprox \delta^{\sigma-2\eta}|J|$. In the arguments that follow, we will describe several refinements of $F$ that endow the set with additional structural properties. Each of these refinements will continue to be a union of (rotated) $\delta$-cubes.

\medskip

\noindent{\bf Refinement 1: Every global grain is popular.} Recall that $(\tubes,Y)_\delta$ satisfies Item $1^\prime$ from Proposition \ref{C2GrainsProp}. Informally, this says that $E_{\tubes}$ is a union of global grains (rectangular prisms of dimensions $1\times\delta\times\delta$), and that most of these grains are almost full. More precisely, we can select a subset $F_1\subset F$ with $|F_1|\geq\delta^{2\eta}|F|$ so that the following property holds: For each $z_1\in J$ and each $w\in (F\cap\{z=z_1\})\cdot (1, f(z_1), 0)$, define the $1\times\delta\times\delta$ rectangular prism
\begin{equation}\label{defnGrainGrz}
G_{w,z_1} = \{ |z-z_1|\leq \delta\} \cap \{ |(x,y)\cdot(1, f(z_1))-w|\leq\delta\}.
\end{equation}
Then $G_{w,z_1}$ has large intersection with $F_1$; specifically, we have
\begin{equation}\label{fullGrains}
| G_{w,z_1}\cap  F_1| \geq \delta^{2\eta}|G_r| \gtrsim \delta^{2+2\eta}.
\end{equation}

\medskip

\noindent{\bf Refinement 2: Every global grain intersects a common $\{y=y_0\}$ plane.}
Each of the global grains $G_{w,z_1}$ from \eqref{defnGrainGrz} will intersect each plane of the form $\{y=y_0\}$ for $0\leq y_0\leq 1$. While it may not be true that every set $G_{w,z_1}\cap F_1$ will intersect every plane $\{y=y_0\}$, it is true that a typical set of this form is likely to intersect a typical plane. Thus we can find a choice of plane $\{y=y_0\}$ that intersects a large proportion of the sets $G_{w,z_1}\cap F_1$. We turn to the details.

Recall that $F$ (and hence $F_1$) can be covered by a union of $\leq\delta^{-\eta}\rho^{-1+\sigma/2}$ cubes of side-length $\rho^{1/2}$; weakening this inequality by a constant factor, we may suppose that these cubes are aligned with the $\rho^{1/2}$-grid. Denote the union of these cubes by $K$. Let $Y_{\operatorname{bad}}\subset[-1,1]$ be the set of $y_0$ for which $|\{y=y_0\}\cap K|\geq \delta^{-7\eta}\rho^{1/2+\sigma/2}$. By Fubini we have $|Y_{\operatorname{bad}}|\leq\delta^{4\eta}.$ If $G_{w,z_1}$ is a set of the form \eqref{defnGrainGrz}, then $|G_{w,z_1}\cap (\RR\times Y_{\operatorname{bad}}\times\RR)| \lesssim \delta^{4\eta+2}$, and hence if $G_{w,z_1}$ satisfies \eqref{fullGrains}, then
\[
|(G_{w,z_1}\cap  F_1) \backslash (\RR\times Y_{\operatorname{bad}}\times\RR)| \geq |G_{w,z_1}\cap  F_1|/2.
\]
Thus if we define $F_1'= F_1\backslash (\RR\times Y_{\operatorname{bad}}\times\RR)$, then (in light of Refinement 1) we have $|F_1'|\geq |F_1|/2$. 

By Fubini, there is a choice of $y_0\in [-1,1]\backslash Y_{\operatorname{bad}}$ so that
\[
\Big| \{y=y_0\}\cap F_1'\Big| \geq\frac{1}{2}|F_1'|,
\]
where on the LHS, $|\cdot|$ denotes 2-dimensional Lebesgue measure and on the RHS, $|\cdot|$ denotes 3-dimensional Lebesgue measure. Let $X=\bigcup G_{w,z_1}$, where the union is taken over all grains $G_{w,z_1}$ with $\{y=y_0\}\cap G_{w,z_1}\cap F_1'\neq\emptyset$, and let $F_2 = F_1'\cap X$. By \eqref{fullGrains}, we have $|F_2|\geq\delta^{4\eta}|F|$.

\medskip

\noindent{\bf Step 3: Local grains at scale $\rho$ extend to global grains.}
Recall that by Item $1'$ from Proposition \ref{C2GrainsProp}, each slice $E_{\tubes}\cap \{z=z_0\}$ is a union of global grains. A priori, the global grains for one value of $z_0$ need not be related to the grains for a different value of $z_0$ (though of course the slopes of these grains are related by the regularity of the slope function $f$). In this step, we will show that the local grains coming from Item 2 of Proposition \ref{C2GrainsProp} ensure consistency between the global grains for different slices $\{z=z_0\}$. We turn to the details.

Cover the rectangle $[-1,1]\times\{y=y_0\}\times J$ by interior-disjoint squares of side-length $\rho^{1/2}=|J|$. Since $y_0\not\in Y_{\operatorname{bad}}$, at most $\delta^{-10\eta}\rho^{\frac{\sigma-1}{2}}$ of these squares intersect $F_2\cap\{y=y_0\}.$ Denote this set of squares by $\mathcal{S}$. 

Let $S\in\mathcal{S}$ be a square, with center $p_S = (x_S, y_0, \rho^{1/2}/2)$. By Item $2$ from Proposition \ref{C2GrainsProp}, we have that $(S\cap F_2)\cdot V(p_S)$ is a $(\rho, 1-\sigma, \delta^{-\eta})_1$-ADset. Write $V(p_S) = (V_x(p_S), V_y(p_S), V_z(p_S))$. Since $S\cap F_2\subset\{y=y_0\}$, if we define $\tilde V(p_S) = \big( V_x(p_S), 0, V_z(p_S)\big)$, then
\begin{equation}\label{tildeVDoTADSet}
(S\cap F_2)\cdot \tilde V(p_S)\ \textrm{is a}\ (\rho, 1-\sigma, \delta^{-\eta})_1\textrm{-ADset}. 
\end{equation}

In particular, $S\cap F_2$ is contained in a union of at most $\delta^{-\eta}\rho^{\frac{\sigma-1}{2}}$ rectangles of dimensions $\rho^{1/2}\times\rho$. If $R\subset S$ is such a rectangle, and if $G_{w,z_1}$ is a grain of the form \eqref{defnGrainGrz} that intersects $R$, then since $|f(z_1)|\leq \rho$, we have $G_{w,z_1}\subset \{(x,y,z)\colon (x,z)\in 2R\}$, where $2R$ is the $\rho^{1/2}\times 2\rho$ rectangle that has the same major axis as $R$. But in light of \eqref{tildeVDoTADSet}, this implies that if we define the prism $P_S = \{(x,y,z)\colon (x,z)\in S\}$, then 
\[
(F_2 \cap P_S)\cdot \tilde V(p_S)\ \textrm{is a}\ (\rho, 1-\sigma, 2\delta^{-\eta})_1\textrm{-ADset}. 
\]
See Figure \ref{largeSlopeFig}

\begin{figure}[h]
\centering
\begin{overpic}[scale=0.4]{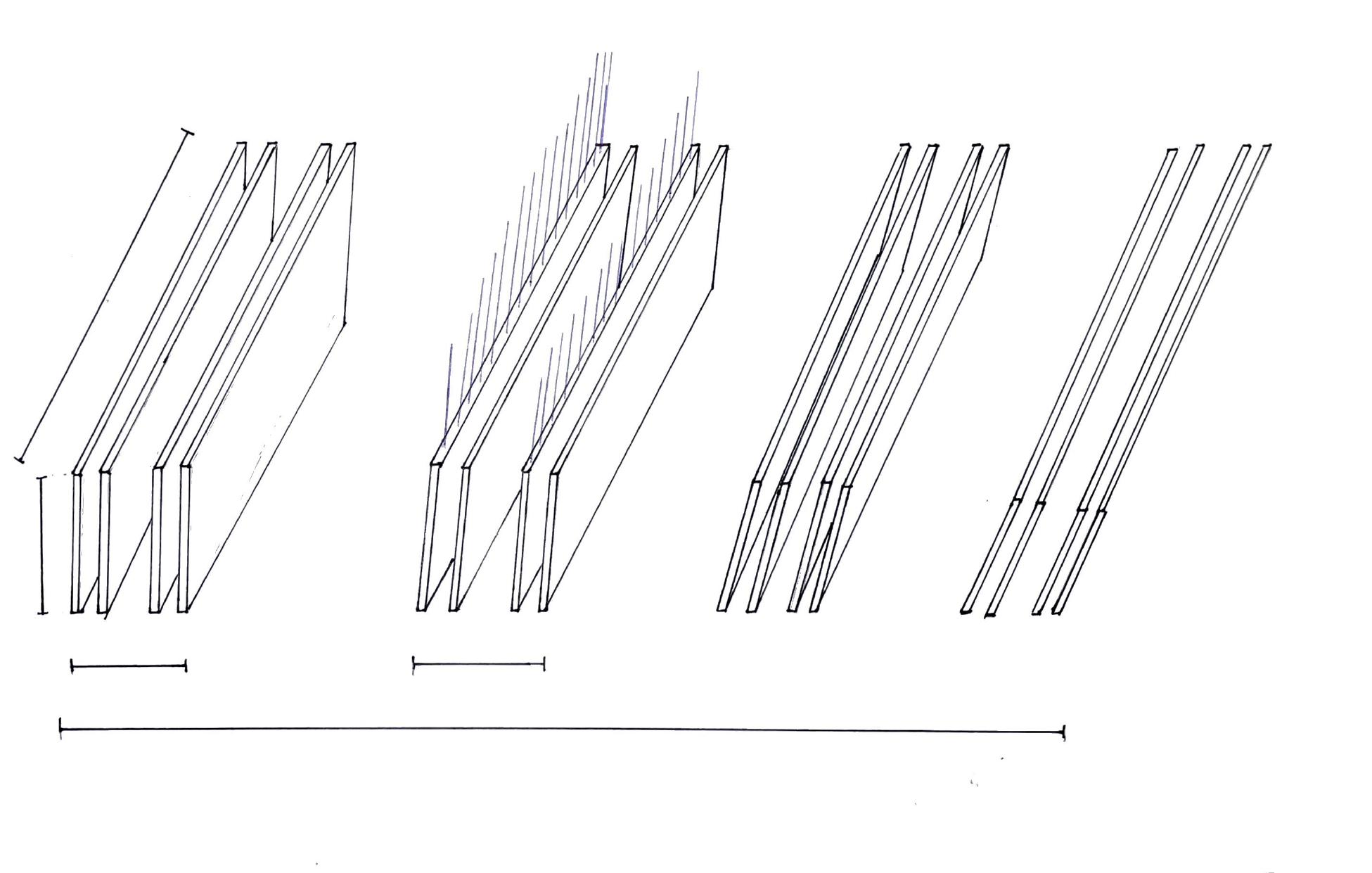}
\put (5,12) {{\footnotesize $\Delta x : \rho^{1/2}$}}
\put (30,12) {{\footnotesize $\Delta x : \rho^{1/2}$}}
\put (40,6) {{\footnotesize $\Delta x : 1$}}
\put (0,19) {{\footnotesize \rotatebox{90}{$\Delta z: \rho^{1/2}$}}}
\put (4,40) {{\footnotesize \rotatebox{63}{$\Delta y: 1$}}}
\end{overpic}
\caption{The rectangular prisms $\{P_S\colon S\in\mathcal{S}\}$, and the tubes from $\tubes$ (blue lines) that intersect them. For clarity, only a few tubes (blue lines) have been drawn.}
\label{largeSlopeFig}
\end{figure}

\medskip

\noindent{\bf Step 4: $\tubes$ contains parallel tubes.}
Many tubes from $\tubes$ must intersect each of the parallel slabs in Figure \ref{largeSlopeFig}, and if a tube $T$ intersects a slab, this imposes constraints on the possible directions of $v(T)$. In this step, we will show that these constraints force many tubes from  $\tubes$ to point in almost the same direction. This contradicts the assumption that the tubes in $\tubes$ point in $\delta$-separated directions. We turn to the details.

For each $T\in\tubes$, define $F_2(T) = Y(T)\cap F_2$. Then recalling \eqref{JFull}, we have 
\[
\sum_{T\in\tubes}|F_2(T)| \geq \delta^{5\eta}|J|=\delta^{5\eta}\rho^{1/2}.
\]
Recall that each $T\in\tubes$ intersects at most 3 prisms $P_S, S\in\mathcal{S}$, and for such a prism we have $|F_2(T)\cap P_S|\leq \rho^{1/2}\delta^2$. Thus by pigeonholing there exists a square $S\in\mathcal{S}$ and a set $\tubes_S\subset\tubes$ with 
\begin{equation}\label{sizeOfTubesS}
\#\tubes_S \gtrsim \delta^{-2}(\#\mathcal{S})^{-1}\gtrsim \delta^{11\eta-2}\rho^{\frac{1-\sigma}{2}},
\end{equation} 
so that $|F_2(T)|\gtrsim \delta^{11\eta}\rho^{1/2}\delta^{2}$ for each $T\in\tubes_S$. 

But for each such $T\in\tubes_S$, we have that $F_2(T) \cdot \tilde V(p_S)\subset (F_2\cap P_S)\cdot \tilde V(p_S)$ is a $(\rho, 1-\sigma, 2\delta^{-\eta})_1$-ADset. Thus by Lemma \ref{dottedADSetLem}, we have 
\[
|\dir(T)\cdot \tilde V(p_S)|\leq \delta^{-C_0\eta}\rho^{1/2},
\]
where the constant $C_0$ depends on $\sigma$. But comparing with \eqref{sizeOfTubesS} and pigeonholing, we conclude that there is a direction $v\in S^2$ so that
\[
\#\{T\in\tubes_S\colon |\dir(T)-v| \leq \delta\}\geq \delta^{C_1\eta}\rho^{-\sigma/2}\geq \delta^{C_1\eta-\sigma \eps},
\]
where $C_1$ is a constant that depends on $\sigma$. If $\eta>0$ is selected sufficiently small (depending on $\sigma$ and $\eps$), then this contradicts the assumption that the tubes in $\tubes$ point in $\delta$-separated directions. 
\end{proof}

We are almost ready to prove Proposition \ref{largeSlopeSlopeFunction}. First we will prove an intermediate result, which constructs an extremal set of tubes $(\tubes,Y)_\delta$ with the property that the shadings $Y(T)$ are supported on a thick slab, and the global grains slope function $f$ is large on this slab. 

\begin{lem}\label{intervalWithBigConstDerivativeLem}
For all $\eps,\delta_0>0$, there exists $\delta\in (0,\delta_0]$ and a set of tubes $(\tubes,Y)_\delta$ that satisfy the conclusions of Proposition \ref{C2GrainsProp}, as well as the following additional properties. First, for each $T\in\tubes$ we have $Y(T)\subset [-1,1]^2\times [0,\delta^{\eps}]$. Second, there is a number $m\geq\delta^{\eps}$ so that $m\leq |f'(z)|\leq 2m$ for all $z\in [0,\delta^{\eps}]$.
\end{lem}
\begin{proof}
Let $\eps_1,\eps_2>0$ be small constants to be chosen below. We will select $\eps_1$ very small compared to $\eps_2$, and $\eps_2$ very small compared to $\eps$.

Let $\delta\in(0,\delta_0]$ and let $(\tubes,Y)_\delta$ be a collection of tubes that satisfy the conclusions of Proposition \ref{C2GrainsProp}, with $\eps_1$ in place of $\eps$. After passing to a sub-collection, we may suppose that the tubes point in $\delta$-separated directions. If $\eps_1$ is selected sufficiently small, then we may apply Proposition \ref{multiScaleStructureExtremal} with $\delta^{\eps_2}$ in place of $\rho$, and we can select a refinement $(\tubes_1,Y_1)$ of $(\tubes,Y)$ so that $E_{\tubes_1}$ is covered by $\lesssim \delta^{-\eps_2+\sigma-3}$ cubes of side-length $\delta^{\eps_2}$. 

By pigeonholing, we can find an interval $J\subset[-1,1]$ of length $\delta^{\eps_2}$ so that $E_{\tubes_1}\cap(\RR^2\times J)$ is covered by $\lesssim \delta^{-2\eps_2+\sigma-2}$ cubes of side-length $\delta^{\eps_2}$, and $\sum_{T\in\tubes_1}|Y_1(T)\cap (\RR^2\times J)|\gtrsim \delta^{2\eps_2}|J|$. 

Finally, if $\eps_2$ is selected sufficiently small compared to $\eps$, then we can apply Lemma \ref{largeDerivativeOnInterval} to conclude that $|f'(z)|\geq\delta^{\eps}$ for all $z\in J$. Since $|f''(z)|\leq 1$, there must be a number $m\geq\delta^{\eps}$ so that $m\leq |f'(z)|\leq 2m$ on $J$. Finally, replace each shading $Y_1(T)$ by $Y_1(T)\cap (\RR^2\times J)$, and apply a translation in the $z$-direction so that $J = [0,\delta^{\eps}]$. 
\end{proof}

Proposition \ref{largeSlopeSlopeFunction} follows by applying a re-scaling to the collection of tubes fro Lemma \ref{intervalWithBigConstDerivativeLem}. The details are as follows.
\begin{proof}[Proof of Proposition \ref{largeSlopeSlopeFunction}]
Let $\eps_1>0$ be a small constant to be chosen below. We will select $\eps_1$ very small compared to $\eps$. Let $\delta\in(0,\delta_0^2]$ and let $(\tubes,Y)_\delta$ satisfy the conclusions of Lemma \ref{intervalWithBigConstDerivativeLem} with $\eps_1$ in place of $\eps$ (and some $m\geq\delta^{\eps_1}$). 

Define $\phi(x,y,z)= (x, \frac{m^2y}{100}, \frac{100z}{m}),$ and define $\tilde f(z) = 100m^{-2}f(mz/100)$. Then 
\begin{equation}\label{boundsOnFpFpp}
1\leq |\tilde f'(z)|\leq 2\quad\textrm{and}\quad |\tilde f''|\leq 1/100,\qquad z\in [0, m^{-1}\delta^{\eps_1}].
\end{equation}
We can modify $\tilde f$ on $[-1,1]\backslash [0, m^{-1}\delta^{\eps}]$ so that \eqref{boundsOnFpFpp} holds for all $z\in [0,1]$. Furthermore, if $(x,y,z)\in [-1,1]^2\times J$, then 
\[
(x,y,z)\cdot(1, f(z),0)=x + yf(z) = \phi(x,y,z)\cdot(1, \tilde f(\tilde z), 0).
\] 
It remains to construct the $\eps$-extremal collection of tubes $(\tilde \tubes, \tilde Y)_{\tilde \delta}$ that satisfies the conclusions of Proposition \ref{largeSlopeSlopeFunction}. We will do this (for some $\tilde\delta\leq\delta^{1-C\eps_1}$; if $\eps_1$ is selected sufficiently small, then $C\eps_1\leq1/2$ and hence $\tilde\delta\leq \delta_0$, as required) by applying the scaling $\phi\colon [-1,-1]^2\times [0, \delta^{\eps_1}] \to [-1,1]^3$ to $(\tubes,Y)_\delta$, and using Lemma \ref{rescalingPreservesExtremality}. 
\end{proof}


\subsection{Twisted projections}

We will re-interpret the conclusion of Proposition \ref{largeSlopeSlopeFunction} as a statement regarding twisted projections. First a definition.

\begin{defn}[Twisted projection]
For $f\colon [-1,1]\rightarrow \mathbb{R}$, we define the \emph{twisted projection} $\pi_f\colon [-1,1]^3\to\RR^2$ by
\[
\pi_f(x,y,z) = (x + f(z) y, z).
\]
\end{defn}
With this definition, the pair $(\tubes,Y)_\delta$ and the function $f$ from the conclusion of Proposition \ref{largeSlopeSlopeFunction} satisfies $|\pi_f(E_{\tubes})|\leq \delta^{\sigma-\eps}$.

Applying H\"older's inequality, we have the following consequence of Proposition \ref{largeSlopeSlopeFunction}.
\begin{cor}\label{bigL3Norm}
For all $\eps>0$, $\delta_0>0$, there exists $\delta\in (0,\delta_0]$; an $\eps$-extremal collection of $\delta$-tubes  $\tubes$ (see Definition~\ref{defnExtremal}); and a function $f\colon [-1,1]\rightarrow \mathbb{R} $ with $1\leq |f'(z)|\leq 2$ and $|f''(z)|\leq 1/100$ for all $z\in[-1,1],$ so that 
\begin{equation}\label{L32NormSumChiPiF}
\Big \Vert \sum_{T\in\tubes}\chi_{\pi_f(T)} \Big\Vert_{L^{3/2}(\RR^2)} \geq \delta^{-1-\sigma/3+\eps}.
\end{equation}
\end{cor}
\begin{proof}
	By Definition~\ref{defnOfM} \ref{DefnOfMBigVolItem} and H\"{o}lder's inequality, 
	\[
	\delta^{-1+\eps} \leq \Big\Vert \sum_{T\in \tubes} \chi_{\pi_f(Y(T))} \Big\Vert_{L^1(\pi_f(E_{\tubes}))} \leq  |\pi_f(E_f)|^{1/3} \Big\Vert \sum_{T\in \tubes} \chi_{\pi_f(Y(T))} \Big\Vert_{L^{3/2} (\mathbb{R}^2)}. 
	\]
	The conclusion follows from Proposition~\ref{largeSlopeSlopeFunction}. 
\end{proof}

\section{A $C^2$ projection theorem}\label{C2ProjThmSec}
In this section we will introduce a version of Wolff's circular maximal theorem \cite{Wo, KoW} for $C^2$ curves, and show how this theorem can bound expressions like the LHS of \eqref{L32NormSumChiPiF}. The goal is to prove the following.
\begin{prop}\label{smallL3Norm}
For all $\eps>0$, there exists $\delta_0>0$ so that the following holds for all $\delta\in (0, \delta_0]$. Let $\tubes$ be a set of $\delta$-tubes pointing in $\delta$-separated directions. Let $f\colon [-1, 1]\rightarrow \mathbb{R}$ with $1\leq |f'(z)|\leq 2$ and $|f''(z)|\leq 1/100$ for all $z\in[-1,1]$. Then
\begin{equation}
\Big \Vert \sum_{T\in\tubes}\chi_{\pi_f(T)} \Big\Vert_{L^{3/2}(\RR^2)} \leq \delta^{-1-\eps}.
\end{equation}
\end{prop}

Comparing Corollary \ref{bigL3Norm} and Proposition \ref{smallL3Norm}, we conclude that $\sigma=0$. Combined with Proposition \ref{sigmaNVsDimensionProp}, this concludes the proof of Theorem \ref{mainThm} and Theorem \ref{mainThm}$^\prime$. All that remains is to prove Proposition \ref{smallL3Norm}.

Our main tool for proving Proposition \ref{smallL3Norm} is a version of Wolff's $L^{3/2}$ circular maximal function bound for $C^2$ curves, which is a special case of Theorem 1.7 in \cite{PYZ}. Before stating the result, we recall Definition 1.6 from \cite{PYZ}.
\begin{defn}\label{cinematicFunctionDefn}
Let $I\subset\RR$ be a compact interval and let $\mathcal{F}\subset C^2(I)$. We say $\mathcal{F}$ is a family of \emph{cinematic functions}, with cinematic constant $K$ and doubling constant $D$ if the following conditions hold.
\begin{enumerate}
\item $\mathcal{F}$ has diameter at most $K$.
\item  $\mathcal{F}$ is a doubling metric space, with doubling constant at most $D$.
\item For all $f,g\in \mathcal{F},$ we have
\begin{equation}\label{lowerBdDifferencesFFpFpp}
\inf_t\big( |f(t)-g(t)| + |f'(t)-g'(t)| + |f''(t)-g''(t)|\big) \geq K^{-1} \Vert f-g\Vert_{C^2(I)}.
\end{equation}
\end{enumerate}
\end{defn}

\begin{thm}\label{twistedProjectionPYZ}
Let $I\subset\RR$ be a compact interval and let $\mathcal{F}\subset C^2(I)$ be a family of cinematic functions. Then for all $\eps>0$, there exists $\delta_0>0$ depending on $\eps, K, D$ so that the following holds for all $\delta\in (0,\delta_0]$. Let $F\subset\mathcal{F}$ be a $\delta$-separated set that satisfies the Frostman non-concentration condition
\begin{equation}\label{frostmanCondition}
\#\big(F\cap B\big)\leq \delta^{-\eps}(r/\delta)\quad\textrm{for all balls}\ B\subset C^2(I)\ \textrm{of diameter}\ r.
\end{equation}
Then
\[
\Big\Vert \sum_{f\in F}\chi_{f^\delta}\Big\Vert_{3/2} \leq \delta^{-\eps},
\]
where $f^\delta$ is the $\delta$-neighborhood of $\operatorname{graph}(f)$.
\end{thm}

We will be interested in the following family of functions. For $f\in C^2(I)$ and $a,b,d\in\RR$, define 
\begin{equation}\label{defnGabd}
g_{a,b,d}(t) = a + bf(t) + dtf(t).
\end{equation}
Define $M_f = \{g_{a,b,d}\colon a,b,d\in [-1,1]\}$. We will show that under certain conditions on $f$, this is a family of cinematic functions.
\begin{lem}\label{GabdCinematic}
Let $f\colon [0,1]\to\RR$ with $f(0)=0$, $1\leq |f'(t)|\leq 2$, and $|f''(t)|\leq 1/100$ for all $t\in [0,1]$. Then $M_f$ is a family of cinematic functions, and the implicit constants in Definition \ref{cinematicFunctionDefn} are bounded independently of the choice of $f$.
\end{lem}
\begin{proof}
Let $a_1,a_2,b_1,b_2,d_1,d_2\in [-1,1]$, and let $g_i = g_{a_i,b_i,d_i}$. We have $\Vert g_1-g_2\Vert_{C^2}\sim |a_1-a_2|+|b_1-b_2|+|d_1-d_2|$, so it suffices to show that
\begin{equation}\label{bigInf}
\inf_{t\in [0,1]} |g_1(t)-g_2(t)| + |g_1'(t) - g_2'(t)| + |g_1''(t) - g_2''(t)| \geq \frac{1}{200}\big(|a_1-a_2|+|b_1-b_2|+|d_1-d_2|\big).
\end{equation}

\noindent{\bf Case 1:} $|a_1-a_2|\geq  10 ( |b_1-b_2|+ |d_1-d_2|)$.\\
We have  $|f(t)|\leq 2t$ and 
\[
\inf_{t\in [0,1]} |g_1(t)-g_2(t)|\geq |a_1-a_2|- 2(|b_1-b_2| + |d_1-d_2|) \geq \frac{1}{2}|a_1-a_2|.
\]
Thus \eqref{bigInf} holds and we are done.

\medskip

\noindent{\bf Case 2a:} $|a_1-a_2|<  10 ( |b_1-b_2|+ |d_1-d_2|)$ and $|b_1-b_2|\geq 10|d_1-d_2|$.\\
We have 
\begin{equation*}
\begin{split}
|g_1'(t) - g_2'(t)| 
& \geq |b_1-b_2|\ |f'(t)| - |d_1-d_2|\big( |f(t)| + |tf'(t)|\big) \\
&\geq |b_1-b_2| - \frac{1}{10}|b_1-b_2|\big( |f(t)| + |tf'(t)|\big) \geq \frac{1}{2}|b_1-b_2|.
\end{split}
\end{equation*}
Thus \eqref{bigInf} holds and we are done. 

\medskip

\noindent{\bf Case 2b:} $|a_1-a_2|<  10 ( |b_1-b_2|+ |d_1-d_2|)$ and $|b_1-b_2|< 10|d_1-d_2|$ .\\
We have 
\[
|g_1''(t) - g_2''(t)| \geq |d_1-d_2| |f'(t)| - \big(|b_1-b_2| + |d_1-d_2|t \big)|f''(t)|\geq |d_1-d_2| - 11|d_1-d_2|\ |f''(t)| \geq \frac{1}{2}|d_1-d_2|.
\]
Thus \eqref{bigInf} holds and we are done.

\end{proof}

\noindent We are now ready to prove Proposition \ref{smallL3Norm}.
\begin{proof}[Proof of Proposition \ref{smallL3Norm}]
Each $T\in\tubes$ has a coaxial line pointing in the direction $(c,d,1)$. Without loss of generality, we may suppose that $c,d\in \delta\ZZ\cap[-1,1]$, and the directions $(c,d,1)$ are distinct for distinct $T\in\tubes$. By the triangle inequality, for each $c\in[-1,1]$ it suffices to prove the estimate

\begin{equation}\label{estimatOneDimFamilyTubes}
\Big \Vert \sum_{i\in\mathcal{I}} \chi_{\pi_f(T_i)} \Big\Vert_{L^{3/2}(\RR^2)} \leq \delta^{-\eps},
\end{equation}
where $N=\lceil \delta^{-1} \rceil, \mathcal{I}\subset \{-N,\ldots,N\}$  (the set of indices $\mathcal{I}$ depends on $c$) and for each $i\in\mathcal{I}$, $T_i$ is a $\delta$-tube in $\mathbb{T}$ whose coaxial line $L_i$ is of the form $(a_i, b_i,0)+\RR(c, i\delta,1)$ with $a_i,b_i\in [-1,1]$.
For each tube $T_i$, $\pi_f(T_i)$ is contained in the $2\delta$-neighborhood of the plane curve $\pi_f(L_i)$, which is the graph of the function $g_i = g_{a_i, b_i, d_i}(t) + ct$ (recall that $g_{a_i, b_i, d_i}$ is as defined in \eqref{defnGabd}). By Lemma \ref{GabdCinematic}, the set  $\{g_{a,b,d}(t) + ct\colon a,b,d\in [-1,1]\}\subset C^2([0,1])$ is a family of cinematic functions, and the implicit constants in Definition \ref{cinematicFunctionDefn} are bounded independently of $c$, $f$, and $N$ (and hence $\delta$). 

The estimate \eqref{estimatOneDimFamilyTubes} thus follows from Theorem \ref{twistedProjectionPYZ}, since the set of functions $\{g_i\}_{i\in\mathcal{I}}$ clearly satisfies the Frostman condition \eqref{frostmanCondition}.
\end{proof}

\begin{rem}
It is also possible to conclude that $\sigma = 0$ using restricted projections \cite{GGGHMW, GGM, GGW} introduced by F\"{a}ssler and Orponen in \cite{FOrestr}, rather than using Theorem \ref{twistedProjectionPYZ}.   We will briefly sketch this latter approach here. We will use the same notation as in the proof of Proposition~\ref{estimatOneDimFamilyTubes}. By pigeonholing, there exists $c\in [-1, 1]$ and the corresponding set $\mathcal{I}\subset \{-N, \dots, N\}$ such that  $\#\mathcal{I} \geq \delta^{-\eps/100}$ and for each $i\in \mathcal{I}$, $T_i$ is a $\delta$-tube in $\mathbb{T}$ whose coaxial line $L_i$ is of the form $(a_i, b_i, 0) + \mathbb{R}(c, i\delta, 1)$ with $a_i, b_i \in [-1, 1]$.  For this fixed $c$ and  each $z\in [-1,1]$, define the (translated) restricted projection  as
 \begin{equation} 
 	P_z\colon (a, b, d) \mapsto ( a +cz + f(z) (b+ dz))= (a, b, d) \cdot (1, f(z), f(z)z) + cz.  \end{equation} 
Since $f$ is $C^2$ and $1\leq |f'(z)|\leq 2$ for each $z\in [-1, 1]$, there exists a nondegenerate curve $\gamma\colon [-1, 1]\rightarrow \mathbb{R}^3$ such that $\gamma'(z)=  (1, f(z), f(z)z)$. Here nondegenerate means $\det(\gamma'(z), \gamma''(z), \gamma'''(z)) \neq 0$ for each $z\in [-1, 1]$. The conclusion follows from  \cite[Theorem 2]{GGM} with $s=\sigma+\eps/10$, $t=1-\eps/10$ (note that \cite[Theorem2]{GGM} still holds if we assume each $\delta$-ball contained in $H$ intersects $\delta^{\eps/10} \#\Theta$ many slabs). See also \cite[Theorem 2.1]{GGW}   with $m=1, \alpha=1$. 
 	
\end{rem}

\section{Proof of Theorem \ref{SWThm}}\label{SWProofSec}

In this section we prove Theorem \ref{SWThm}. For the reader's convenience we will recall it here. 

\begin{SWThmRepeat}
For all $\eps>0$, there exists $\eta,\delta_0>0$ so that the following holds for all $\delta\in (0, \delta_0]$. Let $F,G\subset[0,1]^2$ be $(\delta, 1, \delta^{-\eta})$-ADsets. Then at least one of the following must hold.
\begin{enumerate}[label=(\Alph*)]
	\item\label{SWThmItemARepeat} There exist orthogonal lines $\ell$ and $\ell^\perp$ whose $\delta$-neighborhoods cover a significant fraction of $F$ and $G$ respectively. In particular, 
\begin{equation*}\tag{\ref{largeIntersectionRec}}
\mathcal{E}_\delta(N_\delta(\ell)\cap F)\geq \delta^{\eps-1} \ \textrm{and}\ \ 
\mathcal{E}_\delta(N_\delta(\ell^\perp)\cap G)\geq \delta^{\eps-1}.
\end{equation*}

\item\label{SWThmItemBRepeat} Let $\mathcal{H}\subset F\times G\times G$ be any set satisfying $\mathcal{E}_{\delta}(\mathcal{H})\geq \delta^{\eta-3}$. 
Then there exists $\rho\geq\delta$ and an interval $I$ of length at least $\delta^{-\eta}\rho$ such that 
\begin{equation*}\tag{\ref{bigRCoveringNumber}}
\mathcal{E}_\rho\big(I \cap \{a\cdot(b_1-b_2)\colon (a,b_1,b_2)\in \mathcal{H}\} \big) \geq \big(|I|/\rho\big)^{1-\eps}.
\end{equation*}
\end{enumerate}
\end{SWThmRepeat}

A key ingredient in our proof is a result about radial projections from \cite{OSW}, which in turn builds off of ideas from \cite{OS21, SW0}, and ultimately relies on Bourgain's discretized sum-product theorem. Our proof of Theorem \ref{SWThm} actually proves a slightly stronger statement that has weaker hypotheses. We first recall the following definition.

\begin{defn}\label{deltasCSet}
For $n\geq 1,\ \alpha\in (0, n]$, $C\geq 1$ and $\delta>0$, We say a set $E\subset\RR^n$ is a $(\delta,\alpha, C)_n$-Frostman set if $E$ is a union of $\delta$-cubes, and for all $x\in\RR^n$ and all $r\geq\delta$ we have  
	\begin{equation}\label{deltasCSetEqn}
	|E\cap B(x,r)| \leq C  r^\alpha|E|.
	\end{equation}
We form an analogous definition if $E\subset S^1$ is a union of $\delta$-arcs. 
\end{defn}
In practice we will often consider sets $E\subset \RR^2$, so we will drop the subscript $2$ from our notation. Definition \ref{deltasCSet} first appeared in \cite{BGam} and is closely related to Katz and Tao's notion of a $(\delta,\alpha)_n$-set from \cite{KT} (sets obeying Katz and Tao's non-concentration condition are now sometimes referred to as $(\delta,\alpha)_n$-Katz-Tao sets to avoid confusion with Definition \ref{deltasCSet}). However, in Katz and Tao's definition from \cite{KT}, the quantity $|E|$ on the RHS of \eqref{deltasCSetEqn} is replaced by $\delta^{n-\alpha}$; the former could be much larger. 

We can now state the following slightly stronger variant of Theorem \ref{SWThm}. 

\begin{SWThmRepeatPrime}
For all $\eps>0$, there exists $\eta,\delta_0>0$ so that the following holds for all $\delta\in (0, \delta_0]$. Let $F,G_1,G_2\subset[0,1]^2$ be $(\delta, 1, \delta^{-\eta})$-Frostman sets. Then at least one of the following must hold.
\begin{enumerate}[label=(\Alph*)]
	\item\label{SWThmItemAPrimed} There exist orthogonal lines $\ell$ and $\ell^\perp$ whose $\delta$-neighborhoods cover a significant fraction of $F$ and $G_1\cup G_2$ respectively. In particular, 
\begin{equation}\label{largeIntersectionRecPrime}
\mathcal{E}_\delta(N_\delta(\ell)\cap F)\geq \delta^{\eps-1},\quad \mathcal{E}_\delta(N_\delta(\ell^\perp)\cap G_1)\geq \delta^{\eps-1},\quad
\mathcal{E}_\delta(N_\delta(\ell^\perp)\cap G_2)\geq \delta^{\eps-1}.
\end{equation}

\item\label{SWThmItemBPrimed} Let $\mathcal{H}\subset F\times G_1\times G_2$ be any set satisfying $\mathcal{E}_{\delta}(\mathcal{H})\geq \delta^{\eta-3}$. 
Then there exists $\rho\geq\delta$ and an interval $I$ of length at least $\delta^{-\eta}\rho$ such that 
\begin{equation}\label{bigRCoveringNumberPrime}
\mathcal{E}_\rho\big(I \cap \{a\cdot(b_1-b_2)\colon (a,b_1,b_2)\in \mathcal{H}\} \big) \geq \big(|I|/\rho\big)^{1-\eps}.
\end{equation}
\end{enumerate}
\end{SWThmRepeatPrime}

Observe that if $E$ is a union of $\delta$-cubes that is also a $(\delta,\alpha, C)_n$-ADset, and if $|E| \geq c\delta^{n-\alpha}$ for some $c>0$, then $E$ is a $(\delta,\alpha, C/c)_n$-Frostman set, and hence Theorem \ref{SWThm}$'$ (with $G_1=G_2=G$) implies Theorem \ref{SWThm}. Indeed, if $F,G$ satisfy the hypotheses of Theorem \ref{SWThm}, then $\mathcal{E}_\delta(F)\leq \delta^{-1-\eta}$ and $\mathcal{E}_\delta(G)\leq \delta^{-1-\eta}$. If at least one of these sets has $\delta$-covering number $\leq \delta^{-1+4\eta}$, then $\mathcal{E}_\delta(F\times G\times G)\leq\delta^{2\eta-3}$, and thus Conclusion (B) of Theorem \ref{SWThm} is vacuously true. On the other hand, if both $F$ and $G$ have $\delta$-covering number $\geq\delta^{-1+4\eta}$, then $F$ and $G$ are $(\delta,1,\delta^{-5\eta})$-Frostman sets (is harmless to suppose that $F$ and $G$ are unions of $\delta$-squares, and $\mathcal{H}$ is a union of $\delta$ cubes), and thus the sets $F$ and $G_1=G_2=G$ satisfy the hypotheses of Theorem \ref{SWThm}$'$ with $5\eta$ in place of $\eta$.

\subsubsection{Notation conventions}\label{notationConventionSection}
For the remainder of Section \ref{SWProofSec}, we adopt the convention that if $X,Y$ are unions of $\delta$-cubes in $\RR^n$ (resp.~$\delta$-arcs in $S^1$), then we will only consider subsets $\mathcal{H}\subset X\times Y$ that are unions of sets $S_1\times S_2$, where $S_1\subset X$ and $S_2\subset Y$ are $\delta$-cubes (resp.~$\delta$-arcs). Note that the truth (or falsehood) of Theorem \ref{SWThm}$'$ is not affected by the adoption of this convention. 

If $Z$ is a union of $\delta$ cubes in $\RR^n$, it will sometimes be convenient to think of $Z$ as a set of $\delta$ cubes. For example, we will sometimes identify the (infinite) sets of points $\mathcal{H}\subset X\times Y$ with the corresponding (finite) sets of $\delta$ cubes; thus we we can think of $\mathcal{H}$ as a bipartite graph on the (finite) vertex sets $X$ and $Y$. In this spirit, we will sometimes use $\# Z$ to refer to the number of $\delta$ cubes in $Z$, i.e. $\#Z = \delta^{-n}|Z|$.


\subsection{A key special case: $F\cdot(G_1-G_2)$ is large if $(G_1,G_2)$ have thin tubes}
In this section we will introduce the results and definitions needed to state and prove a key special case of Theorem \ref{SWThm}$'$; this is Proposition \ref{largeDotProductThinTubesProp} below. Proposition \ref{largeDotProductThinTubesProp} says that Conclusion (B) holds (with $\rho=\delta$ and $I = [-2,2]$), provided that the graph $\mathcal{H}$ satisfies a mild uniformity condition, and (far more importantly) that the pair $(G_1,G_2)$ satisfies a strong non-concentration condition on lines. 
\subsubsection{Hypergraphs and uniform density}
The set $\mathcal{H}$ from Conclusion (B) of Theorem \ref{SWThm}$'$ is a 3-regular tripartite hypergraph. Some of our arguments can be stated more simply if each vertex of this this hypergraph is a member of (at least) an average number of edges. The precise definition is as follows.

\begin{defn}
Let $A_1,\ldots,A_k$ be finite sets and let $\mathcal{H}\subset A_1\times\ldots\times A_k$. 
For each nonempty multi-index $I\subset[k]=\{1,\ldots,k\},$ define $A_I = \bigtimes_{i\in I}A_i$. For each $a_I = (a_i)_{i\in I}\in A_I$, define

\[
n_{\mathcal{H}}[a_I] =\{ (h_1,\ldots,h_k)\in \mathcal{H}\colon a_i = h_i\ \textrm{for all}\ i\in I\}.
\]
For example if $\mathcal{H} \subset A_1\times A_2\times A_3$ and $(a_1,a_3)\in A_1\times A_3$, then $n_{\mathcal{H}}[(a_1,a_3)]$ is the set of all edges of the form $(h_1,h_2,h_3)\in \mathcal{H}$ with $h_1=a_1$ and $h_3=a_3$. 
\end{defn}

\begin{defn}
If $\mathcal{H}\subset A_1\times\ldots\times A_k$ and $I\subset[k]$, we define $\pi_I\colon \mathcal{H}\to A_I$ to be the projection onto $A_I$. If $I=\{i\}$ we write $\pi_i$ instead of $\pi_{\{i\}}$. Finally, if $h\in \mathcal{H}$, we write $h_I$ to denote $\pi_I(h)$. 
\end{defn}
\begin{rem}
In some of the sections that follow, we will consider sets $\mathcal{H}\subset F\times G_1\times G_2$. To avoid confusion, we will use the notation $\pi_F\colon \mathcal{H}\to F$ to denote the projection map to $F$, $\pi_{F\times G_1}$ to denote the projection to $F\times G_1$, etc.
\end{rem}

\begin{defn}
Let $0<c\leq 1$. We say a $k$-uniform $k$-partite hypergraph $\mathcal{H}\subset A_1\times\ldots\times A_k$ is \emph{uniformly} $c$-\emph{dense} if for each edge $h\in \mathcal{H}$ and each $I\subset[k]$ we have 
\begin{equation}\label{sizeOfNbhd}
\#n_{\mathcal{H}}[h_I]\geq c\prod_{i\not\in I}\#A_i.
\end{equation}
\end{defn}
This definition is useful because some pigeonholing arguments are simplified when the corresponding graph is uniformly $c$-dense, and if $\mathcal{H}$ is uniformly $c$-dense then its induced subgraphs will still be dense (though possibly not uniformly dense). The next result says that every hypergraph $\mathcal{H}$ has a uniformly $c$-dense refinement, where $c$ is comparable to the average density of $\mathcal{H}$.

\begin{lem}[Hypergraph refinement lemma]\label{hypergraphRefinementLem}
 Let $A_1,\ldots,A_k$ be finite sets, let $\mathcal{H} \subset A_1\times\ldots\times A_k$ be a $k$-uniform hypergraph with density $d = \#\mathcal{H}/\prod_{i=1}^k\#A_i$, and let $0<\eps<1$. Then there exists a uniformly $(2^{-k}\eps d)$-dense subgraph $\mathcal{H}'\subset \mathcal{H}$ with $\#\mathcal{H}'\geq(1-\eps)\#\mathcal{H}$. 
\end{lem}
Lemma \ref{hypergraphRefinementLem} was proved by Dvir and Gopi \cite{DG} in the special case $k=2$ and $\eps = 1/2$. A similar proof works in general; we give the details below.
\begin{proof}
To begin, define $A_I'=A_I$ for each multi-index $I$, and define $\mathcal{H}'=\mathcal{H}$. Next, if there exists a multi-index $I$ and an element $h_I\in A_I'$ with 
\begin{equation}\label{removalCriteria}
\#n_{\mathcal{H}'}[h_I]<\frac{\eps}{2^k}\frac{\#\mathcal{H}}{\prod_{i\in I}\#A_i},
\end{equation} 
then we remove $h_I$ from $A_I'$, and we remove all of the edges $n_{\mathcal{H}'}[h_I]$ from $\mathcal{H}'$. We repeat this process until no $h_I\in A_I'$ fits the removal criteria \eqref{removalCriteria} for any multi-index $I$. 

Each $h_I\in A_I'$ that is removed deletes at most $\frac{\eps}{2^k}\frac{\#\mathcal{H}}{\prod_{i\in I}\#A_i}$ edges from $\mathcal{H}$. Thus the number of edges deleted by this procedure is at most
\[
\sum_{I\subset[k]}\frac{\eps}{2^k}\frac{\#\mathcal{H}}{\prod_{i\in I}\#A_i}\#A_I = \sum_{I\subset[k]}\frac{\eps}{2^k} \#\mathcal{H} = \eps\#\mathcal{H}.
\]
After this procedure has finished, the surviving graph $\mathcal{H}'$ is uniformly $(2^{-k}\eps d)$-dense.
\end{proof}



\begin{rem}
Let $\mathcal{H}\subset A_1\times \ldots\times A_k$ be uniformly $c$-dense, let $1\leq j\leq k$, and suppose that the projection map $\pi_j$ is onto. Let $A_j'\subset A_j$, and define $\tilde A_i=A_i$ for $i\neq j$, and $\tilde A_j = A_j'$. Let $\mathcal{H}' = \mathcal{H}\cap (\tilde A_1\times\ldots \times \tilde A_k),$ be the corresponding induced subgraph. Then $\big(\#\mathcal{H}'\big) / \big(\prod_{i=1}^k \#\tilde A_i\big)\geq c$. 

$\mathcal{H}'\subset \tilde A_1\times \ldots\times \tilde A_k$ need not be uniformly $c$-dense, but by Lemma \ref{hypergraphRefinementLem} it has a uniformly $2^{-k+1}c$-dense subgraph $\mathcal{H}''\subset \mathcal{H}'$ with $\#\mathcal{H}''\geq\frac{1}{2}\#\mathcal{H}'$. We will use this observation frequently in the arguments below. 
\end{rem}


\subsubsection{Kaufman's projection theorem}
The next tool we will need is Kaufman's projection theorem \cite{Kau}. The precise version we will need is given below; we include a proof for completeness.

\begin{thm}\label{thm: Kaufman} 
	Let $0<\alpha\leq\beta\leq 1$ and let $\eps>0$. Then there exists $\delta_0>0$ so that the following holds for all $\delta\in (0,\delta_0]$. Let $C\geq 1$, let $F\subset \mathbb{R}^2$ be a $(\delta, \alpha, C)$-Frostman set, and let $\Lambda\subset S^1$ be a $(\delta, \beta, C)$-Frostman set.  Let $\mathcal{H}\subset \Lambda \times F$, and let $d = |\mathcal{H}| / (|\Lambda|\,|F|)$. Then there exists $\Lambda'\subset \Lambda$ with $|\Lambda'| \gtrsim  d|\Lambda|$ such that for every $\theta\in \Lambda'$, we have
	\[
	|\theta\cdot F_{\theta}| \gtrsim C^{-2}d^3 \delta^{1-\alpha+\eps},
	\]
	where $F_{\theta}=\{ x\in F\colon  (\theta, x)\in H\}$.
\end{thm}
\begin{proof}
Apply Lemma \ref{hypergraphRefinementLem} to $\mathcal{H}$ with $\eps=1/2$, let $\mathcal{H}'\subset \mathcal{H}$ be the resulting refinement, and let $\Lambda_0\subset\Lambda$ be the projection of $\mathcal{H}'$ to the first coordinate. Then $|\Lambda_0|\gtrsim d |\Lambda|$ (so in particular $\Lambda_0$ is a $(\delta, \beta, C/d)$-Frostman set), and for each $\theta\in \Lambda_0$, we have $|F_{\theta}| \gtrsim d |F|$. 

Define $\nu=\frac{\chi_{\Lambda_0}}{|\Lambda_0|}$, $\mu = \frac{\chi_F}{|F|}$, $\mu_{\theta}=\frac{\chi_{F_\theta}}{|F_\theta|}$. If $\theta\in \Lambda_0$ then $\frac{d\mu_{\theta}}{d\mu} \lesssim d^{-1}$. For $x\neq y$, define $\pi^y(x)=\frac{x-y}{|x-y|}$.
Fix $0<\gamma<\alpha$. Since $\mu$ and $\nu$ satisfy the Frostman conditions $\mu(B(x,r))\leq Cr^\alpha$ and $\nu(B(x,r))\leq (C/d)r^\beta$, there is a constant $K$ depending on $\alpha-\gamma$  so that 
\[
\int_F\int_F |x-y|^{-\gamma}d\mu(x)d\mu(y)\leq K C,
\]
and 
\[
\int_{\Lambda_0} |\theta\cdot e|^{-\gamma} d\nu (\theta) \leq K d^{-1}C
\quad\textrm{for every}\ e\in S^1.
\]
In what follows, all implicit constants are allowed to depend on $\alpha-\gamma$. We compute

	\begin{align*}
	 \int_{\Lambda_0} \int_F \int_F  & |\theta\cdot x-\theta\cdot y|^{-\gamma} d\mu_{\theta}(x) d\mu_{\theta}(y) d\nu(\theta) \\
	& \lesssim d^{-2}  \int_F\int_F \int_{\Lambda_0} |\theta\cdot x-\theta\cdot y|^{-\gamma} d\nu(\theta) d\mu(x) d\mu(y)\\
	& =d^{-2} \int_F\int_F \int_{\Lambda_0} |x-y|^{-\gamma} \cdot |\theta\cdot \pi^y(x)|^{-\gamma} d\nu(\theta) d\mu(x) d\mu(y)\\
	&\lesssim C d^{-3} \int_F\int_F |x-y|^{-\gamma} d\mu(x)d\mu(y)\\
	& \lesssim  C^2 d^{-3}. 
	\end{align*}
	
	By Fubini, there exists a subset $\Lambda'\subset \Lambda_0$ with $|\Lambda'| \gtrsim |\Lambda_0|$, so that for every $\theta\in \Lambda'$ we have
	\[
\int_F\int_F |\theta\cdot x-\theta\cdot y|^{-\gamma} d\mu_{\theta}(x) d\mu_{\theta}(y) \lesssim C^2 d^{-3}.
	\]
We conclude that $|\theta\cdot F_{\theta}| \gtrsim C^{-2}d^3 \delta^{1-\gamma}$. The result now follows by selecting $\gamma = \alpha-\eps/2$ and choosing $\delta_0=\delta_0(\eps)>0$ sufficiently small.
\end{proof}

\subsubsection{Thin tubes}
In \cite{SW0}, Shmerkin and the first author introduced a non-concentration condition on pairs of Borel measures in the plane, which prevents the supports of these measures from concentrating on a common line. We recall this definition here.

\begin{defn}\label{thinTubes}
Let $\beta\geq 0$, $K\geq 1$, and $c \in [0,1)$. Let $\nu_1,\nu_2$ be Borel probability measures on $\RR^2$. We say that $(\nu_1,\nu_2)$ have $(\beta,K,1-c)$-\emph{thin tubes} if there exists a Borel set $E\subset \operatorname{supp}(\nu_1)\times \operatorname{supp}(\nu_2)$ with $\nu_1\times\nu_2(E)\geq 1-c$ with the following property: for all $b_1\in \operatorname{supp}(\nu_1)$, all lines $\ell\subset\RR^2$ containing $b_1,$ and all $r>0$, we have 
\begin{equation}\label{frostmanOnThickenedLines}
\nu_2\big( \{b_2\in N_r(\ell)\colon (b_1,b_2)\in E\}\big) \leq K r^\beta.
\end{equation}
\end{defn}

In our setting, we will consider unions of $\delta$-cubes rather than probability measures. To this end, we introduce the following variant of Definition \ref{thinTubes}.

\begin{discretizedThinTubesDefn}
Let $\beta\geq 0$, $K\geq 1$, and $c \in [0,1)$. Let $G_1,G_2\subset \RR^2$ be unions of $\delta$-squares. We say that $(G_1,G_2)$ have $\delta$-discretized $(\beta, K, 1-c)$ \emph{thin tubes} if there exists a union of $\delta$-cubes $E \subset G_1 \times G_2$ with $|E| \geq (1-c)|G_1|\,|G_2|$ with the following property: for all $b_1 \in G_1$, all lines $\ell\subset\RR^n$ containing $b_1,$ and all $r\geq\delta$, we have
\begin{equation}\label{nonConcentrationNearLinesDiscretized}
\big|\{b_2\in G_2\cap N_r(\ell)\colon (b_1,b_2)\in E\}\big| \leq K \cdot r^{\beta}|G_2|.
\end{equation}
We will sometimes suppress the parameters $K$ and $c$, and just say that $(G_1,G_2)$ have $\delta$-discretized $\beta$-thin tubes. 
\end{discretizedThinTubesDefn}

Definitions \ref{thinTubes} and \ref{thinTubes}$'$ are related as follows:
\begin{lem}\label{relateThinVsDiscretizedThinTubes}
Let $G_1,G_2\subset\RR^2$ be unions of $\delta$-squares, and define the probability measures $\nu_i = \chi_{G_i}/|G_i|$. If $(\nu_1,\nu_2)$ has $(\beta, K, 1-c)$-thin tubes, then $(G_1,G_2)$ has $\delta$-discretized $(\beta, K', 1-c)$-thin tubes, with $K'=4^\beta K/c$.
\end{lem}
\begin{proof}
Let $E\subset \operatorname{supp}(\nu_1)\times \operatorname{supp}(\nu_2)$ be a set with measure $\nu_1\times\nu_2(E)\geq 1-c$ that satisfies \eqref{frostmanOnThickenedLines} for all $b_1\in \operatorname{supp}(\nu_1)$, all lines $\ell\subset\RR^2$ containing $b_1,$ and all $r>0$. 

Define $E'$ to be the union of those $\delta$ cubes $S = S_1\times S_2\subset G_1\times G_2$ (each $S$ is a $4$-dimensional $\delta$-cube)  satisfying $\nu_1\times\nu_2(S\cap E)\geq c\ \nu_1\times\nu_2(S)=c\delta^4/|G_1||G_2|$. We have 
\[
|(G_1\times G_2)\backslash E'| = \big(\nu_1\times\nu_2((G_1\times G_2) \backslash E')\big)|G_1||G_2| < c|G_1||G_2|,
\]
and thus
\[
|E'| \geq (1-c)|G_1||G_2|.
\]

We will show that this set $E'$ satisfies \eqref{nonConcentrationNearLinesDiscretized}. Fix a choice of $b_1\in G_1$, a line $\ell\subset\RR^2$ containing $b_1$, and $r\geq\delta$. Our goal is to establish \eqref{nonConcentrationNearLinesDiscretized} with $4^\beta K/c$ in place of $K$.  Let $S_1\subset G_1$ be the $\delta$-square containing $b_1$. Observe that for all $b_1'\in S_1$, we have
\begin{equation}\label{nu2EstimateNearb1}
\nu_2\big(\{b_2\in N_{2r}(\ell)\colon (b_1',b_2)\in E\}\big)\leq K(4r)^\beta.
\end{equation}
Indeed, this is because $N_{2r}(\ell)$ is contained in the $4r$-neighbourhood of a line passing though $b_1'$; we apply \eqref{frostmanOnThickenedLines} to this latter line. 

Let $\mathcal{S}$ be the set of $\delta$-squares intersecting $N_r(\ell)$. We have
\begin{equation}\label{unionSContained2sNbhd}
N_{r}(\ell)\subset \bigcup_{S\in\mathcal{S}}S \subset N_{2r}(\ell).
\end{equation}
We compute
\begin{equation}
\begin{split}
|\{b_2\in G_2\cap N_r(\ell)\colon (b_1,b_2)\in E'\}| & \leq \delta^2\#\{S_2\in\mathcal{S}\colon S_1\times S_2\subset E'\}\\
& = \delta^2\#\{S_2\in\mathcal{S}\colon \nu_1\times\nu_2((S_1\times S_2)\cap E) \geq c\delta^4/|G_1||G_2|\}\\
&\leq c^{-1}\frac{|G_1|}{\delta^2}|G_2| \sum_{S\in\mathcal{S}}\nu_1\times\nu_2((S_1\times S_2)\cap E)\\
&\leq c^{-1}\frac{|G_1|}{\delta^2}|G_2| \int_{S_1}\int_{N_{2r}(\ell)}\chi_E(b_1',b_2)d\nu_2(b_2)d\nu_1(b_1')\\
&\leq c^{-1}\frac{|G_1|}{\delta^2}|G_2| \int_{S_1} K(4r)^\beta d\nu_1(b_1')\\
&\leq \frac{4^\beta K}{c}r^\beta |G_2|.
\end{split}
\end{equation}
The first inequality used the first inequality of \eqref{unionSContained2sNbhd}. The next line follows from the definition of $E'$. The fourth line used the second inequality of \eqref{unionSContained2sNbhd}, and the fifth line used \eqref{nu2EstimateNearb1}.
\end{proof}


\subsubsection{$F\cdot(G_1-G_2)$ is large if $(G_1,G_2)$ have thin tubes}

In several of the results that follow, it will be helpful to suppose that the sets $F,G_1,$ and $G_2$ are localized to small balls, and that these balls are separated from each other and from the origin. We record this assumption below
\begin{defn}
We say that sets $F,G_1,G_2\subset [0,1]^2$ satisfy the \emph{standard separation conditions} if
\begin{equation}\label{wellSeparationCondition}
\operatorname{diam}(F),\ \operatorname{diam}(G_1),\ \operatorname{diam}(G_2)\leq \frac{1}{10};\quad \operatorname{dist}(G_1,G_2)\geq\frac{1}{2};\quad\operatorname{dist}(F,0)\geq\frac{1}{2}.
\end{equation}
\end{defn}

\noindent We can now state and prove the following special case of Theorem \ref{SWThm}$'$.

\begin{lem}\label{largeDotProductThinTubesProp}
For all $\eps>0$, there exists $\delta_0>0$ so that the following holds for all $\delta\in (0,\delta_0]$. Let $c\in [0,1)$ and $K\geq 1$. Let $F,G_1,G_2 \subset[0,1]^2$ be $(\delta, 1, K)$-Frostman sets that satisfy the standard separation conditions, and suppose furthermore that $(G_1,G_2)$ has  $\delta$-discretized $(1, K, 1-c)$-thin tubes. Let $\mathcal{H}\subset F\times G_1 \times G_2$ be uniformly $2c$-dense.  Then
\begin{equation}
\mathcal{E}_{\delta}\big( \{a\cdot(b_1-b_2)\colon (a,b_1,b_2)\in \mathcal{H}\}\big) \gtrsim c^5K^{-2} \delta^{\eps-1}.
\end{equation}
\end{lem}
\begin{proof}
Since $(G_1,G_2)$ has $\delta$-discretized $(1, K, 1-c)$-thin tubes, there exists a set $E\subset G_1\times G_2$ with $|E|\geq (1-c)|G_1|\,|G_2|$ that satisfies \eqref{nonConcentrationNearLinesDiscretized} with $\beta=1$ and $K$ as above. Since $\mathcal{H}$ is $2c$ dense, we have $|\pi_{G_1\times G_2}(\mathcal{H})|\geq 2c|G_1|\,|G_2|$. Thus $|E\cap \pi_{G_1\times G_2}(\mathcal{H})|\geq (2c-c)|G_1|\,|G_2|\geq c|G_1|\,|G_2|$; denote this new set by $E'$. 

 By Fubini, there exists $b_1\in G_1$ so that $| \pi_{G_2}(n_{E'}[(b_1)])| \geq c |G_2|$. Fix this choice of $b_1$ and define $G_2' =  \pi_{G_2}(n_{E'}[(b_1)])$. Let $L$ be a maximal set of lines containing $b_1$ with $\delta$-separated directions, each of which intersects $G_2'$. By \eqref{nonConcentrationNearLinesDiscretized} with $\beta=1$ and $r=\delta$, we have $\#L \geq \frac{ |G_2'|}{K\delta |G_2|}\geq \frac{c}{K}\delta^{-1}.$ Let $\Lambda\subset S^1$ be the set of (grid-aligned) $\delta$-arcs that intersect the directions of lines in $L$. Then $\Lambda$ is a $(\delta, 1, K/c)$-Frostman set, and for each $\theta\in\Lambda$ there exists (at least one) $b_2\in G_2$ so that $(b_1,b_2)\in \pi_{G_1\times G_2}(\mathcal{H})$; $|b_1-b_2|\geq\frac{1}{2}$; and $\frac{b_1-b_2}{|b_1-b_2|}=\theta$. Choose such a $b_2\in G_2$ for each $\theta\in\Lambda$, and let $H\subset\Lambda\times F$ be the set of pairs $(\theta,a)$ so that $(a,b_1,b_2)\in \mathcal{H}$. Since $\mathcal{H}$ is $2c$ dense, for each $\theta\in\Lambda$ there are at least $2c(\# F)$ choices of ($\delta$-separated) $a\in F$ with $(\theta,a)\in H$. Thus $|H|\geq 2c |\Lambda|\,|F|$. 

We now apply Theorem \ref{thm: Kaufman} to $\Lambda$, $F$, and $H$, with $\alpha=\beta=1$ and $\eps$ as above. Let $\delta_1>0$ and $\Lambda'\subset\Lambda$ be the output from that lemma. Select $\theta\in\Lambda'$, let $(b_1,b_2)\in \pi_{G_1\times G_2}(\mathcal{H})$ be the pair of points associated to $\theta$, and let $r=|b_1-b_2|\geq\frac{1}{2}$. We have
\begin{equation}
\begin{split}
\{a\cdot (b_1-b_2)\colon (a,b_1,b_2)\in \mathcal{H}\}& \supset \{a\cdot\theta r \colon a\in (F)_\theta \}.
\end{split}
\end{equation}
By Theorem \ref{thm: Kaufman} the latter set has $\delta$-covering number $\gtrsim c^5K^{-2}\delta^{\eps-1}$.
\end{proof}

\begin{rem}
Note that an analogue of Lemma \ref{largeDotProductThinTubesProp} is also true over $\CC$, i.e.~this result does not distinguish between the fields $\RR$ and $\CC$. 
\end{rem}

\subsection{ Radial projection: $(G_1, G_2)$ have thin tubes unless they are concentrated on lines}
The main result of this section is the following result, which says that the conclusion of Lemma \ref{largeDotProductThinTubesProp} continues to hold if we replace the hypothesis that $(G_1,G_2)$ have 1-thin tubes with the hypothesis that $G_1$ and $G_2$ satisfy a weak non-concentration near lines estimate.

\begin{prop}\label{weakNonConcentrationOnLinesImpliesConclusionBProp}
For all $0<\eps,\zeta<1$, there exists $\lambda=\lambda(\eps)>0$, $\alpha=\alpha(\eps,\zeta)>0$, and $\delta_0=\delta_0(\eps,\zeta)>0$ so that the following holds for all $\delta\in (0,\delta_0]$. %
Let $F,G_1,G_2 \subset[0,1]^2$ be $(\delta, 1, \delta^{-\lambda})$-Frostman sets that satisfy the standard separation conditions.
Suppose that $G_1$ and $G_2$ satisfy the following non-concentration near lines estimate:
\begin{equation}\label{nonConcentratedOnLinesInProp}
|G_i\cap N_r(\ell)|\leq (\delta^{-\lambda}r)^{\zeta}|G_i|\quad\textrm{for all lines}\ \ell\ \textrm{and all}\ r\geq\delta,\ \quad i=1,2.
\end{equation}
Let $\mathcal{H}\subset F \times G_1 \times G_2$ be uniformly $\delta^{\alpha}$-dense. Then
\begin{equation}\label{nonConcentrationNearLinesGivesManyDotsEqn}
\mathcal{E}_{\delta}\big( \{a\cdot(b_1-b_2)\colon (a,b_1,b_2)\in \mathcal{H}\}\big) \geq \delta^{\eps-1}.
\end{equation}
\end{prop}

Proposition \ref{weakNonConcentrationOnLinesImpliesConclusionBProp} is essentially a small variant of the radial projection theorem in  \cite{OSW};  we reprove it here for completeness.  The proof of  radial projection theorem \cite{OSW} uses the ``bootstrapping lemma'' that first appeared in \cite{SW0}, which in turn uses a result from \cite{OS21} that obtains an $\eps$-improved Furstenberg set estimate. This latter result made crucial use of Bourgain's discretized sum-product theorem \cite{bour, bour2}. In particular, the analogue of Proposition \ref{weakNonConcentrationOnLinesImpliesConclusionBProp} is not true if we replace the field $\RR$ by $\CC$. The ``bootstrapping lemma'' we state here is Lemma 2.9 from \cite{OSW}.

\begin{prop}[\cite{OSW}, Lemma 2.9]\label{bootstrapLem}
Let $\beta,\eps>0$, $\sigma\in[\beta,1-\eps],$ $c\in(0,1/10)$, and $K,C\geq 1$. Then there exists $\tau = \tau(\beta,\eps)>0$ and $M(\beta,\eps)\geq 1$ so that the following holds.
Let $\nu_1,\nu_2$ be Borel probability measures on $\RR^2$ with $\operatorname{dist}(\operatorname{supp}(\nu_1),\ \operatorname{supp}(\nu_2))\geq 1/2$ that satisfy the Frostman non-concentration estimate
\[
\nu_1(B(x,r))\leq Cr\qquad\textrm{and}\qquad \nu_2(B(x,r))\leq Cr,
\] 
for all balls $B(x,r)\subset\RR^2$. Suppose furthermore that both $(\nu_1,\nu_2)$ and $(\nu_2,\nu_1)$ have $(\sigma, K, 1-c)$-thin tubes. 

Then both $(\nu_1,\nu_2)$ and $(\nu_2,\nu_1)$ have $(\sigma+\tau, K', 1-3c)$-thin tubes, where
\[
K' = \max\big\{K, \frac{C^2 M}{c} \big\}^{M}.
\]
\end{prop}
\begin{rem}
The statement of Lemma 2.9 in \cite{OSW} differs slightly from the version given above. First, Proposition \ref{bootstrapLem} considers the special case $s=1$ from \cite[Lemma 2.9]{OSW}, and thus $\sigma<1$; this allows us to simply a few expressions. Second, \cite[Lemma 2.9]{OSW} specifies that $\sigma\in[\beta,1)$, and the corresponding value of $\tau$ is bounded away from 0 on compact subsets of $\{\beta\leq \sigma< 1)$. We have fixed the compact subset $\{\beta\leq\sigma\leq 1-\eps\}$ in advance, and thus $\tau = \tau(\beta,\eps)>0$. 

Most importantly, the parameter $\beta$ in the statement of \cite[Lemma 2.9]{OSW} is not permitted to be chosen freely. Instead, \cite[Lemma 2.9]{OSW} specifies a particular value of $\beta$ (a careful reading of \cite{Orp} shows that $\beta = 1/2$, though this fact is not necessary for our arguments here). Thus the version of Proposition \ref{bootstrapLem} stated above appears to be superficially stronger than what is stated in \cite[Lemma 2.9]{OSW}. However, examining the (short) proof  of \cite[Lemma 2.9]{OSW}, it does not actually require a specific choice of $\beta$; any $\beta>0$ works equally well (though the value of $\tau$ will depend on $\beta$). 
For our arguments, we will fix $\beta = 1/4$.
\end{rem}

In the discretized setting (and fixing the choice $\beta = 1/4$), Proposition \ref{bootstrapLem} can be restated as follows:

\begin{bootstrapLemDiscretized}
Let $\eps>0$, $\sigma\in[1/4,1-\eps],$ $c\in(0,1/10)$, and $K,C\geq 1$. Then there exists $\tau = \tau(\eps)>0$ and $M(\eps)\geq 1$ so that the following holds.
Let $G_1,G_2\subset[0,1]^2$  be $(\delta, 1,C)$-Frostman sets, with  $\operatorname{dist}(G_1,\ G_2)\geq 1/2$. Suppose that both $(G_1,G_2)$ and $(G_2,G_1)$ have $\delta$-discretized $(\sigma, K, 1-c)$-thin tubes. 

Then both $(G_1,G_2)$ and $(G_2,G_1)$ have $\delta$-discretized $(\sigma+\tau, K', 1-3c)$-thin tubes, where
\[
K' = \max\big\{\frac{4K}{3c}, \frac{4C^2 M}{3c^2} \big\}^{M}.
\]
\end{bootstrapLemDiscretized}
Indeed, we obtain Proposition \ref{bootstrapLem}$'$ by applying Proposition \ref{bootstrapLem} to the probability measures $\nu_1 = \chi_{G_1}/|G_1|$ and $\nu_2 = \chi_{G_2}/|G_2|$ and then using Lemma \ref{relateThinVsDiscretizedThinTubes} (this final step worsens $K'$ by a factor of $4/(3c)$).

By repeatedly applying Proposition \ref{bootstrapLem}$'$ we obtain the following.

\begin{cor}\label{quarterThinTubesCor}
For all $\eps>0$, there exists $N = N(\eps)\geq 1$ so that the following holds.
Let $c\in(0,1/10)$, $K,C\geq 1$. Let $G_1,G_2\subset[0,1]^2$  be $(\delta, 1,C)$-Frostman sets, with  $\operatorname{dist}(G_1, G_2)\geq 1/2$. Suppose that both $(G_1,G_2)$ and $(G_2,G_1)$ have $\delta$-discretized $(1/4, K, 1-c)$-thin tubes. 

Then both $(G_1,G_2)$ and $(G_2,G_1)$ have $\delta$-discretized $(1-\eps, K', 1-3^N c)$-thin tubes, and thus have $\delta$-discretized $(1, \delta^{-\eps}K', 1-3^Nc)$-thin tubes, where
\[
K' = \max\big\{\frac{4K}{3c}, \frac{4C^2 N}{3c^2} \big\}^N.
\]
\end{cor}

\subsubsection{Grabbing the bootstraps}
Corollary \ref{quarterThinTubesCor} partially bridges the gap between Lemma \ref{largeDotProductThinTubesProp} and Proposition \ref{weakNonConcentrationOnLinesImpliesConclusionBProp}. Corollary \ref{quarterThinTubesCor} says that $(G_1,G_2)$ have $\delta$-discretized 1-thin tubes, which is the condition needed to apply Lemma \ref{largeDotProductThinTubesProp}. Unfortunately, the hypothesis \eqref{nonConcentratedOnLinesInProp} of Proposition \ref{weakNonConcentrationOnLinesImpliesConclusionBProp} only implies that $(G_1,G_2)$ and $(G_2,G_1)$ have $\delta$-discretized $\delta^\zeta$-thin tubes, and if $\zeta<1/4$ then we cannot immediately apply Corollary \ref{quarterThinTubesCor}.

To fix this problem, we will first prove that the pair $(G_1,G_2)$ from Proposition \ref{weakNonConcentrationOnLinesImpliesConclusionBProp} has $1/4$-thin tubes; this is Lemma \ref{lem: liushen} below. We will then be able to apply Corollary \ref{quarterThinTubesCor} and Lemma \ref{largeDotProductThinTubesProp}. The precise result we need is as follows. Results of this type first appeared in  \cite{Orp} (see also \cite{LS} and \cite[Appendix B]{Shmerkin}), but our setup is slightly different, so we will include a proof. Note that an analogue of the arguments below would also work over $\CC$, i.e.~the argument below does not distinguish between the fields $\RR$ and $\CC$.

\begin{lem}\label{lem: liushen}
Let $\lambda,\zeta,\alpha>0$. Then there exists $\delta_0>0$ so that the following is true for all $\delta\in (0,\delta_0]$. Let $G_1,G_2\subset[0,1]^2$ be $(\delta, 1, \delta^{-\lambda})$-Frostman sets with $\text{dist}(G_1, G_2)\geq 1/2$, and suppose that $G_2$ satisfies the non-concentration near lines estimate
\begin{equation}\label{nonConcentrationOnLinesForG2}
|G_2\cap N_r(\ell)|\leq (\delta^{-\lambda}r)^{\zeta}|G_2|\quad\textrm{for all lines}\ \ell\ \textrm{and all}\ r\geq\delta.
\end{equation}
Then $(G_1,G_2)$ have $\delta$-discretized $(\frac{1}{4}, K, 1-\delta^{\alpha})$-thin tubes, where $K = \delta^{-\frac{3\alpha}{\zeta}-\lambda}.$
\end{lem}
\begin{proof}
Fix $\beta = 1/4$; our goal in Lemma \ref{lem: liushen} is to prove that $(G_1,G_2)$ have $\delta$-discretized $\beta$-thin tubes. The exponent $\beta=1/4$ is chosen for concreteness, but it is not crucial for our argument---any exponent $\beta < 1/2$ would work equally well. We will use $\beta$ in place of $1/4$ to distinguish this quantity from other less important constants that occur elsewhere in the proof. 

Let $\delta\leq r\leq 1$; we say a pair of $\delta$-squares $(Q,S)\subset G_1\times G_2$ is \emph{bad} at scale $r$ if there is an infinite strip $T$ of width $r$ that intersects $Q$ and $S$, with $|T\cap G_2|\geq K r^{\beta}|G_2|$. Observe that no pairs of $\delta$-squares are bad at any scale greater than $K^{-1/\beta}$. 
Define
\[
E_{\operatorname{bad}} = \bigcup_{\substack{\delta\leq r\leq K^{-4}\\ r\ \textrm{dyadic}}} \bigcup_{(Q,S)\ \textrm{bad}} Q\times S,
\]
where the first union is taken over all numbers $\delta\leq r\leq K^{-1/\beta}$ of the form $2^k\delta$, $k\in\ZZ$, and the second union is taken over all pairs $(Q,S)$ of $\delta$-squares that are bad at scale $r$. Define $E = (G_1\times G_2)\backslash E_{\operatorname{bad}}$. By construction, each $b_1\in G_1$ satisfies \eqref{nonConcentrationNearLinesDiscretized} for this set $E$, with $K$ as above. Clearly $E$ is a union of $\delta$-cubes. To prove the lemma, it remains to show that for $K$ as above, we have $|E|\geq (1-\delta^\alpha)|G_1|\,|G_2|$. To do this, it suffices to show that for each dyadic $\delta\leq r\leq K^{-1/\beta}$, 
\begin{equation}\label{boundOnNumberBadPairs}
\Big| \bigcup_{(Q,S)\ \textrm{bad}} Q\times S\Big| \leq \frac{\delta^\alpha}{|\log\delta|}|G_1|\,|G_2|.
\end{equation}

Fix $r$. For each square $Q\subset G_1$, let $G_2^{Q}\subset G_2$ be the union of those squares $S\subset G_2$ so that $(Q,S)$ is bad at scale $r$. Define $G_1'$ to be the union of those squares $Q\subset G_1$ for which $|G_2^Q|\leq \frac{\delta^\alpha}{2|\log\delta|}|G_2|$, and let $G_2'' = G_2\backslash G_2'$. The contribution to \eqref{boundOnNumberBadPairs} from pairs $(Q,S)$ with $Q\subset G_1'$ is at most $\frac{\delta^\alpha}{|\log\delta|}|G_1|\,|G_2|$, which is acceptable. 

It remains to analyze the contribution from those pairs $(Q,S)$ with $Q\subset G_1''$, and in particular it suffices to show that 
\begin{equation}\label{boundOnG1pp}
|G_1''|\leq \frac{\delta^\alpha}{2|\log\delta|}|G_1|.
\end{equation} 
We will prove this by contradiction. Suppose that \eqref{boundOnG1pp} fails. For each square $Q\subset G_1''$, select a set $\mathcal{T}_Q$ of width-$r$ strips pointing in $r$-separated directions, each of which intersect $Q$ and each of which satisfy 
\begin{equation}\label{eq: 1/4thintube}
	|T\cap G_2|\geq K r^{\beta}|G_2|.
	\end{equation}  We have $|G_2\cap \bigcup_{T\in\mathcal{T}_Q}T|\gtrsim\frac{\delta^\alpha}{|\log\delta|}|G_2|,$ and since $G_1$ and $G_2$ are $1/2$-separated, each point $b_2\in G_2$ is contained in at most 10 of the strips from $\mathcal{T}_Q.$ If $Q\subset G_1''$ and $S\subset G_2$, write $Q\sim S$ if $S$ intersects one of the strips from $\mathcal{T}_Q$. Define
\[
\mathcal{W}=\{(Q,Q',S)\colon Q\sim S,\ Q'\sim S\}.
\]
If \eqref{boundOnG1pp} fails, then by Cauchy-Schwarz we have 
\begin{equation}\label{lowerBoundOnCalT}
\#\mathcal{W}\gtrsim \Big(\frac{\delta^\alpha}{|\log\delta|}\Big)^4(\#G_1)^2(\#G_2)
\end{equation} 
(recall that we use $\#G_1$ to denote the number of $\delta$-squares in $G_1$, and similarly for $\#G_2$). We will show that for $K$ as above, then \eqref{lowerBoundOnCalT} is impossible. 

First we will bound the number of triples $(Q,Q',S)\in\mathcal{W}$ with $\operatorname{dist}(Q,Q')\leq c_0 \delta^{\lambda}\delta^{-4\alpha}|\log\delta|^4$, where $c_0>0$ is a small absolute constant. Since $G_1$ is a $(\delta, 1, \delta^{-\lambda})$-Frostman set, after $Q$ has been chosen there are at most $c_0\delta^{4\alpha}|\log\delta|^{-4}(\#G_2)$ choices for $Q'$, and thus the number of such triples is at most $c_0\delta^{4\alpha}|\log\delta|^{-4}(\#G_1)^2(\#G_2)$. 

Next we will bound the number of triples $(Q,Q',S)\in\mathcal{T}$ where $S$ is contained in the $s=\delta^{\lambda}\big(c_0 \frac{\delta^\alpha}{|\log\delta|}\big)^{4/\zeta}$ neighborhood of the line $\ell_{Q,Q'}$ that connects the centers of $Q$ and $Q'$. By \eqref{nonConcentrationOnLinesForG2},  after $Q$ and $Q'$ have been chosen there are at $(\delta^{-\lambda}s)^{\zeta}\#G_2=c_0^4 \big(\frac{\delta^\alpha}{|\log\delta|}\big)^4\#G_2$ choices for $S$, and thus the number of such triples is at most $c_0^4 \big(\frac{\delta^\alpha}{|\log\delta|}\big)^4(\#G_1)^2(\#G_2)$. 

Finally, we will bound the number of remaining triples $(Q,Q',S)\in\mathcal{W}$. For each such triple, there are $r$-tubes $T\in\mathcal{T}_Q$ and $T'\in\mathcal{T}_{Q'}$ with $\angle(T,T')\gtrsim s$ (by $\angle(T,T')$ we mean the angle between their coaxial lines), so that $S$ intersects $T\cap T'$. Thus for each pair $Q,Q'\subset G_1''$, the corresponding set of squares $S\subset G_2$ intersect the set
\[
\bigcup_{\substack{ (T,T')\in \mathcal{T}_Q\times \mathcal{T}_{Q'} \\ \angle(T,T')\gtrsim s}}\!\!\!\!\!\!\!\! T\cap T'.
\]
Each set $T\cap T'$ in the above union is contained in a ball of radius $\sim s^{-1}r$, and since $G_2$ is a $(\delta, 1, \delta^{-\lambda})$-Frostman set, each set $T\cap T'$ can intersect at most $\delta^{-\lambda}s^{-1}r(\#G_2)$ squares from $G_2$. By \eqref{eq: 1/4thintube}, there are at most $(\#\mathcal{T}_Q)(\#\mathcal{T}_{Q'})\leq r^{-2\beta}$ such pairs $(T,T')$, and hence when $\delta$ is sufficiently small,  the total number of triples of this type is at most
\[
(\#G_1)^2\big(\delta^{-\lambda}s^{-1}r(\#G_2)\big)r^{-2\beta} \leq \delta^{-\lambda}s^{-1}K^{-2} (\#G_1)^2(\#G_2)\leq  c_0 \Big(\frac{\delta^\alpha}{|\log\delta|}\Big)^4(\#G_1)^2(\#G_2),
\]
where here we recall that $\beta = 1/4$ (and thus the exponent $r^{1-2\beta}$ is positive), and we used the fact that $r\leq K^{-1/\beta} = K^{-4}$ and $K=\delta^{-\frac{3\alpha}{\zeta}-\lambda},\, \,  s=  \delta^{\lambda}\big(c_0 \frac{\delta^\alpha}{|\log\delta|}\big)^{4/\zeta}$. We conclude that if $\delta>0$ is sufficiently small, then $\#\mathcal{W}\leq 3 c_0 \Big(\frac{\delta^{\alpha}}{|\log\delta|}\Big)^4(\#G_1)^2(\#G_2)$. If $c_0>0$ is chosen sufficiently small, then this contradicts \eqref{lowerBoundOnCalT}, and completes the proof. 
\end{proof}

\begin{rem}
It is tempting to avoid Lemma \ref{lem: liushen} by instead iterating Proposition \ref{bootstrapLem}$'$ multiple times starting with $\sigma = \beta = \zeta$ rather than $\sigma = \beta = 1/4$. Unfortunately, this would require $O_{\zeta,\eps}(1)$ iterations, and this in turn would mean that the quantity $\lambda=\lambda(\eps)$ from Proposition \ref{weakNonConcentrationOnLinesImpliesConclusionBProp} would also have to be sufficiently small  depending on $\zeta$ (in particular, much smaller than $\zeta$). This is not acceptable, because in our application below we must use a value of $\lambda$ that is at least as large as $\zeta$. 
\end{rem}

We can now combine Lemma \ref{largeDotProductThinTubesProp}, Corollary \ref{quarterThinTubesCor}, and Lemma \ref{lem: liushen} to prove Proposition \ref{weakNonConcentrationOnLinesImpliesConclusionBProp}. The details are as follows.
\begin{proof}[Proof of Proposition \ref{weakNonConcentrationOnLinesImpliesConclusionBProp}]
Let $\eps,\zeta\in(0,1)$. Let $\lambda=\lambda(\eps),$ $\alpha=\alpha(\eps,\zeta)$, and $\delta_0=\delta_0(\eps,\zeta)$ be specified below. Let $\delta\in(0,\delta_0]$ and let $F,G_1,G_2,\mathcal{H}$ satisfy the hypotheses of Proposition \ref{weakNonConcentrationOnLinesImpliesConclusionBProp}. 

Let $\alpha'>0$ be a quantity to be chosen below. Applying Lemma \ref{lem: liushen} with this choice of $\alpha'$, and with $\zeta,\lambda$ as above, we conclude that $(G_1,G_2)$ and $(G_2,G_1)$ have $\delta$-discretized $(\frac{1}{4}, \delta^{-\frac{3\alpha'}{\zeta}-\lambda}, 1-\delta^{\alpha'})$-thin tubes.

Applying Corollary \ref{quarterThinTubesCor} with $\eps/16$ in place of $\eps$, we conclude that there exists $N = N(\eps)$ so that $(G_1,G_2)$ and $(G_2,G_1)$ have $\delta$-discretized $(1, K, 1-3^N \delta^{\alpha'})$-thin tubes, where
\begin{equation}\label{boundOnKForTubes}
K = \max\{4\delta^{-\frac{4\alpha'}{\zeta}-\lambda},\ 4\delta^{-2\lambda-2\alpha'}N\}^N \leq (4N)^N\delta^{-\frac{4N\alpha'}{\zeta}-2N\lambda}.
\end{equation}
We will select
\[
\lambda = \frac{\eps}{100 N},\quad \alpha' = \frac{\eps\zeta}{100 N}.
\]
Since $N = N(\eps)$ depends only on $\eps$, we have that $\lambda$ depends only on $\eps$, as required. $\alpha'$ depends on $\eps$ and $\zeta$. With these choices, the RHS of \eqref{boundOnKForTubes} is at most $\delta^{-\eps/8}$, provided $\delta_0>0$ is selected sufficiently small.

Let $\alpha=\alpha'/2$, and select $\delta_0$ sufficiently small (depending on $\alpha'$ and $N$, which in turn depend on $\eps$ and $\zeta$) so that $\delta_0^\alpha \geq 2 \cdot 3^N\delta_0^{\alpha'}$, and thus $\delta^\alpha \geq 2 \cdot 3^N\delta^{\alpha'}$. Applying Lemma \ref{largeDotProductThinTubesProp} with $c = 3^N\delta^{\alpha'}$, $K = \delta^{-\eps/8}$, and $\eps/2$ in place of $\eps$, we conclude (provided $\delta_0$ is selected sufficiently small) that
\[
\mathcal{E}_{\delta}\big( \{a\cdot(b_1-b_2)\colon (a,b_1,b_2)\in \mathcal{H}\}\big) 
\geq 3^{5N}\delta^{5\alpha'}\delta^{\eps/4} \delta^{\eps/2-1}
\geq \delta^{\eps-1},
\]
where the final inequality used the fact that $5\alpha' = \frac{5\eps\zeta}{100 N}\leq \eps/5$.
\end{proof}


\subsection{Theorem \ref{SWThm}$'$ for well separated sets and uniform hypergraphs}
Our goal in this section is to prove the following special case of Theorem \ref{SWThm}$'$.
\begin{prop}\label{wellSeparatedAndUniformProp}
For all $\eps>0$, there exists $\eta,\delta_0>0$ so that the following holds for all $\delta\in(0,\delta_0]$. Let $F,G_1,G_2\subset[0,1]^2$ be $(\delta,1,\delta^{-\eta})$-Frostman sets that satisfy the standard separation conditions. Let $\mathcal{H}\subset F\times G_1\times G_2$ be uniformly $\delta^{\eta}$-dense. Then at least one of the following must hold:
\begin{enumerate}[label=(\Alph*)]
	\item\label{assertion1ItemA} There exist two lines $\ell$  and  $\ell^\perp$ (where $\ell^\perp$ passes through the origin and is orthogonal to $\ell$) that satisfy \eqref{largeIntersectionRecPrime}:
	\[ 
\mathcal{E}_\delta(N_\delta(\ell)\cap F)\geq \delta^{\eps-1},\quad \mathcal{E}_\delta(N_\delta(\ell^\perp)\cap G_1)\geq \delta^{\eps-1},\quad
\mathcal{E}_\delta(N_\delta(\ell^\perp)\cap G_2)\geq \delta^{\eps-1}.		\tag{\ref{largeIntersectionRecPrime}}
	\]

\item\label{assertion1ItemB} There exists $\rho\geq\delta$ and an interval $I$ of length at least $\delta^{-\eta}\rho$ that satisfies \eqref{bigRCoveringNumberPrime}:
\[
\mathcal{E}_\rho\big(I \cap \{a\cdot(b_1-b_2)\colon (a,b_1,b_2)\in \mathcal{H}\} \big) \geq \big(|I|/\rho\big)^{1-\eps} \tag{\ref{bigRCoveringNumberPrime}}.
\]
\end{enumerate}
\end{prop}


\subsubsection{A non-concentration condition near lines, and anisotropic rescaling}
We will prove Proposition \ref{wellSeparatedAndUniformProp} by reducing it to Proposition \ref{weakNonConcentrationOnLinesImpliesConclusionBProp}. One key difference between these two results is that Proposition \ref{weakNonConcentrationOnLinesImpliesConclusionBProp} assumes the non-concentration estimate \eqref{nonConcentratedOnLinesInProp}. First, a standard ``two-ends reduction'' argument shows that every set $E\subset[0,1]^2$ has a large subset that satisfies a re-scaled analogue of \eqref{nonConcentratedOnLinesInProp} localized inside a rectangle of some width $w$. A precise version is given below. The proof is standard, and is omitted. 

\begin{lem}\label{twoEndsOnLines}
Let $E\subset[0,1]^2$ be a union of $\delta$-squares and let $0<\zeta\leq 1/4$. Then there exists a number $w>0$ and a line $\ell$, so that if we define $E'=E\cap N_w(\ell)$, then $|E'|\geq \frac{1}{2}w^\zeta|E|$, and for all $0<r\leq 1$ and all lines $\ell$, we have 
\begin{equation}\label{twoEndsInsideRect}
|E'\cap N_r(\ell)| \leq (r/w)^{\zeta}|E'|.
\end{equation}
\end{lem}

If we anisotropically re-scale the output set $E'\subset  N_w(\ell)$ from Lemma \ref{twoEndsOnLines} by a factor of $w^{-1}$ in the $l^\perp$ direction, then the non-concentration estimate \eqref{twoEndsInsideRect} from Lemma \ref{twoEndsOnLines} will resemble the non-concentration hypothesis \eqref{nonConcentratedOnLinesInProp} from Proposition \ref{wellSeparatedAndUniformProp}. The problem is that if $E$ was a $(\delta,1,\delta^{-\eta})$-Frostman set, then $E'$ will be a $(\delta,1,\delta^{-\eta-\zeta})$-Frostman set, but the image of $E'$ under this anisotropic rescaling might not be a $(\delta,1,C)$-Frostman set for any reasonable value of $C$. The next lemma says that we can fix this problem at the cost of refining $E'$ slightly.

\begin{lem}\label{lem: rescaling}
For all $\eps>0$, there exists $\delta_0>0$ so that the following holds for all $\delta\in (0,\delta_0]$. Let $w\in [\delta^{1-\eps},1]$, let $R\subset \mathbb{R}^2$ be a rectangle of dimensions $1\times w$, and let $E\subset R$ be a $(\delta, 1, C)$-Frostman set for some $C\geq 1$. Let $\phi\colon\RR^2\to\RR^2$ be an affine transform that maps $R$ to $[0,1]^2$. Then there exists a subset $E'\subset E$ with $|E'|\geq (\delta/w)^{\eps} |E|$ so that the union of $\delta/w$ squares that intersect $\phi(E')$ form a $(\delta/w, 1, (w/\delta)^{\eps}C)$--Frostman set and each $\delta/w$-square has about the same size of intersection with $\phi(E')$. 
\end{lem}

\begin{proof}
Let $\mathcal{R}_0$ be the tiling of $R$ by rectangles of dimensions $\delta\times \delta/w$ pointing in the same direction as $R$ (we will suppose for simplicity that $1/\delta$ and $w/\delta$ are integers; if not, then we can shrink $R$ slightly on two sides and prune the set $E$; this step is harmless). By pigeonholing, we can select sets $E'\subset E$ and $\mathcal{R}\subset\mathcal{R}_0$, and a number $m>0$, so that $|E'|\gtrsim|\log\delta|^{-1}|E|$, and $|\tilde R\cap E'| \in [ m, 2m)$ for each $\tilde R\in\mathcal{R}$ with $\tilde R \cap E'\neq\emptyset$. 

Let $F\subset[0,1]^2$ be the union of $\delta/w$ squares that intersect $\phi(E')$. It remains to verify that $F$ is a $(\delta/w, 1, (w/\delta)^{\eps}C)$-Frostman set. Let $r\geq \delta/w$, let $B = B(x,r)$ be a ball, and let $3B$ be its 3-fold dilate. We have
\begin{equation}\label{ballBoundF}
\begin{split}
\frac{|B\cap F|}{|F|} &= \frac{|\phi^{-1}(B)\cap \phi^{-1}(F)|}{|\phi^{-1}(F)|} \\
& \leq \frac{\#\{\tilde R\in\mathcal{R}\colon \tilde R \cap \phi^{-1}(B)\neq\emptyset \}}{\#\mathcal{R}}\\
& \leq \frac{|\phi^{-1}(3B) \cap E'|}{|E'|}\\
& \lesssim \log(1/\delta) \frac{|\phi^{-1}(3B) \cap E|}{|E|}\\
&\leq 3\log(1/\delta) rC.
\end{split}
\end{equation}
For the final inequality, we used the fact that $\phi^{-1}(3B)$ is contained in a ball of radius $3r$, and $E$ is a $(\delta, 1, C)$-Frostman set. Since $(w/\delta)\geq\delta^{-\eps}$, we have $3\log(1/\delta)\leq (w/\delta)^\eps$ provided $\delta_0$ is selected sufficiently small.
\end{proof}

We will sometimes need to move from scale $\delta$ to a coarser scale. The next lemma says that after a refinement, $(\delta,\alpha, C)$-Frostman sets are well-behaved under coarsening.

\begin{lem}\label{coarseningLem}
Let $0<\delta\leq\rho\leq 1$ and let $E\subset \mathbb{R}^2$ be a $(\delta, \alpha, C)$-Frostman set for some $C\geq 1$. Then there is a set $E'\subset E$ with $|E'|\geq|\log\delta|^{-1}|E|$, so that the union of $\rho$ squares intersecting $E'$ is a $(\rho, \alpha, C')$-Frostman set, with $C'\sim |\log\delta|C$. 
\end{lem}

\begin{proof}
Tile $\mathbb{R}^2$ by $\rho$-squares. By dyadic pigeonholing, we can select $E'\subset E$ and $m>0$ so that for each $\rho$-square $Q$ in the tiling, either $|E'\cap Q| = \emptyset$ or $|E'\cap Q|=m$. Let $E_\rho$ be the union of $\rho$-squares that intersect $E'$. Then $|E_\rho| = (|E'|/m)\rho^2\leq |\log\delta|^{-1}(|E'|/m)\rho^2$. Let $B$ be a ball of radius $r\geq\rho$, and let $B'$ be its $2$-fold dilate.  Then $|B'\cap E'|\leq |B'\cap E|\leq C (2r)^\alpha|E|$. In particular, there are at most $C4 r^\alpha|E|/m$ $\rho$-squares that are contained in $B'\cap E$, and hence at most $4C r^\alpha|E|/m$ such cubes that intersect $C\cap E$. We conclude that
\[
|E_\rho \cap B|\lesssim \rho^2 C r^\alpha|E|/m \lesssim  |\log\delta|C  r^\alpha|E_\rho|.\qedhere
\]
\end{proof}


\subsubsection{Localization to rectangles of comparable width}
We now consider the following situation. Suppose we have sets $F,G_1,G_2,\mathcal{H}$ that satisfy the hypotheses of Proposition \ref{wellSeparatedAndUniformProp}, and suppose that $G_1$ and $G_2$ are contained in strips of width $w_1$ and $w_2$, respectively. The next lemma says that either Conclusion \ref{assertion1ItemB} of Proposition \ref{wellSeparatedAndUniformProp} is true, or $w_1$ and $w_2$ are almost identical; $G_1$ and $G_2$ are contained in a common strip; and $F$ is contained in an orthogonal strip of comparable thickness.

\begin{lem}\label{A1A3Localized}
For all $\eps>0$, there exists $\eta,\delta_0>0$ so that the following holds for all $\delta\in (0,\delta_0].$ Let $F,G_1,G_2 \subset[0,1]^2$ be $(\delta,1,\delta^{-\eta})$-Frostman sets that satisfy the standard separation conditions, and let $\mathcal{H}\subset F\times G_1\times G_2$ be uniformly $\delta^{\eta}$-dense. Suppose $G_1\subset N_w(\ell)$ for some $w>0$ and some line $\ell$. Let $\ell^\perp$ be the line through the origin that is perpendicular to $\ell$.

Then at least one of the following two things must happen:
\begin{enumerate}[label=(\Alph*)]
	\item\label{A1A3LocalizedItemA} $\pi_{G_2}(\mathcal{H})\subset N_{\delta^{-\eps}w}(\ell)$ and $\pi_{F}(\mathcal{H})\subset N_{\delta^{-\eps}w}(\ell^\perp)$.

	\item\label{A1A3LocalizedItemB} There exists $\rho\geq\delta$ and an interval $I$ of length at least $\delta^{-\eta}\rho$ that satisfies \eqref{bigRCoveringNumberPrime}:
\[
\mathcal{E}_\rho\big(I \cap \{a\cdot(b_1-b_2)\colon (a,b_1,b_2)\in \mathcal{H}\} \big) \geq \big(|I|/\rho\big)^{1-\eps}. \tag{\ref{bigRCoveringNumberPrime}}
\]
\end{enumerate}
\end{lem}

\begin{proof}
First, we may suppose that $w\leq\delta^{\eps}$, or else Conclusion \ref{A1A3LocalizedItemA} holds and we are done.

Let $0<\eta<\!\!<\eps_1<\!\!<\eps$ be small quantities that will be chosen below. Without loss of generality, we can suppose that the maps $\pi_F\colon \mathcal{H}\to F,\ \pi_{G_1}\colon\mathcal{H}\to G_1$, and $\pi_{G_2}\colon\mathcal{H}\to G_2$ are onto (if not, replace $F,G_1,G_2$ by the image of $\mathcal{H}$ under the corresponding projection map; this reduces the size of each set by at most a multiplicative factor of $\delta^{\eta}$, and hence the new sets are $(\delta,1,\delta^{-2\eta})$-Frostman sets). Let $t=\inf\{r \colon G_1\cup G_2 \subset N_r(\ell)\}$, and let $U = N_t(\ell)$. Clearly $t\geq\delta$, since $G_1\cup G_2$ contains at least one $\delta$-square. Let $U^\perp=N_{\delta^{-\eps_1}t}(\ell^\perp).$ Let $e$ be a unit vector parallel to $\ell$.

\medskip
\noindent {\bf Step 1. Trapping $F$ and $G_1\cup G_2$ inside thick orthogonal strips}\\
First, we will show that either $F\subset U^\perp$, or Conclusion \ref{A1A3LocalizedItemB} holds. If $t\geq\delta^{-\eps_1}$ then $F\subset[0,1]^2\subset U^\perp$ and we are done. Thus we will suppose $t<\delta^{-\eps_1}$. 
Suppose $F\not\subset U^\perp$, and fix an element $a \in F\backslash U^\perp.$ Since $F$ satisfies \eqref{wellSeparationCondition}, we have $\operatorname{dist}(a, 0)\geq\frac{1}{2}$, and hence $|a\cdot e|\geq \frac{1}{4}\delta^{-\eps_1}t$. In particular, if $b_1,b_1'\in G_1$ are $\geq 2t$ separated, then $|e\cdot(b_1-b_1')|\geq\frac{1}{4}|b_1-b_1'|$ and hence
\begin{equation}\label{separationQpQpp}
|a\cdot (b_1-b_1')| \geq |a\cdot e|\,|e\cdot(b_1-b_1')| - |a\cdot e^\perp|\,|e^\perp \cdot(b_1-b_1')| \geq \frac{1}{4}|a\cdot e|\,|b_1-b_1'| - t.
\end{equation}
Let $s = 8t/|a\cdot e|$. Since $a\in [0,1]^2$, we have $|a\cdot e|\leq 2$ and hence $s\geq 4t$, and in particular $s>\delta$. We also have $s\leq 8t/(\frac{1}{4}\delta^{-\eps_1}t)\leq 32\delta^{\eps_1}$. Observe that if $|b_1-b_1'|\geq s$, then by \eqref{separationQpQpp} we have
\begin{equation}\label{separationQpQppEqn2}
|a\cdot (b_1-b_1')|\geq \frac{1}{8}|a\cdot e|\,|b_1-b_1'|.
\end{equation}
Fix an element $b_2\in G_2$ such that $n_{\mathcal{H}}[(a, b_2)]$ is nonempty, and let $G_1'\subset G_1$ be a maximal $s$-separated subset of $\pi_{G_1}( n_\mathcal{H}[(a,b_2)])$. Since $\mathcal{H}$ is uniformly $\delta^{\eta}$-dense, $\pi_{G_1}( n_{\mathcal{H}}[(a,b_2)])$ has measure at least $\delta^{1-\eta}$, and since $G_1$ is a $(\delta,1,\delta^{-\eta})$-Frostman set, we conclude that $\#G_1'\geq\delta^{2\eta}s^{-1}$. 
Thus by \eqref{separationQpQppEqn2}, if we define $\rho = 2t = |a\cdot e|s/4$, then 
\[
\mathcal{E}_\rho\big(\{ a\cdot (b_1-b_2) \colon b_1\in G_1 \}\big) \geq  \mathcal{E}_\rho\big(\{ a\cdot b_1 \colon b_1\in G_1' \}\big) 
\gtrsim\# G_1' \gtrsim\delta^{2\eta}s^{-1}.
\]
Furthermore, $\{ a\cdot (b_1-b_2)\colon b_1\in G_1 \}$ is contained in the interval $J = a\cdot \big([-1,2]^2 \cap (U - b_2)\big)$, which has length $|J|\sim |a\cdot e| \sim \rho/s  \gtrsim\delta^{-\eps_1} \rho\geq\delta^{-\eta}\rho$. We conclude that
\begin{align*}
\mathcal{E}_\rho\big(J \cap \{ a\cdot (b_1-b_2)\colon (a,b_1,b_2)\in \mathcal{H}\}\big) & 
\geq \mathcal{E}_\rho\big(\{ a\cdot(b_1-b_2)\colon b_1 \in G_1' \}\big)\\
&\gtrsim\delta^{2\eta}s^{-1} \gtrsim \delta^{2\eta}\frac{|J|}{\rho}\gtrsim \delta^{2\eta - \eps_1\eps}\Big(\frac{|J|}{\rho}\Big)^{1-\eps}.
\end{align*}
If $\eta\leq \eps_1\eps/2$ and $\delta_0$ is selected sufficiently small, then Conclusion \ref{A1A3LocalizedItemB} holds. 

\medskip
\noindent {\bf Step 2. $G_1-G_2$ spans many directions inside $U$}\\
At this point we may suppose that $F\subset U^\perp$. If $t\leq \delta^{-\eps+\eps_1}w$, then Conclusion \ref{A1A3LocalizedItemA} holds and we are done. Henceforth we will assume that $t> \delta^{-\eps+\eps_1}w$. We will prove that Conclusion \ref{A1A3LocalizedItemB} holds. Fix a point $b_2\in G_2\,\backslash\, N_{t/2}(\ell)$---such a point must exist since otherwise we would have $t\leq 2w$. In particular, we have $t/2\leq \operatorname{dist}(b_2, N_w(\ell))\leq t$. If $b_1,b_1'\in G_1$ with $|b_1-b_1'|\geq 10 w/t$, then 
\[
|\pi^{b_2}(b_1)-\pi^{b_2}(b_1')| \geq \frac{1}{10}t|b_1-b_1'|,
\]
where $\pi^{b_2}(x) = \frac{x-b_2}{|x-b_2|}$ denotes the radial projection from the vantage point $b_2$. Indeed, this follows from the standard separation conditions, see Figure~\ref{angleseparation}. 

\begin{figure}[h]
	\centering
	\begin{overpic}[scale=0.2]{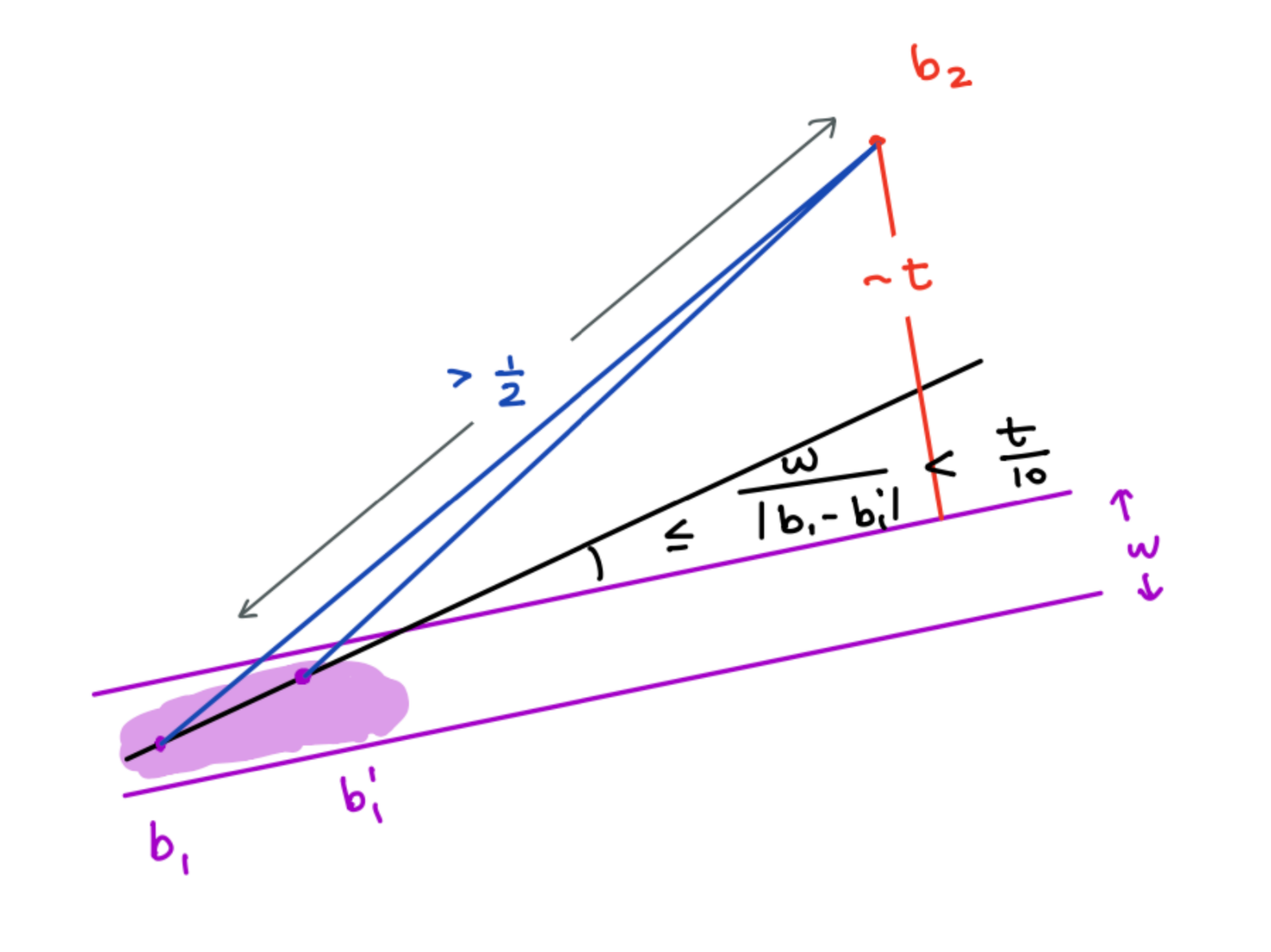}
	\end{overpic}
	\caption{The purple shaded region is $G_1$, which is contained in $N_{w}(\ell)$. The points $b_1, b_1'\in G_1$ have distance $\geq 1/2$ to $b_2\in G_2$, which has distance $\sim t$ to $N_{w}(\ell)$.}
	\label{angleseparation}
\end{figure}

Recall that $G_1$ is contained in a ball $B$ of radius $1/10$, which has distance at least $1/2$ from $b_2$. Thus $G_1\subset B\cap N_w(\ell)$, and $\pi^{b_2}(B\cap N_w(\ell))$ is contained in an arc $J\subset S^1$ of length $|J|\sim t$.

Let $u = 10 w/t$, and let $G_1'\subset G_1$ be a maximal $u$-separated subset of $\pi_{G_1}(n_{\mathcal{H}}[(b_2)])$. Since $\mathcal{H}$ is uniformly $\delta^{\eta}$-dense, $\pi_{G_1}(n_{\mathcal{H}}[(b_2)])$ has measure at least $\delta^{\eta}|G_1|$. Since $G_1$ is a $(\delta,1,\delta^{-\eta})$-Frostman set, we conclude that $\#G_1'\geq \delta^{2\eta}u^{-1}$. Thus we have $\pi^{b_2}(G_1')\subset J$, and $\pi^{b_2}(G_1')$ is dense inside $J$ at scale $w$, in the sense that
\begin{equation}\label{lotsOfDirectionsInsideJ}
\mathcal{E}_w\big (\pi^{b_2}(G_1')\big)\gtrsim \#G_1'\gtrsim \delta^{2\eta}u^{-1}\gtrsim \delta^{2\eta} t w^{-1}.
\end{equation}

Let $X = \pi_{F\times G_1} \big(n_{\mathcal{H}}[(b_2)]\big) \cap (F\times G_1')$; this is the set of pairs $(a,b_1)\in F\times G_1'$ so that $(a,b_1,b_2)\in \mathcal{H}$. Since $\mathcal{H}$ is uniformly $\delta^{\eta}$-dense, we have $\# X \geq\delta^{\eta}(\#F)(\#G_1')$ (recall our notation convention from Section \ref{notationConventionSection}). Apply Lemma \ref{hypergraphRefinementLem} to $X$, with $\eps=1/2$, let $X'\subset X$ be the output, and let $F'=\pi_F(X')$; thus for each $a\in F'$, there are $\gtrsim \delta^{\eta}(\#G_1')$ elements $b_1\in G_1'$ with $(a,b_1)\in X$. 

Let $t' = \delta^{-\eps_1}t$. Let $\tilde R$ be a $2\times t'$ rectangle pointing in direction $e$ that contains $G_1\cup G_2$, and let $\tilde R^\perp$ be a $2\times t'$ rectangle pointing in direction $e^\perp$ that contains $F$. (The existence of such a rectangle $\tilde R$ is easy; just select a suitable subset of $N_{t'}(\ell)$. The existence of $\tilde R^\perp$ is guaranteed by the arguments in Step 1, which says $F\subset U^{\perp}=N_{t'}(\ell^{\perp})$). Let $\phi \colon \RR^2\to\RR^2$ be an affine-linear transformation that sends $\tilde R$ to $[0,1]^2$, and let $\phi^T$ be the transpose of $\phi$; then $\phi^T(\tilde R^\perp)$ is a unit square. 

Apply Lemma \ref{lem: rescaling} to $F\subset \tilde R^\perp$ with $\eta$ as above, and let $F'\subset F$ be the output of that lemma.  Let $\tilde F$ be the union of $\delta/t'$ squares that intersect $\phi^T(F')$; then $\tilde F$ is a $(\delta/t', 1, \delta^{-2\eta})$-Frostman set. Applying Lemma \ref{coarseningLem} with $\delta/t'$ in place of $\delta$ and $w/t'$ in place of $\rho$ (this is valid since $w/t'\leq\delta^\eps<1$), we can find a subset  $Y\subset \tilde F$, so that if $\tilde{F}'$ denotes the set of $w/t'$-squares that intersect $Y$, then $\tilde{F}'$ is a $(w/t', 1, \delta^{-3\eta})$-Frostman set.

Since $|\pi^{b_2}(b_1) - \pi^{b_2}(b_1')| \sim t'|\pi^{\phi(b_2)}(\phi(b_1)) - \pi^{\phi(b_2)}(\phi(b_1'))|$,  by \eqref{lotsOfDirectionsInsideJ} we have
\[
\mathcal{E}_{w/t'}\Big( \pi^{\phi(b_2)}(\phi(G_1')) \Big) \gtrsim \mathcal{E}_w\big (\pi^{b_2}(G_1')\big) \gtrsim \delta^{2\eta} t w^{-1}=\delta^{2\eta+\eps_1} t'w^{-1}.
\]
Thus if $\delta_0$ is chosen sufficiently small, then the radial projection $\Lambda = \pi^{\phi(b_2)}(\phi(G_1'))$ is a $(w/t', 1, \delta^{-3\eta-\eps_1})$-Frostman set. Each $\tilde a \in\tilde{F}'$ can be written as $\tilde a= (\phi^{T})^{-1}(a)$ for some $a\in F'$.  Hence there are $\gtrsim\delta^{\eta}\# G_1'$ unit vectors $\theta\in\Lambda$, so that $\theta = \pi^{\phi(b_2)}(\phi(b_1))$ for some $b_1\in G_1'$ with $(a,b_1)\in X'$ (recall we have fixed $b_2$ at the beginning). In particular, $X'$ has density $d\geq \delta^\eta$.

This is precisely the setting for Theorem \ref{thm: Kaufman}. Applying Theorem \ref{thm: Kaufman} with $\eps/2$ in place of $\eps$, $d=\delta^\eta$ and $C = \delta^{3\eta+\eps_1}$, we conclude that there exists a unit vector $\theta\in\Lambda$ of the form $\theta = \pi^{\phi(b_2)}(\phi(b_1))$, so that if $(\tilde{F}')_\theta$ denotes the set of $\tilde a$ (with preimage $a$) with $(a,b_1)\in X'$, then
\[
\mathcal{E}_{w/t'}(\theta\cdot (\tilde{F}_3')_\theta)\gtrsim \delta^{9\eta+2\eps_1} (t'/w)^{1-\eps/2},
\]
and hence
\[
\mathcal{E}_{w/t'}\big(\{ \phi^T(a) \cdot (\phi(b_1)-\phi(b_2))\colon a \in \pi_F( n_{\mathcal{H}}[(b_1,b_2)])  \}\big) \gtrsim \delta^{9\eta+2\eps_1} (t'/w)^{1-\eps/2}.
\]
But 
\[
\mathcal{E}_{w/t'}\big(\{  \phi^T(a)\cdot (\phi(b_1)-\phi(b_2)) \colon a\in \pi_F( n_{\mathcal{H}}[(b_1,b_2)]) \}\big) = \mathcal{E}_{w}\big(\{ a\cdot (b_1-b_2) \colon a\in \pi_F( n_{\mathcal{H}}[(b_1,b_2)]) \}\big).
\]
The latter set is contained in $\tilde R^\perp\cdot (\tilde R - \tilde R)$, which is an interval $I$ of length about $t'$. Setting $\rho = w$, we conclude that $I$ has length $t'=\delta^{-\eps_1}t\geq\delta^{-\eps}\rho$, and 
\begin{align*}
\mathcal{E}_{\rho}\big(I \cap \{ a\cdot (b_1-b_2)\colon (a,b_1,b_2)\in \mathcal{H} \}\big) & \gtrsim \delta^{9\eta+2\eps_1} (t'/\rho)^{1-\eps/2}\\
&\gtrapproxdelta \delta^{9\eta+2\eps_1} (t'/\rho)^{\eps-\eps/2}\Big(\frac{|I|}{\rho}\Big)^{1-\eps}\\
&\geq \delta^{9\eta+2\eps_1-\eps^2/2}\Big(\frac{|I|}{\rho}\Big)^{1-\eps}.
\end{align*}
If $\eps_1$ is selected sufficiently small (depending on $\eps$) then $9\eta+2\eps_1-\eps^2/2<0$. Selecting $\delta_0$ sufficiently small, we have that Conclusion \ref{A1A3LocalizedItemB} holds. 
\end{proof}


\subsubsection{The second reduction}
Armed with the above lemmas, we are now ready to prove Proposition \ref{wellSeparatedAndUniformProp}. In brief, we apply Lemma \ref{twoEndsOnLines} to localize $G_1$ and $G_2$ to strips $S_1$ and $S_2$. We then use Lemma \ref{A1A3Localized} to show that these strips must coincide, and furthermore, that $F$ is localized to an orthogonal strip of approximately the same width, up to a factor of $\delta^{-\eps}$. We then re-scale these strips using Lemma \ref{lem: rescaling} and apply Proposition \ref{weakNonConcentrationOnLinesImpliesConclusionBProp} to the corresponding re-scaled sets. We now turn to the details.

\begin{proof}[Proof of Proposition \ref{wellSeparatedAndUniformProp}]
Let $\eps>0$. Let $\lambda = \lambda(\eps/2)$ be the output from Proposition \ref{weakNonConcentrationOnLinesImpliesConclusionBProp} with $\eps/2$ in place of $\eps$. Define
\begin{equation}\label{defnEps1}
\eps_1=\eps\lambda/30.
\end{equation}
 Let $\eta_1=\eta_1(\eps_1)$ be the output from Lemma \ref{A1A3Localized} with $\eps_1$ in place of $\eps$. Define
\begin{equation}\label{defnZeta}
\zeta = \min\big( \eta_1/2,\ \eps \lambda/3\big).
\end{equation}
Let $\alpha = \alpha(\eps/2, \zeta)$ be the output from Proposition \ref{weakNonConcentrationOnLinesImpliesConclusionBProp} with $\zeta$ as defined in \eqref{defnZeta} and $\eps/2$ in place of $\eps$. Define
\begin{equation}\label{defnEta}
\eta = \min\big(\eps_1\zeta,\ \eps \alpha/40\big),
\end{equation}
We will be slightly less precise about our choice of $\delta_0$, since this plays a minor role. 

Let $\delta\in (0,\delta_0]$ and let $F,G_1,G_2,\mathcal{H}$ satisfy the hypotheses of Proposition \ref{wellSeparatedAndUniformProp} for these values of $\eta$ and $\delta$. We must show that at least one of Conclusion \ref{assertion1ItemA} or \ref{assertion1ItemB} from Proposition \ref{wellSeparatedAndUniformProp} is true for $\eps$ as above. 

First, apply ``two-ends reduction'' Lemma \ref{twoEndsOnLines} to $G_1$, with $\zeta$ as defined in \eqref{defnZeta}; we obtain a width $w_1$, a line $\ell_1$, and a set $G_1'=G_1\cap N_{w_1}(\ell_1)$. Apply Lemma \ref{hypergraphRefinementLem} to the induced subgraph $\mathcal{H}\cap (F\times G_1'\times G_2)$ with $\eps=1/2$; we obtain a uniformly $c$-dense hypergraph $\mathcal{H}'$ with $c\geq 2^{-4}\delta^{\eta}$. Replacing $G_2$ by its image under $\pi_{G_2}(\mathcal{H}')$, we may suppose that $\pi_{G_2}\colon \mathcal{H}' \to G_2$ is onto; this step might reduce the size of $G_2$ by a multiplicative factor of $c$. Next, apply Lemma \ref{twoEndsOnLines} (with $\zeta$ as defined in \eqref{defnZeta}) to $G_2$; we obtain a width $w_2$, a line $\ell_2$, and a set $G_2'=G_2\cap N_{w_2}(\ell_2)$. Apply Lemma \ref{hypergraphRefinementLem} to the induced subgraph $\mathcal{H}'\cap (F\times G_1'\times G_2')$ with $\eps=1/2$; we obtain a uniformly $c$-dense hypergraph $\mathcal{H}''$ with $c\geq 2^{-8}\delta^{\eta}$. Let $G_1''=\pi_{G_1}(\mathcal{H}'')$, and let $F''=\pi_{F}(\mathcal{H}'')$. 

After this step has been completed, the sets $F'',G_1'',G_2''$ are $(\delta,1, \delta^{-2\zeta}$)-Frostman sets; $\mathcal{H}''$ is uniformly $\delta^{-8}\delta^{\eta}$ regular, and $G_1'',G_2''$ satisfy the following Frostman-type non-concentration on lines estimate:
\begin{equation}\label{frostmanOneAipp}
|G_i''\cap N_r(\ell)| \leq 8 \delta^{-\eta} (r/w_i)^\zeta |G_i''|\quad\textrm{for all lines}\ \ell\ \textrm{and all}\ \delta\leq r\leq 1,\quad i=1,2.
\end{equation}

For notational convenience, we will suppose that $w_1\geq w_2$ (this is harmless since $G_1''$ and $G_2''$ play (anti) symmetric roles). By \eqref{defnZeta} and \eqref{defnEta}, we can apply Lemma \ref{A1A3Localized} with $\eps_1$ in place of $\eps$ (once with $G_1''$ and $G_2''$ as stated, and once with their roles reversed) to conclude that at least one of the following is true. 
\begin{itemize}
\item Conclusion \ref{assertion1ItemB} of Proposition \ref{wellSeparatedAndUniformProp} holds (for $\eps_1$, and hence also for $\eps$).
\item $w_1\leq \delta^{-\eps_1}w_2$; $G_1''\cup G_2''$ is contained in a strip $S$ of width $w = \delta^{-\eps_1}w_2$ (and hence contained in an appropriately chosen rectangle $R$ of dimensions $2\times w$); and $F''$ is contained in an orthogonal rectangle $R^\perp$ of dimensions $2\times w$. 
\end{itemize}
If the first item holds then we are done. Suppose instead that the second item holds. If $w\leq \delta^{1-\eps/10}$, then by pigeonholing we can find orthogonal lines that satisfy \eqref{largeIntersectionRecPrime} for the value of $\eps$ specified above, and hence Conclusion \ref{assertion1ItemA} of Proposition \ref{wellSeparatedAndUniformProp} holds.

Suppose instead that $w>\delta^{1-\eps/10}$. Since $\eta<\eps_1\zeta$, \eqref{frostmanOneAipp} can be re-written as 
\begin{equation}\label{frostmanOnAip}
|G_i''\cap N_r(\ell)| \lesssim \Big(\frac{\delta^{-2\eps_1} r}{w}\Big)^{\zeta}|G_i''|.
\end{equation}

Let $\phi\colon\RR^2\to\RR^2$ be an affine-linear transformation sending $R$ to $[0,1]^2$ and let $\phi^T$ be its transpose; then $\phi^T(R^\perp)$ is a translate of the unit square. We also have
\begin{equation}\label{distortionOfAAT}
a \cdot(b_1-b_2) = w \phi^T(a)\cdot (\phi(b_1)-\phi(b_2))\quad \textrm{for all}\ (a,b_1,b_2)\in R^\perp\times R^2 .
\end{equation}
Let $\tau = \delta/w \leq  \delta^{\eps/10}$. We would like to apply Proposition \ref{weakNonConcentrationOnLinesImpliesConclusionBProp} to the sets $\phi^T(F'')$, $\phi(G_1''),$ and $\phi(G_2'')$ (and the corresponding image of $\mathcal{H}''$) at scale $\tau$ to conclude that Conclusion \ref{assertion1ItemB} of Proposition \ref{wellSeparatedAndUniformProp} must hold. The issue is that while $F'',$ $G_1''$, and $G_2''$ are $(\delta,1,\delta^{-2\zeta})$-Frostman sets, their images under $\phi$ (resp.~$\phi^T$) might not be $(\tau,1,\delta^{-2\zeta})$-Frostman sets. We can fix this by sequentially applying Lemma \ref{lem: rescaling} and Lemma \ref{hypergraphRefinementLem} to each of the sets $F'',$ $G_1''$, and $G_2''$. We now turn to the details. 

We apply Lemma \ref{lem: rescaling} to $G_1''\subset R$, with $\eps_1\eta/10$ in place of $\eps$; we let $G_1'''$ be the resulting subset. We then apply Lemma \ref{hypergraphRefinementLem} to the induced subgraph $\mathcal{H}_1 = \mathcal{H}'''\cap (F''\times G_1'''\times G_2'')$, with $\eps = 1/2$; denote the output by $\mathcal{H}_1'$. Next we apply Lemma \ref{lem: rescaling} to $G_2''\subset R$, with $\eps_1\eta/10$ in place of $\eps$; we let $G_2'''$ be the resulting subset. We then apply Lemma \ref{hypergraphRefinementLem} to the induced subgraph $\mathcal{H}_2 = \mathcal{H}_1'\cap (F''\times G_1'''\times G_2''')$, with $\eps = 1/2$; denote the output by $\mathcal{H}_2'$. Finally we apply Lemma \ref{lem: rescaling} to $F''\subset R^\perp$, with $\eps_1\eta/10$ in place of $\eps$; we let $F'''$ be the resulting subset. We then apply Lemma \ref{hypergraphRefinementLem} to the induced subgraph $\mathcal{H}_3 = \mathcal{H}_2'\cap (F'''\times G_1'''\times G_2''')$, with $\eps = 1/2$; denote the output by $\mathcal{H}_3'$.

Let $\tilde F$  (resp.~$\tilde G_1$, $\tilde G_2$) be the set of $\tau$-squares that intersect $\phi^\perp(F''')$ (resp. $\phi(G_2''')$, $\phi(G_1''')$), let $\tilde{\mathcal{H}}_0$ be the set of $\tau$-cubes that intersect $(\phi\otimes \phi\otimes \phi^T)(\mathcal{H}_3')$, and let $\tilde{\mathcal{H}}_3$ be the refinement obtained by applying Lemma \ref{hypergraphRefinementLem} to $\tilde{\mathcal{H}}_0$ with $\eps=1/2$. We have that $\tilde{\mathcal{H}}_3$ is uniformly $\delta^{4\eta} \geq \tau^{40\eta/\eps}$-dense, and $\tilde G_1,\tilde G_2,\tilde F$ are $(\tau, 1, \tau^{-3\zeta/\eps})$-Frostman sets. 

After these refinements, \eqref{frostmanOnAip} now says that for all $r\in [\delta,w]$ and all lines $\ell$,
\begin{equation}\label{nonConcentrationInAippp}
|G_i'''\cap N_r(\ell)| \leq \Big(\frac{\delta^{-3\eps_1}r}{w}\Big)^{\zeta}|G_i'''|,\quad i=1,2.
\end{equation}
Furthermore, since each $\delta\times \delta/w$ rectangle that intersects $G_1'''$ (resp.~$G_2'''$) has the same size of intersection with $G_1'''$ (resp.~$G_2'''$), \eqref{nonConcentrationInAippp} implies that for all $r\in [\tau,1]$ and all lines $\ell$,
\[
|\tilde G_i\cap N_r(\ell)|\leq (\delta^{-3\eps_1}r)^{\zeta}|\tilde G_i|\leq  (\tau^{-30\eps_1/\eps}r)^{\zeta}|\tilde G_i|,\quad i=1,2.
\]

Let us now verify that the sets $\tilde G_1,\tilde G_2,\tilde F$ and $\tilde{\mathcal{H}}$ satisfy the hypotheses of Proposition \ref{weakNonConcentrationOnLinesImpliesConclusionBProp} with $\eps/2$ in place of $\eps$, and $\tau$ in place of $\delta$. The estimate  \eqref{nonConcentratedOnLinesInProp} holds with $\zeta$ as above, provided $30\eps_1/\eps\leq \lambda$; this is guaranteed by \eqref{defnEps1}. We have that $\tilde G_1,\tilde G_2,$ and $\tilde F$ are $(\tau,1,\tau^{-\lambda})$-Frostman sets, provided $3\zeta/\eps\leq\lambda$; this is guaranteed by \eqref{defnZeta}. Finally, $\tilde{\mathcal{H}}$ is uniformly $\tau^\alpha$-dense, provided $40\eta/\eps\leq \alpha$; this is guaranteed by \eqref{defnEta}. We conclude that there exists $\tau_0$ (depending on $\eps$ and $\zeta$, which in turn depends only on $\eps$) so that 
\begin{equation}\label{largeTauCovering}
\mathcal{E}_\tau(\{\tilde a\cdot (\tilde b_1-\tilde b_2)\colon(\tilde a,\tilde b_1,\tilde b_2)\in \tilde{\mathcal{H}}\})\geq \tau^{\eps/2-1},
\end{equation}
provided $\tau\leq\tau_0$. We can ensure that $\tau\leq\tau_0$ by selecting $\delta_0 = \tau_0^{1/\eps}$. 

Comparing \eqref{largeTauCovering} and \eqref{distortionOfAAT}, we see that there is an interval $I$ of length $w> \delta^{1-\eps/10}$ so that
\[
\mathcal{E}_\delta\big(I \cap \{a\cdot(b_1-b_2)\colon (a,b_1,b_2)\in \mathcal{H}\}\big)\gtrsim \tau^{\eps/2-1}=(|I|/\delta)^{1-\eps/2}.
\]
If $\delta_0>0$ is sufficiently small, then Conclusion \ref{assertion1ItemB} of Proposition \ref{wellSeparatedAndUniformProp} holds.
\end{proof}


\subsection{A final reduction}
In this section we will show that Proposition \ref{wellSeparatedAndUniformProp} implies Theorem \ref{SWThm}$'$. Compared with Theorem~\ref{SWThm}$'$,  the difference between these two results is that in Proposition~\ref{wellSeparatedAndUniformProp}, the sets $F, G_1, G_2$ satisfy the standard separation conditions. 
\begin{proof}[Proof of Theorem \ref{SWThm}$'$]
Let $\eps>0$. Let $\eta_1,\delta_1$ be the output from Proposition \ref{wellSeparatedAndUniformProp}, with $\eps/2$ in place of $\eps$. Let $\eta = \eta_1/100$, $\delta_0 = \delta_1$. Let $F,G_1,G_2, \mathcal{H}$ satisfy the hypotheses of Theorem \ref{SWThm}$'$ with this value of $\eta$ and some $\delta\in(0,\delta_0]$. We must show that at least one of the Conclusions \ref{SWThmItemAPrimed} or \ref{SWThmItemBPrimed} from Theorem \ref{SWThm}$'$ hold for this value of $\eps$ and $\eta$.

Let $\mathcal{B}$ be a cover of $[0,1]^6$ by sets of the form ${\bf B} = B_{F}\times B_{G_1}\times B_{G_2}$, where $B_{F},B_{G_1},B_{G_2}\subset\RR^2$ are balls of radius $\delta^{6\eta}$; we can construct such a cover with $\#\mathcal{B}\lesssim \delta^{-36\eta}.$ This induces a decomposition $\mathcal{H}=\bigcup_{\mathcal{B}}\mathcal{H}_{{\bf B}},$ where $\mathcal{H}_{{\bf B}}=\mathcal{H}\cap(B_{F}\times B_{G_1}\times B_{G_2})$.  Since $F,G_1,G_2$ are $(\delta,1, \delta^{-\eta})$-Frostman sets, we have 
\[
\#\{(a,b_1,b_2)\in \mathcal{H} \colon |a|\geq\delta^{5\eta},\ |b_1-b_2| \geq \delta^{5\eta}\}\geq\frac{1}{2}\#\mathcal{H}. 
\]
Thus by pigeonholing, there exists an element ${\bf B}=B_{F}\times B_{G_1}\times B_{G_2}\in\mathcal{B}$ with $\operatorname{dist}(B_{F}, 0)\geq  \delta^{5\eta}-\delta^{6\eta}$ and $\operatorname{dist}(B_{G_1},B_{G_2})\geq \delta^{5\eta}-\delta^{6\eta}$ so that $\# \mathcal{H}_{{\bf B}}\geq \frac{1}{2}\#\mathcal{H}/\#\mathcal{B}\gtrsim \delta^{36\eta}\# \mathcal{H}$. Let $F'=F\cap B_F,$ $G_1' = G_1\cap B_{G_1},$ and $G_2' = G_2\cap B_{G_2}$. After applying a common translation $\phi$ to the sets $G_1'$ and $G_2'$ (this translation is harmless since it preserves expressions of the form $b_1-b_2$ with $b_1\in G_1'$ and $b_2\in G_2'$; we may also assume it sends $\delta$-squares to $\delta$-squares), we may assume that $G_1'$ and $G_2'$ are contained in $B(0,r),$ with $r = \operatorname{dist}(B_{G_1},B_{G_2})\gtrsim\delta^{5\eta}$.  Let $G_1''$ (resp $G_2''$) be the image of $G_1'$ under the map $\psi_{G} \colon (x,y)\mapsto(x/r, y/r)$. Then $G_1''$ and $G_2''$ are $(\delta, 1, \delta^{-12\eta})$-Frostman sets that are contained in balls of radius $\leq 1/100$ that are $1/2$-separated. Similarly, let $F''$ be the image of $F'$ under the map $\psi_F \colon (x,y)\mapsto(x/r', y/r')$, with $r' = \operatorname{dist}(B_{G_F}, 0)\gtrsim\delta^{5\eta}$; then $F''$ is a $(\delta, 1, \delta^{-12\eta})$-Frostman set that is contained in a ball of radius $\leq 1/100$ that has distance $\geq 1/2$ from the origin. 

Let $\phi_G=\psi_G\circ\phi \colon \RR^2\to\RR^2$ be the translation and dilation described above that was applied to $G_1'$ and $G_2'$.

 Define $\mathcal{H}'=\psi_F\otimes \phi_G\otimes \phi_G(\mathcal{H}_{{\bf B}})$. The sets $F'',G_1'',G_2''$ and $\mathcal{H}'$ might no longer be unions of $\delta$-cubes, but it is harmless to replace them with the unions of $\delta$-cubes that they intersect. The sets $F'',G_1'',G_2''$ now satisfy the standard separation conditions \eqref{wellSeparationCondition}.

Apply Lemma \ref{hypergraphRefinementLem} with $\eps = \delta^{20\eta}$ to $\mathcal{H}'$, and let $\mathcal{H}''$ denote the resulting refinement; this set is uniformly $\delta^{40\eta}$-dense.  We now apply Proposition \ref{wellSeparatedAndUniformProp} to $\mathcal{H}''\subset F''\times G_1''\times G_2''$ with $\eps/2$ in place of $\eps$; we have that $F'',G_1'',G_2''$ satisfy the standard separation conditions \eqref{wellSeparationCondition}, and since each of $F'', G_1''$, and $G_2''$ is a $(\delta,1,\delta^{-12\eta})$-Frostman set, it is also a $(\delta,1,\delta^{-100\eta})$-Frostman set. We have that $\mathcal{H}''$ is uniformly $\delta^{100\eta}$-dense. 

Suppose Conclusion \ref{assertion1ItemA} of Proposition \ref{wellSeparatedAndUniformProp} is true, i.e.~there are  lines $\ell_0$,  and  $\ell_0^\perp$ (where $\ell_0^\perp$ passes through the origin and is perpendicular to $\ell_0$) with 
\[
|N_\delta(\ell_0^\perp)\cap F''|\geq \delta^{\eps/2}|F|,\quad |N_\delta(\ell_0)\cap G_1''|\geq \delta^{\eps/2}|G_1|,\ \ \textrm{and}\ \ |N_\delta(\ell_0)\cap G_2''|\geq \delta^{\eps/2}|G_2|.
\] 
Let $\ell=\phi_G^{-1}(\ell_0) = \phi^{-1}\circ\psi_{G}^{-1}$, then $\ell^\perp\colon=\phi_F^{-1}(\ell_0^\perp)$ is perpendicular to $\ell$. Then $|N_\delta(\ell)\cap F|\geq\delta^{100\eta}|N_\delta(\ell)\cap F''|\geq\delta^{1+\eps/2+100\eta}\geq\delta^{1+\eps}$ and similarly for $|N_\delta(\ell)\cap G_1|$ and $|N_\delta(\ell^\perp)\cap G_2|$; thus Conclusion \ref{SWThmItemAPrimed} from Theorem \ref{SWThm}$'$ holds.

Next, suppose Conclusion \ref{assertion1ItemB} of Proposition \ref{wellSeparatedAndUniformProp} is true, i.e.~there exists $\rho\geq\delta$ and an interval $I$ of length at least $\delta^{-100\eta}\rho$ so that \eqref{bigRCoveringNumberPrime} holds for $F'',G_1'',G_2''$ and $\mathcal{H}''$ with $\eps/2$ in place of $\eps$. Unwinding definitions, we can verify that this implies Conclusion \ref{SWThmItemBPrimed} from Theorem \ref{SWThm}$'$.
\end{proof}


%

\end{document}